\documentclass[11pt,a4paper]{amsart}
\usepackage{setspace}
\usepackage{xcolor}
\usepackage[utf8]{inputenc}
\usepackage{tikz}
\usetikzlibrary{cd}
\usepackage{amsmath}
\usepackage{soul}
\usepackage{amssymb}
\usepackage{framed}
\usepackage{amsthm}
\usepackage{mathrsfs}
\usepackage{stmaryrd}
\usepackage{yhmath}
\usepackage{mathtools}
\usepackage[margin=1.25in]{geometry}
\usepackage{amsfonts}
\usepackage[font=small,format=plain,labelfont=bf,up,textfont=it,up]{caption}
\usepackage{amscd}
\usepackage[all,cmtip]{xy}\usepackage{pinlabel}
\usepackage{stmaryrd}
\usepackage{type1cm}
\usepackage{calc}
\usepackage{graphicx}
\usepackage{subcaption}
\usepackage{enumerate}
\usepackage{geometry}
\usepackage[many]{tcolorbox}
\tcbuselibrary{breakable}
\usepackage[colorlinks=true, linkcolor=blue, citecolor=cyan, urlcolor=cyan]{hyperref}
\usepackage{import}
\usepackage{xifthen}
\usepackage{pdfpages}
\usepackage{transparent}

\newcommand{\R}{\mathbb{R}}
\newcommand{\C}{\mathbb{C}}
\newcommand{\im}{\operatorname{im}}

\newcommand{\Q}{\mathbb{Q}}
\newcommand{\Z}{\mathbb{Z}}

\newcommand{\inv}{^{-1}}
\newcommand{\tautwo}{\tau^{(2)}}

\newcommand{\dimN}[1]{\operatorname{dim}_{\mathcal N#1}}

\newcommand{\whdet}{\operatorname{det}_{\rm w}}
\newcommand{\redet}{\operatorname{det}_{\rm r}}
\newcommand{\poly}{\mathcal P_\Z}
\newcommand{\whpoly}{\mathcal P_\Z^{\rm Wh}}
\newcommand{\mcalD}[1]{\mathcal{D}_{#1}}

\newcommand{\rank}{\operatorname{rank}}

\newcommand{\bb}{\backslash\!\backslash}

\newcommand{\ra}{\rightarrow}

\newtheorem{thm}{Theorem}[section]
\newtheorem{theorem}[thm]{Theorem}

\newtheorem{proposition}[thm]{Proposition}

\newtheorem*{claim*}{Claim}
\newtheorem{lemma}[thm]{Lemma}
\newtheorem{conjecture}[thm]{Conjecture}
\newtheorem{question}[thm]{Question}

\theoremstyle{definition}
\newtheorem{definition}[thm]{Definition}

\newtheorem{example}[thm]{Example}

\newtheorem{remark}[thm]{Remark}

\title{Universal $L^2$-torsion and sutured decomposition for 3-manifolds}
\author{Jianru Duan}
\address{Beijing International Center for Mathematical Research, Peking University, No. 5 Yiheyuan Road,
Haidian District, Beijing 100871, China}
\email{duanjr@stu.pku.edu.cn}

\begin{document}
\bibliographystyle{amsalpha}

\begin{abstract}
     
     Given an admissible 3-manifold $M$ and a cohomology class $\phi\in H^1(M;\R)$, we prove that the universal $L^2$-torsion of $M$ detects the fiberedness of $\phi$, except when $M$ is a closed graph manifold that admits no non-positively curved metric. We further extend this invariant to sutured 3-manifolds and derive a decomposition formula for taut sutured decompositions. Moreover, we show that a taut sutured manifold is a product if and only if its universal $L^2$-torsion is trivial. Our methods are based on a detailed study of the leading term map over Linnell’s skew field. As an application, we apply the theory to homomorphisms between finitely generated free groups, which enables explicit computations of the invariant for sutured handlebodies.
\end{abstract}
\maketitle

\section{Introduction}

Let $M$ be a compact orientable 3-manifold. A homomorphism $\phi\colon \pi_1(M)\ra\Z$ is called \textit{fibered} if it is induced by a fibration of $M$ over $S^1$. In his seminal work, Thurston \cite{thurston1986norm} introduced a semi-norm on $H^1(M;\R)$, now known as the \textit{Thurston norm}, and showed that its unit ball $B_x(M)$ is a finite-sided polyhedron. Moreover, he showed that the fibered classes in $H^1(M;\R)$ correspond exactly to certain open cones over the top-dimensional faces of $B_x(M)$. Marked by the success of Gabai's sutured manifold theory \cite{Gabai1983Foliations,gabai1987foliationsII} and the confirmation of the Virtual Fibering Conjecture \cite{agol2013virtual}, it has become a central theme in three-dimensional topology to determine the Thurston norm and the fibered structure of a 3-manifold. In this paper we focus on the class of admissible 3-manifolds, defined as follows: 
\begin{definition}[Admissible 3-manifold]
    A 3-manifold is called \emph{admissible} if it is compact, connected, orientable, and irreducible, its boundary is either empty or a collection of tori, and its fundamental group is infinite.  
\end{definition}

The universal $L^2$-torsion introduced by Friedl and L\"uck \cite{friedl2017universal} is defined for a finite CW-complex $X$ with vanishing $L^2$-Betti numbers. In the case when $X$ is an admissible 3-manifold, this invariant $\tautwo_u(X)$ takes values in the weak Whitehead group $\operatorname{Wh}^w(\pi_1(X))$ associated to the fundamental group of $X$. It has been shown that the universal $L^2$-torsion completely determines the Thurston norm of such manifolds \cite{friedl2017universal}. This leads naturally to the following questions: 
\begin{itemize}
    \item Does the universal $L^2$-torsion also characterize the fibered structure of an admissible 3-manifold? If so, in what manner?
    \item Can this invariant be extended to fit into Gabai's sutured manifold theory? If so, how does it change under sutured manifold decompositions? 
\end{itemize}

This paper is devoted to answering these two questions.

To investigate the first question, we are motivated by the role of the leading coefficient of the Alexander polynomial in detecting fiberedness. For any cohomology class $\phi\in H^1(M;\R)$ of an admissible 3-manifold $M$, we define a natural ``leading term map" 
\[
    L_\phi\colon \operatorname{Wh}^w(\pi_1(M))\ra \operatorname{Wh}^w(\pi_1(M))
\]
on the weak Whitehead group of $\pi_1(M)$. For an admissible 3-manifold $M$ that is not a closed graph manifold without a non-positively curved (NPC) metric, we show that a class $\phi\in H^1(M;\R)\setminus\{0\}$ is fibered if and only if the universal $L^2$-torsion of $M$ has trivial $\phi$-leading term. 
\begin{theorem}\label{Main Theorem fibered iff face map zero}
    Suppose $M$ is an admissible 3-manifold that is not a closed graph manifold without an NPC metric. For any nonzero cohomology class $\phi\in H^1(M;\R)$, the class $\phi$ is fibered if and only if $L_\phi\tautwo_u(M)=1\in \operatorname{Wh}^w(\pi_1(M))$.
\end{theorem}
As a consequence, Theorem \ref{Main Theorem fibered iff face map zero} provides a clear description of the \textit{marking} on the $L^2$-torsion polytope $\mathcal P(M)$ introduced in \cite{Kielak2020BNSInvariants} which determines the fibered structure of $M$.
 
The solution to Theorem \ref{Main Theorem fibered iff face map zero} in fact relies on the study of the other main question, namely, extending the theory of universal $L^2$-torsion to sutured 3-manifolds, an object invented by Gabai \cite{Gabai1983Foliations,gabai1987foliationsII} to describe 3-manifolds via cut-and-paste constructions. A \textit{sutured manifold} $(N,R_+,R_-,\gamma)$ is a compact oriented 3-manifold whose boundary is partitioned into two oriented subsurfaces $R_+$ and $R_-$, meeting along their common boundary $\gamma$. A key result of Herrmann \cite{herrmann2023sutured} shows that a sutured manifold $(N,R_+,R_-,\gamma)$ is \textit{taut} if and only if the pair $(N,R_+)$ has trivial $L^2$-Betti numbers (see Theorem \ref{Theorem Taut iff L2acyclic}). This motivates the definition of the universal $L^2$-torsion $\tautwo_u(N, R_+)$ for a taut sutured manifold, which takes values in $\operatorname{Wh}^w(\pi_1(N))$.

A sutured manifold can be decomposed along a nicely embedded surface $S$ and the resulting manifold is again a sutured manifold, we write  $$(N,R_+,R_-,\gamma)\stackrel{S}\rightsquigarrow (N',R_+',R_-',\gamma')$$ for such a sutured decomposition. A sutured decomposition is called taut if the resulting sutured manifold is taut.

The second main result of this paper describes how the universal $L^2$-torsion behaves under taut sutured decompositions. Specifically, its change is described by the leading term map introduced earlier.

\begin{theorem}\label{Main Theorem taut decomposition and face map}
    Let $(N,R_+,R_-,\gamma)\stackrel{\Sigma}\leadsto (N',R_+',R_-',\gamma')$ be a taut sutured decomposition and let $\phi\in H^1(N;\Z)$ be the Poincar\'e dual of the decomposition surface $\Sigma$, then
    \[
        j_*\tautwo_u(N',R_+')=L_\phi\tautwo_u(N,R_+)
    \]
    where $j_*\colon \operatorname{Wh}^w(\pi_1(N'))\ra \operatorname{Wh}^w(\pi_1(N))$ is induced by the inclusion $j\colon N'\hookrightarrow N$.
\end{theorem}

Furthermore, we show that the universal $L^2$-torsion serves as the only obstruction to the product structure on sutured manifolds. Indeed, a sutured manifold $(N,R_+,R_-,\gamma)$ may be viewed as a cobordism between two compact surfaces $R_\pm$. When $N$ is taut, the inclusion map $R_+\hookrightarrow N$ induces isomorphism on $L^2$-homology, which can be considered as an ``$L^2$-homology cobordism". We prove that this cobordism is a trivial product if and only if its universal $L^2$-torsion is trivial.

\begin{theorem}\label{Main Theorem universal L2 torsion detect product sutured manifold}
    Let $(N,R_+,R_-,\gamma)$ be a taut sutured manifold such that both $R_+$ and $R_-$ are non-empty. Then $N$ is isomorphic to the product sutured manifold $R_+\times [0,1]$ if and only if $\tautwo_u(N,R_+)=1\in \operatorname{Wh}^w(\pi_1(N))$.
\end{theorem}

We further study the computation of the universal $L^2$-torsions. It is natural to define the universal $L^2$-torsion of a continuous map $f:X\ra Y$, generalizing that of a space pair $(Y,X)$ (see Section \ref{Section universal L2 torsion}). Motivated by many sutured 3-manifold examples $(N,R_+)$ where $N$ is a 3-dimensional handlebody and $R_+\subset \partial N$ is a connected subsurface with boundary, we are particularly interested in $f\colon X\ra Y$ where $X,Y$ are finite classifying spaces of finitely generated free groups. In this setting, the universal $L^2$-torsion of $f$ is completely determined by the induced homomorphism $\varphi\colon \pi_1(X)\ra \pi_1(Y)$ on fundamental groups, and we obtain the following explicit formula (see Theorem \ref{Theorem Properties of universal torsion for free group homomorphism}):

\begin{theorem}\label{Main Theorem formula for universal L2 torsion of free groups}
    Let $f\colon X\ra Y$ be a continuous map where $X$ and $Y$ are finite connected classifying spaces of free groups. Let $\varphi\colon \pi_1(X) \ra \pi_1(Y)$ be the induced homomorphism on fundamental groups. Then: 
    \[
        \tautwo_u(f)=[J_\varphi]\in \operatorname{Wh}^w(\pi_1(Y))\sqcup \{0\}
    \]
    where $J_\varphi$ is the Fox Jacobian matrix of $\varphi$ over $\Z[\pi_1(Y)]$, whose entries are the Fox derivatives
    \[
        (J_\varphi)_{ij}=\frac{\partial \varphi(x_i)}{\partial y_j}\in \Z [\pi_1(Y)],\quad 1\leqslant i\leqslant n,\ 1\leqslant j\leqslant m
    \]
     with respect to any chosen bases $\pi_1(X)=\langle x_1,\ldots,x_n\rangle$, $\pi_1(Y)=\langle y_1,\ldots,y_m\rangle$. 
\end{theorem}
We apply this formula to  explicitly compute the universal $L^2$-torsion of certain sutured handlebodies. A key example is the sutured manifold $S^3\bb \Sigma$ obtained by decomposing the $n$-chain link complement along a minimal-genus Seifert surface $\Sigma$ (see Figure \ref{fig:4-chain link in Intro} and Example \ref{Example n-chain link}). We show that
\[
    \tautwo_u(S^3\bb \Sigma,\Sigma_+)=[1+y_1+\cdots+y_{n-1}]\in \operatorname{Wh}^w(\pi_1(S^3\bb \Sigma))
\]
where $\{y_1,\ldots,y_{n-1}\}$ is a free generating set of $\pi_1(S^3\bb \Sigma)$. Combining this with the calculations of \cite{ben2022fuglede}, we obtain the following formula for the classical $L^2$-torsion: 
\[
    \tautwo(S^3\bb \Sigma,\Sigma_+)=\frac{(n-1)^{\frac{n-1}{2}}}{n^{\frac{n-2}{2}}}\sim \sqrt{n/e},\quad \text{as $n\ra +\infty$}.
\]
Here, the left-hand side denotes the classical $L^2$-torsion of the pair $(S^3\bb \Sigma,\Sigma_+)$, obtained from the universal $L^2$-torsion via the Fuglede--Kadison determinant. Together with results from \cite{duan2025Guts}, this computation yields an infinite family of hyperbolic manifolds for which the leading coefficient of the $L^2$-Alexander torsion associated to a nonzero class is greater than 1. This confirms the final remaining case of \cite[Conjecture 1.7]{ben2022leading}.
\begin{figure}[htbp]
    \centering
    
\def\svgwidth{.4\columnwidth}
\begingroup%
  \makeatletter%
  \providecommand\color[2][]{%
    \errmessage{(Inkscape) Color is used for the text in Inkscape, but the package 'color.sty' is not loaded}%
    \renewcommand\color[2][]{}%
  }%
  \providecommand\transparent[1]{%
    \errmessage{(Inkscape) Transparency is used (non-zero) for the text in Inkscape, but the package 'transparent.sty' is not loaded}%
    \renewcommand\transparent[1]{}%
  }%
  \providecommand\rotatebox[2]{#2}%
  \newcommand*\fsize{\dimexpr\f@size pt\relax}%
  \newcommand*\lineheight[1]{\fontsize{\fsize}{#1\fsize}\selectfont}%
  \ifx\svgwidth\undefined%
    \setlength{\unitlength}{67.6412273bp}%
    \ifx\svgscale\undefined%
      \relax%
    \else%
      \setlength{\unitlength}{\unitlength * \real{\svgscale}}%
    \fi%
  \else%
    \setlength{\unitlength}{\svgwidth}%
  \fi%
  \global\let\svgwidth\undefined%
  \global\let\svgscale\undefined%
  \makeatother%
  \begin{picture}(1,0.75200198)%
    \lineheight{1}%
    \setlength\tabcolsep{0pt}%
    \put(0,0){\includegraphics[width=\unitlength,page=1]{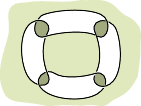}}%
    \put(0.08794095,0.12461284){\color[rgb]{0,0,0}\makebox(0,0)[lt]{\lineheight{1.25}\smash{\begin{tabular}[t]{l}$\Sigma$\end{tabular}}}}%
  \end{picture}%
\endgroup%

    \caption{The $n$-chain link and the minimal-genus Seifert surface $\Sigma$, where $n=4$.}
    \label{fig:4-chain link in Intro}
\end{figure}

\subsection{Motivations}
\subsubsection{Why universal $L^2$-torsion} 
The torsion of an exact chain complex is an invariant generalizing the notion of ``determinant" in linear algebra.
To apply this invariant to a finite CW-complex $X$, one wish to extract an exact complex from the cellular chain complex of the universal cover $\widehat X$.  One way to do this is to extend the scalars from the group ring $\Z [\pi_1(X)]$ to a \emph{commutative} field. For 3-manifolds, this yields the Reidemeister--Franz torsion and the (multi-variable) Alexander polynomials \cite{milnor1962duality,McMullen2002Alexander,friedl2011survey}, via scalar extensions to $\C$ and to the field of rational functions $\Q(H_1(X)/{\rm Tor})$, respectively. However, this extension of scalars loses non-commutative information of the fundamental group.

Whitehead torsion, introduced by J.H.C Whitehead, avoids this issue by retaining the full group ring structure. It takes values in the Whitehead group $\operatorname{Wh}(\pi_1(X))$ consisting of invertible matrices over $\Z[\pi_1(X)]$ modulo elementary relations \cite{milnor1966whitehead,cohen1973course}. However, the Whitehead torsion is only defined for pairs $(X,Y)$ where $X$ deformation retracts to $Y$. Also, it is too restrictive for a matrix to be invertible over $\Z [\pi_1(X)]$. Indeed, it is a well-known conjecture that the Whitehead group of any torsion-free group should vanish. 

The universal $L^2$-torsion (see Section \ref{Section universal L2 torsion} for definitions) emerges as a powerful and widely applicable torsion invariant. 
Here are some of its advantages:
\begin{enumerate}[1.]
    \item The universal $L^2$-torsion is defined via the faithful regular representation of the fundamental group into its $\ell^2$-space and captures its deep noncommutative information.
    \item It applies to a broad class of spaces with vanishing $L^2$-Betti numbers, including mapping tori \cite{luck1994l2}, spaces with infinite amenable fundamental groups \cite{cheeger1986l2}, admissible 3-manifolds \cite{lott19952}, and all odd-dimensional closed hyperbolic manifolds \cite{Hess1998} (see \cite{lueck2002l2} for further examples).
    \item The weak Whitehead group $\operatorname{Wh}^w(G)$ 
    is often highly nontrivial. It supports interesting homomorphisms such as the Fuglede--Kadison determinant and the polytope maps \cite{friedl2017universal}.
\end{enumerate}

Recently, the universal $L^2$-torsion has been used to define group-theoretic Thurston norms for wider classes of groups \cite{friedl2017thurston,funke2018alexander,henneke2020agrarian,kielak2024agrarian,kudlinska2024thurston}. The BNS-invariant, which generalizes the notion of fiberedness from 3-manifold groups to arbitrary finitely generated groups, is also deeply connected to the universal $L^2$-torsion, as evidenced by recent work on two-generator one-relator groups \cite{friedl2017thurston,friedl2020two,henneke2020agrarian}, free-by-cyclic groups \cite{funke2018alexander} and more general agrarian groups  \cite{Kielak2020BNSInvariants}. These studies focus on the $L^2$-torsion polytope derived from the universal $L^2$-torsion and investigate the existence of a \emph{marking} on the polytope that determines the BNS-invariant.

A key novelty of our Theorem \ref{Main Theorem fibered iff face map zero} is to provide a clear and direct description of the fibered structure for 3-manifolds in terms of the universal $L^2$-torsion itself. We expect that our technique can be modified to study more groups as mentioned above.

\subsubsection{Fiberedness and the leading term of torsions}
It is well-known that if a class $\phi$ is fibered, then its Alexander polynomial and twisted Alexander polynomial have trivial leading coefficients \cite{McMullen2002Alexander,goda2005reidemeister,FriedlKim2006Thurston}. While the converse does not hold in general, Friedl and Vidussi \cite{friedl2011twisted} showed that the collection of all twisted Alexander polynomials associated to $\phi$ does determine whether $\phi$ is fibered. This suggests that the leading term of torsion invariants in fact carry enough information for the detection of fiberedness. 

Our main Theorem \ref{Main Theorem fibered iff face map zero} is largely motivated by the study of $L^2$-Alexander torsions. For a cohomology class $\phi\in H^1(M;\R)$, the $L^2$-Alexander torsion of $\phi$ was introduced by Dubois, Friedl and L\"uck \cite{Dubois2016L2AlexanderTorsion} as an $L^2$-analogue of the classical Alexander polynomial. Liu \cite{liu2017degree} established the existence of the leading coefficient $C(M,\phi)\in[1,+\infty)$ of the $L^2$-Alexander torsion, and it is known that $C(M,\phi)=1$ whenever $\phi$ is a fibered class. Conversely, it is not difficult to find examples that $C(M,\phi)=1$ while $\phi$ is non-fibered. For example, take $M$ to be any two-bridge knot complement and $\phi$ the canonical class; the leading coefficient $C(M,\phi)$ can be interpreted as the classical $L^2$-torsion of the guts $\Gamma(\phi)$ \cite{duan2025Guts}, which is a disjoint union of solid tori relative to an annulus on its boundary \cite{agol2022guts}. It is not hard to see (using Theorem \ref{Main Theorem formula for universal L2 torsion of free groups}) that each component has universal $L^2$-torsion of the form $$\Big[\frac{t^n-1}{t-1}\Big]\in\operatorname{Wh}^w(\Z)\cong \Q(t^{\pm})^\times/\{\pm t^n\mid n\in \Z\},$$ while their classical $L^2$-torsions are trivial since they are mapped to $1$ by the Fuglede--Kadison determinant
\[
    \operatorname{det}_{\mathcal N\pi_1(M)}\colon \operatorname{Wh}^w(\Z)\ra \R_+.
\] This observation suggests the need for a more refined invariant that is defined algebraically and does not rely on the Fuglede--Kadison determinant.

Accordingly, we consider the universal $L^2$-torsion and introduce the leading term map $$L_\phi\colon \operatorname{Wh}^w(\pi_1(M))\ra \operatorname{Wh}^w(\pi_1(M))$$ which serves as an algebraic analogue of “taking leading coefficient” within the weak Whitehead group. We refer to Section \ref{Section Leading term map restriction map and polytope map} for further details.

\subsubsection{Analogy with sutured Floer homology}
The Heegaard Floer homology, defined by Ozsváth and Szabó \cite{ozsvath2004holomorphic} for closed 3-manifolds, was generalized by Juhász \cite{juhasz2006holomorphic} to sutured Floer homology ($SFH$), an invariant for balanced sutured manifolds. In \cite{juhasz2008floer} the fundamental properties of $SFH$ are proved:
\begin{enumerate}[\rm \quad (1)]
    \item If $(M,\gamma)$ is a taut balanced sutured manifold, then $SFH(M,\gamma)$ is non-trivial. 
    \item If $(M,\gamma)\rightsquigarrow (M',\gamma')$ is a sutured decomposition of balanced sutured manifolds, then $SFH(M',\gamma')$ is a direct summand of $SFH(M,\gamma)$.
    \item A taut balanced sutured manifold $(M,\gamma)$ is a product sutured manifold if and only if $SFH(M,\gamma)=\Z$.
\end{enumerate}
We observe that analogous properties hold in the theory of universal $L^2$-torsions. Specifically, property (1) corresponds to Herrmann's result (Theorem \ref{Theorem Taut iff L2acyclic}), while (2) and (3) are reflected in Theorems \ref{Main Theorem taut decomposition and face map}--\ref{Main Theorem universal L2 torsion detect product sutured manifold}, respectively. \begin{question}
Is there a deeper connection between sutured Floer homology and the universal $L^2$-torsion of a taut sutured manifold?
\end{question}
This question is particularly interesting because $SFH$ is defined in terms of holomorphic curves and is not directly related to the fundamental group, whereas the universal $L^2$-torsion is algebraic and computable from a presentation of the fundamental group.

Furthermore, Juhász \cite{juhasz2010sutured} showed that the rank of $SFH$ serves as a complexity measure for sutured manifolds that decreases under non-trivial decompositions. This motivates a parallel question in the $L^2$-context:

\begin{question}
Can one define a natural complexity function on the weak Whitehead group that decreases under the leading term map $L_\phi$?
\end{question}

\subsection{Proof ingredients} The remaining part of the paper is divided into five sections. We briefly discuss the contents of each.
\subsubsection{Algebraic preliminaries}
Section \ref{Section Algebraic preliminaries} reviews the basic notions of $L^2$-theory, focusing on three topics that play a central role in this paper: the \emph{Hilbert modules}, the \emph{Atiyah Conjecture} and \emph{Linnell's skew field}, and \emph{$K_1$-groups} together with the \emph{Dieudonn\'e determinants}.

\subsubsection{Universal $L^2$-torsion}
In Section \ref{Section universal L2 torsion}, we define the universal $L^2$-torsion and discuss its computation. Let $X$ be a finite CW-complex whose fundamental group $G$ is torsion-free and satisfies the Atiyah Conjecture, the universal $L^2$-torsion $\tautwo_u(X)$ is defined as the torsion of the chain complex $\mcalD G\otimes _{\Z G} C_*(\widehat X)$ viewed as a complex of modules over Linnell's skew field $\mcalD G$, and takes values in the Whitehead group $\operatorname{Wh}(\mcalD G)$. Slightly different from the original definition in \cite{friedl2017universal}, where the universal $L^2$-torsion lives in the weak Whitehead group $\operatorname{Wh}^w(G)$, our definition is more restrictive but is better suited to algebraic manipulation. Importantly, the two definitions coincide when $X$ is a 3-manifold, or more generally when $G$ belongs to Linnell’s class $\mathcal C$; see Section \ref{Remark of definition of universal L2 torsion} for a detailed discussion.

We further extend the definition of universal $L^2$-torsion to CW-pairs and to continuous mappings between finite CW-complexes via mapping cylinders. Several basic properties are established. This generalization brings greater flexibility to the application of the invariant.

Finally, we generalize Turaev's \emph{matrix chain method} \cite{Turaev2001IntroductionToComb} for the direct computation of the universal $L^2$-torsion of a chain complex. Informally, a chain complex is acyclic if and only if we can extract a chain of invertible submatrices from the matrices representing the boundary operators. The torsion is then given by the alternating product of their determinants.

\subsubsection{Leading term map, restriction map and polytope map}
The content of Section \ref{Section Leading term map restriction map and polytope map} is mostly technical, focusing on the study of three important homomorphisms related to Linnell's skew field $\mcalD G$.

Given any real character $\phi\colon G\ra \R$ and a nonzero element $a=\sum_{g\in G}n_g\cdot g\in\Z G$, the $\phi$-leading term of $a$ is defined to be the sum of nonzero terms $n_g\cdot g$ for which $\phi(g)$ is minimal. Using the crossed product structure of the skew field $\mcalD G$, this extends to the \emph{leading term map} $$L_\phi\colon \mcalD G^\times\ra \mcalD G^\times.$$
This is a homomorphism that induces homomorphisms on the $K_1$-group and the Whitehead group of $\mcalD G$. A key result (Theorem \ref{Theorem Leading term of matrices}) concerning the leading term map says that, roughly speaking, if $A,B$ are square matrices over $\mcalD G$ such that the $\phi$-degree of entries of $B$ are strictly greater than those of $A$, then the $\phi$-leading term of the determinant $\det(A+B)$ is independent of $B$. This intuitive result is surprisingly useful throughout our paper.

The \emph{restriction map} relates the weak $K_1$-group of $G$ to that of its finite-index normal subgroup $L<G$:
\[
    \operatorname{res}^G_L\colon K_1(\mcalD G)\ra K_1(\mcalD L).
\]
This is defined by interpreting an element $z\in K_1(\mcalD G)$ as an operator $r_z\colon \ell^2(L)^{[G\colon L]}\ra \ell^2(L)^{[G\colon L]}$ and then taking the Dieudonn\'e determinant. We prove that the leading term map commutes with the restriction map (see Theorem \ref{Theorem leading term map commutes with restrictions}).

Every element $z\in \mcalD G^\times$ may be viewed as being ``supported" on a certain Newton polytope in the real vector space $H_1(G;\R)$. The \emph{polytope map}
\[
    \mathbb P\colon \mcalD G^\times \ra \poly(H)
\]
formalizes this intuition, where $H$ is the free abelianization of $G$ and 
$\poly(H)$ is the Grothendieck group of the integral polytopes in $\R\otimes_\Z H$ under the Minkowski sum. It is known that $\poly(H)$ is a free abelian group possessing an explicit basis \cite{funke2021integral}. The polytope map is particularly useful for detecting nontrivial elements in $\mcalD G^\times/[\mcalD G^\times,\mcalD G^\times]=K_1(\mcalD G)$.

At the end of Section \ref{Section leading term map} we prove the ``if" direction of Theorem \ref{Main Theorem fibered iff face map zero}, namely Theorem \ref{Theorem if part}. The idea is as follows: if a class $\phi\in H^1(M;\R)\setminus\{0\}$ satisfies $L_\phi\tautwo_u(M)=1$, then this condition is preserved under any finite regular covering of $M$ since the restriction map commutes with the leading term map. By applying the polytope map, it follows that $\phi$ is lifted to a top-dimensional Thurston cone in any such coverings. When $M$ is not a closed graph manifold without an NPC metric, Agol's Virtual Fibering Theorem (see \cite{Aschenbrenner20153ManifoldGroups} and references therein) asserts that $\phi$ can be lifted to a \textit{quasi-fibered} class in some regular finite coverings. A quasi-fibered class in a top-dimensional Thurston cone must be a fibered class, hence $\phi$ can be lifted to a fibered class and is itself fibered.

\subsubsection{Universal $L^2$-torsion for taut sutured manifolds}
In Section \ref{Section taut sutured manifolds}, we investigate the universal $L^2$-torsion for \textit{sutured 3-manifolds}. In order to prove Theorem \ref{Main Theorem taut decomposition and face map}, namely the decomposition formula for universal $L^2$-torsion under taut sutured decompositions, we first adopt the idea of \textit{Turaev's algorithm} \cite{Turaev2002Homological,ben2022leading} to reduce to the case of non-separating decomposition surfaces. Then we choose an explicit CW-structure under which the chain complex of the decomposed manifold is exactly the ``leading term" of the chain complex of the original manifold. Taking their universal $L^2$-torsions results in the required formula. This idea is rigorously formulated in Theorem \ref{Theorem main theorem decomposition formula}.

Once the decomposition formula is established, the ``only if" direction of Theorem \ref{Main Theorem fibered iff face map zero} follows easily, see Theorem \ref{Theorem only if part}.

A refined doubling trick (Lemma \ref{Lemma double is not closed graph}) enables us to convert a taut sutured manifold into an admissible 3-manifold which is not a closed graph manifold. Combining this trick with the fiberedness criterion for admissible 3-manifolds (Theorem \ref{Main Theorem fibered iff face map zero}) and the decomposition formula (Theorem \ref{Main Theorem taut decomposition and face map}), we show that the universal $L^2$-torsion detects product sutured manifolds (Theorem \ref{Main Theorem universal L2 torsion detect product sutured manifold}).

\subsubsection{Applications of the universal $L^2$-torsion}
Section \ref{Section Applications} is devoted to group-theoretic applications and some concrete computations of the universal $L^2$-torsion. We begin by defining this invariant for homomorphisms between finitely generated free groups. Conjecturally this is the only obstruction for a homomorphism to be an isomorphism (Conjecture \ref{Conjecture free goup isomorphism}).

An explicit computational formula is established in Theorem \ref{Theorem Properties of universal torsion for free group homomorphism}, which in particular proves Theorem \ref{Main Theorem formula for universal L2 torsion of free groups}. 
We apply this formula to sutured manifolds modeled on 3-dimensional handlebodies, leading to criteria for detecting whether such a sutured manifold is taut or a product (Proposition \ref{Proposition universal L2torsion of handlebody}). Explicit computations are carried out for the family of sutured manifolds obtained from cutting up the $n$-chain link complement along a Seifert surface (Example \ref{Example n-chain link}).

\subsection{Acknowledgments} The author is deeply grateful to his advisor Yi Liu for his guidance and encouragement. He would also like to thank Dawid Kielak, Marco Linton, Qiuyu Ren, Bin Sun and Xiaolei Wu for stimulating conversations. Special thanks go to Ian Agol for his hospitality during the author's visit to Berkeley, where part of this work was carried out.

\section{Algebraic preliminaries}\label{Section Algebraic preliminaries}
Throughout this paper we adopt the following conventions: all groups are discrete, rings are unital but not necessarily commutative, and modules are left-modules. Fields with noncommutative multiplication are termed skew fields.

\subsection{Hilbert modules}
Let $G$ be a group. Consider the following Hilbert space 
\[
    \ell^2(G)=\Big\{\sum_{g\in G}c_g\cdot g\ \Big|\ c_g\in\C,\  \sum_{g\in G}|c_g|^2<\infty\Big\}
\]
with inner product
\[
    \Big\langle\sum_{g\in G}c_g\cdot g,\sum_{g\in G}d_g\cdot g\Big\rangle = \sum_{g\in G} c_g\overline{d_g}.
\]
This Hilbert space admits natural left and right isometric $G$-actions by multiplication. The \emph{group von Neumann algebra} $\mathcal NG$ is defined as the $\C$-algebra of all bounded linear operators of $\ell^2(G)$ that commutes with the left $G$-action.
Every $G$-invariant closed subspace $V$ of $\ell^2(G)^n$ is called a \emph{Hilbert $\mathcal NG$-module} and can be assigned the \emph{von-Neumann dimension} $\dimN{G}V$ which is a real number in $[0,n]$.

Let $\mathcal UG$ be the set of all densely-defined, closed operators (possibly unbounded) on $\ell^2(G)$ that commute with the left $G$-action. The composition and addition of two operators in $\mathcal UG$ are well-defined \cite[Section 8.1]{lueck2002l2}, making $\mathcal UG$ a $\C$-algebra and is called the \emph{algebra of operators affiliated to $\mathcal NG$}. In particular, there are natural inclusions
\[
    \Z G\subset \mathcal NG\subset  \mathcal UG
\]
where the integral group ring $\Z G$ embeds into $\mathcal NG$ by the right regular representation on $\ell^2(G)$. Moreover, any $m\times n$ matrix $A$ over $\Z G$ can be viewed as a $G$-invariant bounded operator $r_A\colon \ell^2(G)^m\ra \ell^2(G)^n$ defined by right multiplication.

\begin{definition}[Weak isomorphism]\label{Definition of weak isomorphisms}
    A $G$-invariant bounded operator $f\colon \ell^2(G)^m\ra \ell^2(G)^n$ is called a \emph{weak isomorphism} if $f$ is injective with dense image. Note that if $f$ is a weak isomorphism then $m$ and $n$ must be equal. A square matrix $A$ over $\Z G$ is called a \emph{weak isomorphism} if the $G$-invariant bounded operator $r_A\colon \ell^2(G)^m\ra \ell^2(G)^n$ is a weak isomorphism.
\end{definition}

\subsection{Atiyah Conjecture and Linnell's skew field}
\begin{definition}[Atiyah Conjecture]
    A torsion-free group $G$ is said to satisfy the \emph{Atiyah Conjecture} if for any matrix $A$ over $\Z G$ the von Neumann dimension of $\ker( r_A)$ is an integer.
\end{definition}
The Atiyah Conjecture has been verified for a large class of groups (c.f. \cite[Theorem 4.2]{Kielak2020BNSInvariants}). We mark the following important class of groups given by Linnell, which is large enough to include all 3-manifold groups.
\begin{theorem}[\cite{linnell1993division}]\label{Theorem Linnell's class C}
    Let $\mathcal C$ be the smallest class of groups which contains all free groups and is closed under directed unions and extensions by elementary amenable groups. Then any torsion-free group in $\mathcal C$ satisfies the Atiyah Conjecture.
\end{theorem}
\begin{theorem}[{\cite[Theorem 1.4]{kielak2024group}}]\label{Theorem 3-manifold group in class C}
    The fundamental group of any connected 3-manifold lies in $\mathcal C$.
\end{theorem}

An alternative algebraic characterization of the Atiyah Conjecture is given by the division closure of $\Z G$ in $\mathcal UG$.

\begin{definition}[Division closure]
    Let $R$ be a subring of a ring $S$. The \emph{division closure} of $R$ in $S$ is the smallest subring of $S$ containing $R$ that is closed under taking inverses in $S$; that is, any element of this subring which is invertible in $S$ has its inverse also in the subring. For a group $G$, we denote by $\mcalD G$ the division closure of $\Z G$ in $\mathcal UG$.
\end{definition}
\begin{theorem}[\cite{linnell1993division}]
    A torsion-free group $G$ satisfies the Atiyah Conjecture if and only if $\mcalD G$ is a skew field.
\end{theorem}

Therefore, for a torsion-free group $G$ satisfying the Atiyah Conjecture, the ring $\mcalD G$ is a skew field called the \emph{Linnell's skew field}. The following Proposition \ref{Proposition subgroup satisfies Atiyah conjecture} summarizes useful functorial properties of this construction.

\begin{proposition}[{\cite[Proposition 4.6]{Kielak2020BNSInvariants}}]\label{Proposition subgroup satisfies Atiyah conjecture}
    Let $G$ be a torsion-free group satisfying the Atiyah Conjecture. Then the following hold.
    \begin{enumerate}[\quad\rm (1)]
        \item Every automorphism of the group $G$ extends to an automorphism of $\mcalD G$.
        \item If $K$ is a subgroup of $G$, then $K$ also satisfies the Atiyah Conjecture. Moreover, the natural embedding $\Z K\hookrightarrow \Z G$ extends to an embedding $\mcalD K\hookrightarrow \mcalD G$.
    \end{enumerate}
\end{proposition}

\subsection{The $K_1$-group and Dieudonn\'e determinant}
Let $\mathcal F$ be a skew field. For any positive integer $n$ let $\operatorname{GL}(n,\mathcal F)$ be the group of invertible $(n\times n)$-matrices over $\mathcal F$. By identifying each $M\in \operatorname{GL}(n,\mathcal F)$ with the block matrix
\[
    \begin{pmatrix}
        M & 0\\ 0 & 1
    \end{pmatrix}\in \operatorname{GL}(n+1,\mathcal F),
\]
we obtain a natural chain of inclusions
\[
    \operatorname{GL}(1,\mathcal F)\subset \operatorname{GL}(2,\mathcal F)\subset\cdots.
\]
The direct union $\operatorname{GL}(\mathcal F)=\bigcup_{n\geqslant1} \operatorname{GL}(n,R)$ is called the \emph{infinite general linear group} over $\mathcal F$. A classical result of Whitehead (see e.g. \cite{milnor1966whitehead}) states that the commutator subgroup $[\operatorname{GL}(\mathcal F),\operatorname{GL}(\mathcal F)]$ coincides with the subgroup generated by all elementary matrices in $\operatorname{GL}(\mathcal F)$.

\begin{definition}[\cite{dieudonne1943determinants,rosenberg1995algebraic}]
    The \emph{Dieudonn\'e determinant} is the unique map $$\det\colon \operatorname{GL}(\mathcal F)\ra \mathcal F^\times/[\mathcal F^\times,\mathcal F^\times]$$ satisfying the following properties (a)--(c):
\begin{enumerate}[\quad (a)]
    \item The determinant is invariant under left elementary row operations. That is, if $A\in \operatorname{GL}(\mathcal F)$ and $A'$ is obtained from $A$ by adding a left multiple of a row to another row, then $\det A=\det A'$.
    \item If $A\in \operatorname{GL}(\mathcal F)$, and $A'$ is obtained from $A$ by left multiplying one of the rows by $a\in\mathcal F^\times$, then $\det A=[a]\cdot\det A'$ where $[a]$ is the image of $a$ in $\mathcal F^\times/[\mathcal F^\times,\mathcal F^\times]$.
    \item The determinant of the identity matrix is $[1]$.
\end{enumerate}
The determinant also has the following additional properties (d)--(e).
\begin{enumerate}[\quad (a)]
\addtocounter{enumi}{3}
    \item If $A,B\in \operatorname{GL}(n,\mathcal F)$, then $\det(AB)=\det A\cdot \det B$.
    \item If $A\in \operatorname{GL}(\mathcal F)$ and $A'$ is obtained from $A$ by interchanging two of its rows, then $\det A'=[-1]\cdot\det A$.
    \item Suppose $\mathcal F$ is equipped with an involution. Then the determinant is compatible with the involution transpose, that is $\det(A^*)=(\det A)^*$.
\end{enumerate}
\end{definition}
\begin{remark}
    The Dieudonn\'e determinant is not invariant under the usual matrix transpose. For example, consider Linnell's skew field $\mcalD F$ of the free group $F=\langle x,y \rangle$. Then the matrix 
    \[
        A=\begin{pmatrix}
            x & 1\\
            xy & y
        \end{pmatrix}
    \]
    has Dieudonn\'e determinant $\det(A)=[yx-xy]$, while $A^T=\begin{pmatrix}
            x & xy\\
            1 & y
        \end{pmatrix}$ is not even invertible.
\end{remark}

The Dieudonn\'e determinant factors through the abelianization of $\operatorname{GL}(\mathcal F)$, and the induced homomorphism is in fact an isomorphism.
\begin{lemma}[{\cite[Corollary 2.2.6]{rosenberg1995algebraic}}]\label{Lemma K1 group equals abelianization}
    For any skew field $\mathcal F$, the Dieudonn\'e determinant induces a group isomorphism
\[
    \operatorname{GL}(\mathcal F)/[\operatorname{GL}(\mathcal F),\operatorname{GL}(\mathcal F)]\xrightarrow{\cong} \mathcal F^\times/[\mathcal F^\times,\mathcal F^\times].
\]
The inverse map sends the class of $a\in \mathcal F^\times$ to the class of the $(1\times 1)$-matrix $(a)$.
\end{lemma}

\begin{definition}[$K_1(\mathcal F)$ and $\widetilde K_1(\mathcal F)$]
     The \emph{$K_1$-group} of $\mathcal F$ is defined as $$K_1(\mathcal F):=\mathcal F^\times/[\mathcal F^\times,\mathcal F^\times].$$
The \emph{reduced $K_1$-group} $\widetilde K_1(\mathcal F)$ is the quotient of $K_1(\mathcal F)$ by the subgroup $\{[\pm1]\}$.
\end{definition}

We now specialize these definitions to  Linnell's skew field $\mcalD G$, where $G$ is a torsion-free group satisfying the Atiyah Conjecture. In this context, we introduce the following Whitehead group: 
\begin{definition}[$\operatorname{Wh}(\mcalD G)$]
    The \emph{Whitehead group} of $\mcalD G$ is defined as the quotient of $K_1(\mcalD G)$ by the subgroup generated by the classes $[\pm g]$ for $g\in G$: 
    \[
        \operatorname{Wh}(\mcalD G):=K_1(\mcalD G)/\langle[\pm g]\mid g\in G\rangle.
    \]
\end{definition}

\begin{definition}\label{Definition det, redet, whdet}
    In the following we will frequently consider determinants and their images in the quotient groups $\widetilde K_1(\mcalD G)$ and $\operatorname{Wh}(\mcalD G)$. We therefore define the homomorphisms
    \[
        \redet\colon \operatorname{GL}(\mcalD G)\ra \widetilde K_1(\mcalD G),\quad \whdet\colon \operatorname{GL}(\mcalD G)\ra \operatorname{Wh}(\mcalD G)
    \]
    to be the compositions of the Dieudonn\'e determinant with the respective quotient maps, so that the following diagram commutes.
    \[\begin{tikzcd}&               & \operatorname{GL}(\mcalD G) \arrow[ld, "\det"'] \arrow[d, "\redet"] \arrow[rd, "\whdet"] &                   \\
\mcalD G^\times/[\mcalD G^\times,\mcalD G^\times] \arrow[r, equals, shorten <=3pt, shorten >=3pt] & K_1(\mcalD G) \arrow[r]   & \widetilde K_1(\mcalD G)   \arrow[r]                                    & \operatorname{Wh}(\mcalD G)  
\end{tikzcd}\]
\end{definition}

\section{Universal $L^2$-torsion}\label{Section universal L2 torsion}
Let $G$ be a torsion-free group satisfying the Atiyah conjecture and let $\mcalD G$ be the Linnell's skew field. The universal $L^2$-torsion defined in this section takes values in $\widetilde K_1(\mcalD G)$ for $\Z G$-chain complexes, and in $\operatorname{Wh}(\mcalD G)$ for finite CW-complexes with fundamental group $G$. Although our definition is more restrictive than the original one given in \cite{friedl2017universal}, the two definitions coincide when $G$ is a 3-manifold group.
A detailed discussion of the relationship between the two is provided in Section \ref{Remark of definition of universal L2 torsion}.

\subsection{Universal $L^2$-torsion of chain complexes}
    A chain complex $C_*$ is called a \emph{finite based free $\Z G$-chain complex} if there exists $n\geqslant 0$ such that 
    \[
        C_*=(0\ra C_n\ra\cdots\ra C_1\ra C_0\ra 0),
    \]
    where each $C_k$ is a finitely generated free (left) $\Z G$-module equipped with a preferred unordered free $\Z G$-basis, and the boundary operators are $\Z G$-linear maps.
    \begin{definition}
        A finite based free $\Z G$-chain complex $C_*$ is said to be \emph{$L^2$-acyclic} if the chain complex $\ell^2(G)\otimes_{\Z G} C_*$ is weakly exact, i.e. the image $\im \partial^{(2)}_{k+1}$ is dense in $\ker\partial^{(2)}_k$ for all $k$, where $\partial^{(2)}_k\colon \ell^2(G)\otimes_{\Z G} C_k\ra \ell^2(G)\otimes_{\Z G} C_{k-1}$ is the boundary operator.
    \end{definition}
     We will not work with those analytic-flavored definitions but prefer the more algebraic-flavored ones given by the following Lemma \ref{Lemma algebraic L2 acyclic}.
    \begin{lemma}[{\cite[Lemma 2.21]{friedl2017universal}}]\label{Lemma algebraic L2 acyclic}
        A square matrix over $\Z G$ is a weak isomorphism if and only if it is invertible over $\mcalD G$. A finite based free $\Z G$-chain complex $C_*$ is $L^2$-acyclic if and only if the chain complex $\mcalD G\otimes_{\Z G}C_*$ is exact as a chain complex of left $\mcalD G$-modules.
    \end{lemma}    
    It is a classical topic to define the torsion of an exact chain complex of free modules, see for example \cite{milnor1966whitehead,cohen1973course,Turaev2001IntroductionToComb}. Our situation is particularly nice since $\mcalD G$ is a skew field and any module over a skew field is free. We follow the definitions given in \cite[Section 3]{Turaev2001IntroductionToComb}.
    \begin{definition}
        Suppose $V$ is a finitely generated $\mcalD G$-module with $\dim V=k$. For any two (unordered) bases $b=\{b_1,\ldots,b_k\}$ and $c=\{c_1,\ldots,c_k\}$, we have 
        \[
            b_i=\sum_{j=1}^k a_{ij}c_j,\quad i=1,\ldots,k,
        \]
        where the transition matrix $(a_{ij})_{i,j=1,\ldots,k}$ is a non-degenerate $(k\times k)$-matrix over $\mcalD G$. Define $[b/c]$ to be the determinant $\redet(a_{ij})\in\widetilde K_1(\mcalD G)$. This definition is independent of the ordering of the elements in $b$ and $c$.
    \end{definition}
    \begin{definition}[Universal $L^2$-torsion of chain complexes]\label{Definition of universal L2 torsion of chain complexes}
        Let $C_*$ be a finite based $\Z G$-chain complex of length $n$ that is $L^2$-acyclic. For each $i$, let $c_i$ be the preferred basis of the free $\mcalD G$-module $\mcalD G\otimes _{\Z G}C_i$ and let $\partial_i$ be the boundary homomorphism. Choose a basis $b_i$ for the free $\mcalD G$-module $B_i:=\im \partial _i$. Then $b_i$ and $b_{i-1}$ combines to form a basis $b_ib_{i-1}$ of $\mcalD G\otimes _{\Z G}C_i$. The \emph{universal $L^2$-torsion of $C_*$} is defined as
    \[
        \tautwo_u(C_*):=\prod_{i=0}^n[b_ib_{i-1}/c_i]^{(-1)^{i+1}}\in \widetilde K_1(\mcalD G).
    \]
    This value is independent of the choice of the bases $b_i$.
    \end{definition}

    An exact sequence $0\ra M_0\stackrel{i}\ra M_1\stackrel{p}\ra M_2\ra 0$ of based free $\Z G$-modules is called \emph{based exact}, if $i(b_0)\subset b_1$ and $p$ maps $b_1\setminus i(b_0)$ bijectively to $b_2$, where $b_i$ is the preferred basis of $M_i$ for $i=0,1,2$. Similarly, an exact sequence $0\ra C_*\ra D_*\ra E_*\ra 0$ of finite based free $\Z G$-chain complex is called \emph{based exact} if in each degree $k$ the sequence
    \[
        0\ra C_k\ra D_k\ra E_k\ra 0
    \]
    is based exact. The following basic property can be found in \cite[Theorem 3.4]{Turaev2001IntroductionToComb}.

    \begin{proposition}\label{Proposition Axiom of torsion of chain complexes}
        The universal $L^2$-torsion satisfies the following properties.
        \begin{enumerate}[\quad\rm (1)]
            \item For any $L^2$-acyclic finite based free $\Z G$-chain complex
            \[
                C_*=(0\ra C_1\stackrel{A}\ra C_0\ra 0)
            \]
            where $A$ is a square matrix over $\Z G$, we have $\tautwo_u(C_*)=\redet (A)\inv\in \widetilde K_1(\mcalD G)$.
            \item Suppose $0\ra C'_*\ra C_*\ra C''_*\ra 0$ is a based exact sequence of finite based free $\Z G$-chain complexes. If $C_*'$ and $C_*''$ are $L^2$-acyclic then so is $C_*$, and $$\tautwo_u(C_*)=\tautwo_u(C'_*)\cdot\tautwo_u(C''_*)\in \widetilde K_1(\mcalD G).$$
        \end{enumerate}
    \end{proposition}

    \begin{definition}[Dual $\Z G$-modules]
        Recall that $\Z G$ is a ring with an involution $x\mapsto x^*$ which sends $\sum n_g\cdot g$ to $\sum n_g\cdot g\inv$. For any left $\Z G$-module $A$, its \emph{dual module} $A^*$ is defined to be $\operatorname{Hom}_{\Z G}(A,\Z G)$, considered as a left $\Z G$-module as follows:  for each $x\in \Z G$ and $f\in A^*$, the map $xf\colon A\ra \Z G$ is given by $(xf)(y)=f(y)\cdot x^*,\ \forall y\in A$.

If $A$ is a free $\Z G$-module with basis $a_i$, then $A^*$ is also free and admits a dual basis $a_i^*$. Moreover, if $f\colon A\ra B$ is a $\Z G$-linear map between based free $\Z G$-modules represented by a matrix $P$ with respect to the given bases, then the dual map $f^*\colon B^*\ra A^*$ is represented by the matrix $P^*$, the involution transpose of $P$.
    \end{definition}

The involution on $\Z G$ is compatible with taking adjoint in $\mathcal UG$. Since $\mcalD G$ is the division closure of $\Z G$ in $\mathcal{U}G$, and the set $\mcalD G\cap (\mcalD G )^*$ forms an inversion-closed subring containing $\mathbb{Z} G$, it follows that Linnell's skew field $\mcalD G$ is itself closed under taking adjoint. Therefore $\mcalD G$ admits a canonical involution extending that of $\Z G$. This involution extends to the $K_1$-group and the Whitehead group of $\mcalD G$. The following Proposition \ref{Proposition torsion of dual chain complex} is a classic property of torsion invariants and can be proved as in \cite{milnor1962duality}.

\begin{proposition}\label{Proposition torsion of dual chain complex}
    If $C_*=(0\ra C_n\ra\cdots\ra C_0\ra 0)$ is a finite based free $\Z G$-chain complex 
    which is $L^2$-acyclic. Then the dual chain complex
    $C^{*}$ is $L^2$-acyclic and 
    \[
        (\tautwo_u(C^*))^*=\tautwo_u(C_*)^{(-1)^{n+1}}\in \widetilde K_1(\mcalD G).
    \]
\end{proposition}

\subsubsection{Comparing Friedl--L\"uck's definition}\label{Remark of definition of universal L2 torsion} Our Definition \ref{Definition of universal L2 torsion of chain complexes} of the universal $L^2$-torsion slightly differs from that of \cite{friedl2017universal}. We first recall the definition of the weak $K_1$-groups introduced by \cite{friedl2017universal}, which is an abelian group analogous to the classical $K_1$-group, but rather than requiring the matrices over $\Z G$ to be invertible, we only require that they are weak isomorphisms (recall Definition \ref{Definition of weak isomorphisms}). 

\begin{definition}[$K_1^w(\Z G),$ $\widetilde K_1^w(\Z G)$ and $\operatorname{Wh}^w(G)$]
    Suppose $G$ is any group. The \emph{weak $K_1$-group} $K_1^w(\Z G)$ is defined in terms of generators and relations as follows. Generators $[A]$ are given by square matrices $A$ over $\Z G$ such that $A$ is a weak isomorphism. There are two sets of relations:
    \begin{enumerate}[\rm\quad (i)]
        \item If $A,B$ are square matrices over $\Z G$ of the same size such that $A,B$ are weak isomorphisms, then
        \[
            [AB]=[A]\cdot [B].
        \]
        \item If $A,B$ are square matrices over $\Z G$ of size $n$ and $m$, respectively. Suppose $A,B$ are weak isomorphisms and let $C$ be any matrix over $\Z G$ of size $n\times m$. Then
        \[
            \bigg[\begin{pmatrix}
                A & C\\
                0 & B
            \end{pmatrix}\bigg]=[A]\cdot [B]=[B]\cdot [A].
        \]
    \end{enumerate}
    The \emph{reduced weak $K_1$-group} $\widetilde K_1^w(\Z G)$ is defined to be the quotient of $K_1^w(\Z G)$ by the subgroup $\{[1],[-1]\}$. The \emph{weak Whitehead group $\operatorname{Wh}^w(G)$} is defined to be the quotient of $K_1^w(\Z G)$ by the subgroup $\{[\pm g]\mid g\in G\}$.
\end{definition}

\begin{proposition}\label{Proposition Natural homo from weak K1 to K1}
    Suppose $G$ is a torsion-free group which satisfies the Atiyah Conjecture. Then the following holds:
    \begin{enumerate}[\rm\quad (1)]
        \item There are natural homomorphisms $$i_G\colon K_1^w(\Z G)\ra K_1(\mcalD G),\quad \tilde i_G\colon \widetilde K_1^w(\Z G)\ra \widetilde K_1(\mcalD G),\quad i_G^w\colon \operatorname{Wh}^w(G)\ra \operatorname{Wh}(\mcalD G).$$
        \item If $G$ falls into the Linnell's class $\mathcal C$ (recall Theorem \ref{Theorem Linnell's class C}), then the three homomorphisms above are all isomorphisms.
    \end{enumerate}
\end{proposition}
\begin{proof}
    By Lemma \ref{Lemma algebraic L2 acyclic}, a square matrix $A$ over $\Z G$ is a weak isomorphism if and only if it is invertible over $\mcalD G$. The homomorphisms in (1) are given by viewing a weak isomorphism $A$ over $\Z G$ as an invertible matrix over $\mcalD G$. (2) is the content of \cite{luck2017localization}.
\end{proof}

        For any group $G$ and any $L^2$-acyclic finite based $\Z G$-chain complex $C_*$, the universal $L^2$-torsion $\rho^{(2)}_u(C_*)$ defined in \cite{friedl2017universal} takes values in $\widetilde K_1^w(\Z G)$. Proposition \ref{Proposition Axiom of torsion of chain complexes} holds true for $\rho^{(2)}_u$ and moreover it characterizes the universal property of $\rho^{(2)}_u$ (in fact, in \cite{friedl2017universal} the torsion $\rho_u^{(2)}(0\ra C_1\stackrel{A}\ra C_0\ra 0)$ is defined to be $[A]\in\widetilde K_1^w(\Z G)$, rather than the inverse $[A\inv]$). Therefore when $G$ is torsion-free and satisfies the Atiyah Conjecture, the universal $L^2$-torsion $\tautwo_u(C_*)$ defined in the present paper is the image of $(-\rho^{(2)}_u(C_*))$ under the natural map
        \[
            \tilde i_G\colon \widetilde K_1^w(\Z G)\ra \widetilde K_1(\mcalD G).
        \]
         If furthermore $G$ lies in Linnell's class $\mathcal C$ then $\tilde i_G$ is a canonical isomorphism by Proposition \ref{Proposition Natural homo from weak K1 to K1}, and the definitions of $\tautwo_u$ and $\rho_u^{(2)}$ are equivalent. In particular, this includes all $3$-manifold groups by Theorem \ref{Theorem 3-manifold group in class C}.
        
        \begin{remark}\label{Remark Convention of torsion}
            In the literature, the torsion of a chain complex is defined under two major conventions. Consider the based exact chain complex \[
                C_*=(0\ra C_1\stackrel{A}\ra C_0\ra 0).
            \] The first convention $\rho$, commonly used for analytic torsions, typically defines $\rho(C_*)$ as $\det A$, viewing the torsion as an element of an \emph{additive} group \cite{milnor1966whitehead,ray1971R-torsion,cohen1973course,lueck2002l2,friedl2017universal,Lueck2018twisting}. The second convention $\tau$, common in topological contexts such as the theory of Alexander polynomials, defines $\tau(C_*)$ to be $(\det A)\inv$ and treats the torsion as an element of a \emph{multiplicative} group \cite{milnor1962duality, kitano1996twisted, Turaev2001IntroductionToComb,friedl2011survey,Dubois2016L2AlexanderTorsion,liu2017degree}. Our paper follows the latter convention.
        \end{remark}
    
\subsection{Universal $L^2$-torsion of CW-complexes}
Let $X$ be a connected finite CW-complex with torsion-free fundamental group $G$ satisfying the Atiyah Conjecture and let $Y\subset X$ be a subcomplex. Let $p\colon \widehat X\ra X$ be the universal covering of $X$ and $\widehat Y:= p\inv (Y)$ be the preimage. Then $\widehat X$ admits the induced CW-structure and $\widehat Y$ is a subcomplex of $\widehat X$.

The natural left $G$-action on $\widehat X$ gives rise to the left $\Z G$-module structure on the cellular chain complex $C_*(\widehat X,\widehat Y)$. By choosing a lift $\widehat \sigma$ for each cell $\sigma$ in $X\setminus Y$,  $C_*(\widehat X,\widehat Y)$ becomes a finite based free $\Z G$-chain complex. The following definition is independent of the choice of lifts.
\begin{definition}[Universal $L^2$-torsion of CW-complexes]
    Let $X$ be a finite connected CW-complex with fundamental group $G$ and let $Y$ be a subcomplex of $X$. The pair $(X,Y)$ is called \emph{$L^2$-acyclic} if the finite based free $\Z G$-chain complex $C_*(\widehat X,\widehat Y)$ is $L^2$-acyclic (c.f.\, Lemma \ref{Lemma algebraic L2 acyclic}). The \emph{universal $L^2$-torsion} $$\tautwo_u(X,Y)\in\operatorname{Wh}(\mcalD G)\sqcup \{0\}$$ is defined as follows: 
    if $(X,Y)$ is $L^2$-acyclic, define $\tautwo_u(X,Y)$
    to be the image of $\tautwo_u(C_*(\widehat X,\widehat Y))$ under the quotient map $\widetilde K_1(\mcalD G)\ra \operatorname{Wh}(\mcalD G)$. Otherwise, define $\tautwo_u(X,Y)$ to be zero.
\end{definition}

When $X$ is not necessarily connected, we say the pair $(X,Y)$ is \emph{$L^2$-acyclic} if for every component $X_i\in \pi_0(X)$ the pair $(X_i,X_i\cap Y)$ is $L^2$-acyclic. Moreover, if the fundamental group $\pi_1(X_i)$ of each component is torsion-free satisfying the Atiyah Conjecture, we define
\[
    \operatorname{Wh}(\mcalD{\Pi(X)}):=\bigoplus_{X_i\in \pi_0(X)} \operatorname{Wh}(\mcalD{\pi_1(X_i)}),\]
    and set \[
    \tautwo_u(X,Y):=(\tautwo_u(X_i,X_i\cap Y))_{X_i\in \pi_0(X)}\in \operatorname{Wh}(\mcalD {\Pi(X)}).\]
If $(X_i,X_i\cap Y)$ is not $L^2$-acyclic for some $i$, we define $\tautwo_u(X,Y):=0$.

Now let $(X,Y)$ and $(X',Y')$ be finite CW-pairs. A CW-mapping $f\colon (X,Y)\ra (X',Y')$ is called \emph{$\pi_1$-injective} if the restriction of $f$ to each component of $X$ induces an injection on fundamental groups. In this case, there is an induced homomorphism $$\iota_*\colon \operatorname{Wh}(\mcalD {\Pi(X)})\sqcup\{0\}\ra \operatorname{Wh}(\mcalD {\Pi(X')})\sqcup\{0\},$$ and we define the pushforward of the universal $L^2$-torsion
    \[
        f_*\tautwo_u(X,Y)\in \operatorname{Wh}(\mcalD {\Pi(X')})\sqcup\{0\}
    \]
    as the image of $\tautwo_u(X,Y)$ under $f_*$. By definition $f_*\tautwo_u(X,Y)=0$ if and only if $\tautwo_u(X,Y)=0$.

\begin{theorem}We record the fundamental properties of the universal $L^2$-torsion.\label{Theorem properties of the torsion of cw complexes}
\begin{enumerate}[\quad\rm (1)]
    \item {\rm (Simple-homotopy invariance)} Let $f\colon(X,X_0)\ra (Y,Y_0)$ be a continuous map of finite CW-pairs such that $f\colon X\ra Y$ and $f|_{X_0}\colon X_0\ra Y_0$ are both simple-homotopy equivalences. Then $$\tautwo_u(Y,Y_0)=f_*\tautwo_u(X,X_0).$$

     \item {\rm (Sum formula)} Let $(U,V)=(X,C)\cup (Y,D)$ where $(X,C)$, $(Y,D)$ and $(X\cap Y,C\cap D)$ are $L^2$-acyclic sub-pairs that embeds $\pi_1$-injectively into $(U,V)$, then 
        \[
            \tautwo_u(U,V)=(\iota_1)_*\tautwo_u(X,C)\cdot(\iota_2)_*\tautwo_u(Y,D)\cdot(\iota_3)_*\tautwo_u(X\cap Y,C\cap D)\inv
        \]
       where $\iota_1, \iota_2, \iota_3$ are the inclusions of $(X, C)$, $(Y, D)$, and $(X \cap Y, C \cap D)$ into $(U, V)$, respectively.

        \item {\rm (Induction)} Let $f\colon (X_0,Y_0)\subset (X,Y)$ be a $\pi_1$-injective inclusion. 
        Let $\widehat X$ be the universal cover of $X$ and let $\widehat  X_0,\widehat Y_0$ be the preimage of $X_0,Y_0$ in $\widehat X$. Then 
        \[
            \tautwo_u(C_*(\widehat X_0,\widehat Y_0))=f_*\tautwo_u(X_0,Y_0).
        \]
        
        \item {\rm (Restriction)} Let $X$ be a connected finite CW-complex with $\pi_1(X)=G$, and let $\overline X$ be a connected finite degree regular covering of $X$ with $\pi_1(\overline X)=H<G$. Denote by $\operatorname{res}^G_H\colon \operatorname{Wh}(\mcalD G)\sqcup\{0\}\ra \operatorname{Wh}(\mcalD H)\sqcup\{0\}$ the restriction homomorphism (to be discussed in details in Section \ref{Section of restriction map}). For a subcomplex $Y\subset X$ with preimage $\overline Y\subset\overline X$, we have
        \[
            \tautwo_u(\overline X,\overline Y)=\operatorname{res}^G_H\tautwo  _u(X,Y).
        \]
        \end{enumerate}
\end{theorem}

\begin{proof}
    Property (1) will be proved in Remark \ref{Remark simple homotopy invariance} after introducing the universal $L^2$-torsion for mappings.
    Properties (2)--(4) are natural generalization of \cite[Theorem 3.5]{friedl2017universal} to CW-pairs and the proof carry over without essential changes to the relative cases. 
\end{proof}

\subsection{Universal $L^2$-torsion of mappings}
 \begin{definition}[Universal $L^2$-torsion of mappings]
     Let $f\colon Y\ra X$ be a cellular map between finite CW-complexes. The \textit{mapping cylinder} of $f$ is a finite CW-complex $$M_f:=((Y\times I)\sqcup  X)/\sim,\quad \text{where } (y,0)\sim f(y) \text{ for all $y\in Y$}.$$ We identify $Y$ with the subcomplex $Y\times\{1\}$ of $M_f$. If the fundamental group of $X$ is torsion-free satisfying the Atiyah Conjecture, then the \emph{universal $L^2$-torsion of the mapping $f$} is defined as
     \[
         \tautwo_u(f):=\iota_*\tautwo_u(M_f,Y)\in\operatorname{Wh}(\mcalD {\Pi(X)})\sqcup \{0\}
     \]
    where $\iota\colon M_f\ra X$ is the canonical deformation retraction.
\end{definition}

\begin{definition}[$L^2$-homology equivalence] A map $f\colon Y\ra X$ is called an \emph{$L^2$-homology equivalence} if $\tautwo_u(f)\not=0$, or equivalently $(M_f,Y)$ is $L^2$-acyclic.  
\end{definition}

\begin{proposition}\label{Proposition Properties of torsion of mappings}
    Suppose the spaces $X,Y,Z$ are finite CW-complexes whose fundamental groups are torsion-free satisfying the Atiyah Conjecture.
    \begin{enumerate}[\quad \rm (1)]
        \item If $(X,Y)$ is a CW-pair and $f\colon Y\ra X$ is the inclusion map. Then $$\tautwo_u(X,Y)=\tautwo_u(f).$$
        \item If $f,g\colon X\ra Z$ are homotopic cellular maps. Then $\tautwo_u(f)=\tautwo_u(g)$.
        \item If $f\colon X\ra Z$ is a simple-homotopy equivalence, then $\tautwo_u(f)=1$.
        \item If either one of $f\colon X\ra Y$ and $g\colon Y\ra Z$ is an $L^2$-homology equivalence. Suppose $g$ is $\pi_1$-injective. Then $$\tautwo_u(g\circ f)=g_*\tautwo_u(f)\cdot \tautwo_u(g).$$
    \end{enumerate}
\end{proposition}
\begin{proof}
    Following \cite{cohen1973course}, we write $K\curvearrowright L$ if two finite CW-complexes $K$ and $L$ are related by a finite sequence of elementary collapses or expansions. If there is a common subcomplex $K_0$ that remains unchanged during the process, we write $K\curvearrowright L$ rel $K_0$. In this case, we have $\tautwo_u(K,K_0)=\iota_*\tautwo_u(L,K_0)$ where $\iota\colon L\ra K$ is the natural homotopy equivalence. 

    We now prove the proposition part by part:
    
    For (1), there are elementary expansions $M_f\curvearrowright X\times I$ and elementary collapses $X\times I\curvearrowright X\times\{1\}$, both relative to $Y=Y\times\{1\}$ \cite[(5.1B)]{cohen1973course}. Let $\iota\colon M_f\ra X$ be the deformation retract, then $\tautwo_u(f)=\iota_*\tautwo_u(M_f,Y)=\tautwo_u(X,Y)$.
    
    For (2), if $f,g\colon X\ra Z$ are homotopic cellular maps, then $M_f\curvearrowright M_g$ rel $X$ \cite[(5.5)]{cohen1973course}, which implies $\tautwo_u(f)=\tautwo_u(g)$.
    
    For (3), if $f\colon X\ra Z$ is a simple-homotopy equivalence then $M_f\curvearrowright X$ rel $X$ \cite[(5.8)]{cohen1973course} and $\tautwo_u(f)=1$.
    
    For the proof of (4) we need the following ``$L^2$-excision" property:
    \begin{lemma}[Excision]\label{Lemma L2 Excision}
        Let $K,L$ be subcomplexes of the complex $K\cup L$ and let $M=K\cap L$. If the inclusion  $i\colon K\hookrightarrow K\cup L$ is  $\pi_1$-injective. Then $\tautwo_u(K\cup L,L)=i_*\tautwo_u(K,M)$.
    \end{lemma}
\begin{proof}
    As in \cite[(20.3)]{cohen1973course}, we may assume that $L$ and $K\cup L$ are connected. Let $\widehat{K\cup L}$ be the universal covering of $K\cup L$ and let $\widehat L$, $\widehat K$ and $\widehat M$ be the preimages of $L$, $K$ and $M$ under the covering, respectively. There is an isomorphism of chain complexes $C_*(\widehat{K\cup L},\widehat L)=C_*(\widehat K,\widehat M)$, so $\tautwo_u(K\cup L,L)=\tautwo_u(C_*(\widehat K,\widehat M))=i_*\tautwo_u( K, M)$ where the second equality follows from the induction property (Theorem \ref{Theorem properties of the torsion of cw complexes}).
\end{proof}

We now complete the proof of Proposition \ref{Proposition Properties of torsion of mappings}(4). Let $M$ be the union of $M_f$ and $M_g$ along the identity map on $Y$. Then $M\curvearrowright M_{g\circ f}$ rel $X\cup Z$ by \cite[(5.6)]{cohen1973course}. There is a diagram commutative up to homotopy:
\[
\begin{minipage}[c]{0.45\textwidth}
\centering
\small
    
\def\svgwidth{.5\columnwidth}
\begingroup%
  \makeatletter%
  \providecommand\color[2][]{%
    \errmessage{(Inkscape) Color is used for the text in Inkscape, but the package 'color.sty' is not loaded}%
    \renewcommand\color[2][]{}%
  }%
  \providecommand\transparent[1]{%
    \errmessage{(Inkscape) Transparency is used (non-zero) for the text in Inkscape, but the package 'transparent.sty' is not loaded}%
    \renewcommand\transparent[1]{}%
  }%
  \providecommand\rotatebox[2]{#2}%
  \newcommand*\fsize{\dimexpr\f@size pt\relax}%
  \newcommand*\lineheight[1]{\fontsize{\fsize}{#1\fsize}\selectfont}%
  \ifx\svgwidth\undefined%
    \setlength{\unitlength}{53.51119274bp}%
    \ifx\svgscale\undefined%
      \relax%
    \else%
      \setlength{\unitlength}{\unitlength * \real{\svgscale}}%
    \fi%
  \else%
    \setlength{\unitlength}{\svgwidth}%
  \fi%
  \global\let\svgwidth\undefined%
  \global\let\svgscale\undefined%
  \makeatother%
  \begin{picture}(1,0.79001317)%
    \lineheight{1}%
    \setlength\tabcolsep{0pt}%
    \put(0,0){\includegraphics[width=\unitlength,page=1]{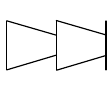}}%
    \put(-0.00400916,0.65741606){\color[rgb]{0,0,0}\makebox(0,0)[lt]{\lineheight{1.25}\smash{\begin{tabular}[t]{l}$X$\end{tabular}}}}%
    \put(0.45475657,0.65741606){\color[rgb]{0,0,0}\makebox(0,0)[lt]{\lineheight{1.25}\smash{\begin{tabular}[t]{l}$Y$\end{tabular}}}}%
    \put(0.9000291,0.65741606){\color[rgb]{0,0,0}\makebox(0,0)[lt]{\lineheight{1.25}\smash{\begin{tabular}[t]{l}$Z$\end{tabular}}}}%
    \put(0.18826754,0.36639645){\color[rgb]{0,0,0}\makebox(0,0)[lt]{\lineheight{1.25}\smash{\begin{tabular}[t]{l}$M_f$\end{tabular}}}}%
    \put(0.64560096,0.36639645){\color[rgb]{0,0,0}\makebox(0,0)[lt]{\lineheight{1.25}\smash{\begin{tabular}[t]{l}$M_g$\end{tabular}}}}%
    \put(0.45138312,0.00933752){\color[rgb]{0,0,0}\makebox(0,0)[lt]{\lineheight{1.25}\smash{\begin{tabular}[t]{l}$M$\end{tabular}}}}%
  \end{picture}%
\endgroup%

\end{minipage}
\hfill
\begin{minipage}[c]{0.45\textwidth}
\centering
\begin{tikzcd}
X \arrow[r, "f"] \arrow[d, "i_X"', hook]                                              & Y \arrow[d, "i_Y", hook] \arrow[r, "g"]                & Z                           \\
M_f \arrow[ru, "\iota_f" description] \arrow[rr, "i_1" description, hook, bend right] & M_g \arrow[ru, "\iota_g" description] \arrow[r, "i_2", hook] & M \arrow[u, "\iota"']
\end{tikzcd}
\end{minipage}\]
where $i_Y,i_1,i_2$ are $\pi_1$-injective inclusions and $\iota,\iota_f,\iota_g$ are the canonical deformation retracts.

By Excision, $\tautwo_u(M,M_f)=(i_2)_*\tautwo_u(M_g,Y)$.
Since at least one of $(M_f,X)$ and $(M_g,Y)$ is $L^2$-acyclic, the short exact sequence
\[
    0\ra C_*(\widehat {M_f},\widehat X)\ra C_*(\widehat M, \widehat X)\ra C_*(\widehat M, \widehat {M_f})\ra 0
\]
implies that $\tautwo_u(M,X)=\tautwo_u(M,M_f)\cdot (i_1)_{*}\tautwo_u(M_f,X)$. Therefore
\begin{align*}
    \tautwo_u(g\circ f)&=\iota_*\tautwo_u(M,X)\\&
    =\iota_*((i_2)_*\tautwo_u(M_g,Y)\cdot (i_1)_{*}\tautwo_u(M_f,X) )\\&
    =(\iota_g)_*\tautwo_u(M_g,Y)\cdot g_*(\iota_f)_* \tautwo_u(M_f,X)
    \\&
    =\tautwo_u(g)\cdot g_*\tautwo_u(f).
\end{align*}
This completes the proof.
\end{proof}
    
The identity $\tautwo_u(g\circ f)=g_*\tautwo_u(f)\cdot \tautwo_u(g)$ is called the \emph{multiplicativity} of the universal $L^2$-torsion. The conditions that either $f$ or $g$ is an $L^2$-homology equivalences and $g$ is $\pi_1$-injective are in general necessary. But when one of the maps is a simple homotopy equivalence, the conditions can be relaxed as follows.
\begin{lemma}\label{Lemma simple homotopy invariance of mappings}
    Let $X,Y,Z,W$ be finite CW-complexes and consider the chain of maps
        $X\xrightarrow{f}  Y\xrightarrow{g}Z\xrightarrow{h} W$.
    Suppose $g\colon Y\ra Z$ is a simple homotopy equivalence. Then
    \begin{enumerate}[\quad\rm(1)]
        \item 
        $\tautwo_u(g\circ f)=g_*\tautwo_u(f)$, and
        \item 
        $\tautwo_u(h\circ g)=\tautwo_u(h)$.
    \end{enumerate}
\end{lemma}
\begin{proof}
    (1) follows directly from Proposition \ref{Proposition Properties of torsion of mappings}(3), (4).

    

    For (2), let $M$ be the union of $M_g$ and $M_h$ along the identity map on $Z$. Then $M\curvearrowright M_{h\circ g}$ rel $Y\cup W$. The Excision Lemma \ref{Lemma L2 Excision} applied to $i\colon M_h\subset M$ shows that  $\tautwo_u(M,M_g)=i_*\tautwo_u(M_h,Z)$.  Since $g$ is a simple homotopy equivalence, we have $M_g\curvearrowright Y$ rel $Z$. Let $\iota\colon M\ra W$ be the deformation retract. Then
    \[\tautwo_u(h\circ g)=\iota_*\tautwo_u(M,Y)=\iota_*\tautwo_u(M,M_g)=(\iota\circ i)_*\tautwo_u(M_h,Z)=\tautwo_u(h),
    \]
    completing the proof.
\end{proof}
\begin{remark}\label{Remark simple homotopy invariance}
    As a corollary of Lemma \ref{Lemma simple homotopy invariance of mappings}, we prove the simple-homotopy invariance stated in Theorem \ref{Theorem properties of the torsion of cw complexes}. Let $f\colon (X,X_0)\ra (Y,Y_0)$ be a mapping of CW-pairs such that $f\colon X\ra Y$ and $f|_{X_0}\colon X_0\ra Y_0$ are simple-homotopy equivalences. Then we have the following commutative diagram
    \[
    \begin{tikzcd}
X_0 \arrow[r, "f_0"] \arrow[d, "i_X"', hook] & Y_0 \arrow[d, "i_Y", hook] \\
X \arrow[r, "f"]                             & Y                         
\end{tikzcd}
    \]
    and
    \[
        f_*\tautwo_u(X,X_0)=f_*\tautwo_u(i_X)=\tautwo_u(f\circ i_X)=\tautwo_u(i_Y\circ f_0)=\tautwo_u(i_Y)=\tautwo_u(Y,Y_0).
    \]
\end{remark}

\subsection{Universal $L^2$-torsion of manifolds}
We now define the universal $L^2$-torsion for smooth manifold pairs. Recall that a \emph{smooth triangulation} of a smooth manifold $M$ is a homeomorphism from a simplicial complex to $M$ that is smooth on each simplex.
\begin{definition}[Universal $L^2$-torsion of manifold pairs]
    Let $M$ be a compact, smooth manifold, possibly with boundary, and let $N$ be a compact, smooth submanifold of $M$. Suppose $M$ admits a smooth triangulation $X$ in which $N$ is a subcomplex $Y$. Then we define $\tautwo_u(M,N):=\tautwo_u(X,Y)$.
\end{definition}
    In fact, this is a special case of the following:


\begin{definition}[Universal $L^2$-torsion of mappings between manifolds]
    Let $f\colon N\ra M$ be a continuous mapping between compact smooth manifolds (possibly with boundaries). Choose smooth triangulations of $M,N$ and let $g$ be a simplicial approximation of $f$. 
    The \emph{universal $L^2$-torsion of $f$} is defined as
\[
    \tautwo_u(f):=\tautwo_u(g)\in\operatorname{Wh}(\mcalD{\Pi(M)})\sqcup\{0\}.
\]
\end{definition}

It follows from Lemma \ref{Lemma simple homotopy invariance of mappings} that $\tautwo_u(f)$ is independent of the choice of triangulations on $M,N$, since any two such triangulations have a common subdivision and are simple homotopy equivalent as CW-complexes \cite{Whitehead1940OnC1Complexes}. The definition is also independent of the simplicial approximation $g$ by homotopy invariance. In particular, if $N\subset M$ is a smooth submanifold with $f$ the inclusion map, then $\tautwo_u(f)=\tautwo_u(M,N)$.

\subsection{Methods for computations}
        We now present the \emph{matrix chain method} for computing the universal $L^2$-torsion of a chain complex (c.f. \cite[Theorem 2.2]{Turaev2001IntroductionToComb}).

        Let $C_*=(0\ra C_n\ra\cdots\ra C_1\ra C_0\ra 0)$ be a finite based free $\Z G$-chain complex and let $\partial_i\colon C_i\ra C_{i-1}$ be the boundary operator. Suppose $d_i:= \operatorname{rank}_{\Z G}C_i$ is the rank of the free module $C_i$. Then $\partial_i$ is given by a matrix 
        \[
            A_i=(a^i_{jk})_{\begin{subarray}{l}
                j=1,\ldots,d_i\\ k=1,\ldots,d_{i-1}
            \end{subarray} },\quad a^i_{jk}\in \Z G.
        \]
        \begin{definition}
            A \emph{matrix chain} for $C_*$ is a collection of finite sets $\mathcal A=\{\mathcal I_0,\ldots,\mathcal I_n\}$ where
         $\mathcal I_i\subset\{1,\ldots, d_i\}$ and $\mathcal I_n=\emptyset$.    
     Let $B_i$ be the submatrix of $A_i$ formed by the entries $a^i_{jk}$ with $j\notin \mathcal I_i$ and $k\in \mathcal I_{i-1}$. Then $\{B_i\}$ are called the \emph{matrices associated to the matrix chain}.
     
     A matrix chain is called \emph{non-degenerate} if each associated matrix $B_i$ is a square matrix invertible over $\mcalD G$.
    \end{definition}
\begin{theorem}\label{Matrix chain for acyclic chain complex}
    A finite based free $\Z G$-chain complex $C_*$ is $L^2$-acyclic if and only if there exists a non-degenerate matrix chain $\mathcal A=\{\mathcal I_0,\ldots,\mathcal I_n\}$ for $C_*$. If this happens, then 
    \[
        \tautwo_u(C_*)=\prod_{i=1}^n \redet (B_i)^{(-1)^i}\in \widetilde K_1(\mcalD G)
    \]
    where $B_i$ are the matrices associated to the matrix chain.
\end{theorem}
The proof is a generalization of the idea of \cite[Lemma 3.1]{Dubois2016L2AlexanderTorsion} to larger chain complexes and is well-known to experts. We give a sketched proof here.
\begin{proof}[{Proof Sketch}]
    Suppose $C_*$ is $L^2$-acyclic. Then $A_n$ is injective over $\mcalD G$ and there is a $(d_n\times d_n)$-submatrix $B_n$ of $A_n$ such that $B_n$ is invertible over $\mcalD G$. Let $\mathcal I_{n-1}\subset \{1,\ldots,d_{n-1}\}$ be the set of indices of the columns of $B_n$. Write $C_{n-1}=C_{n-1}'\oplus C_{n-1}''$ where $C_{n-1}'$ corresponds to $\mathcal I_{n-1}$ and $C_{n-1}''$ corresponds to the remaining indices. The boundary map $A_{n-1}\colon C_{n-1}\ra C_{n-2}$ vanishes on $C_{n-1}'$, yielding the following commuting diagram with exact rows and columns: \[
    \begin{tikzcd}
            & 0 \arrow[d]                           & 0 \arrow[d]                                     &                                            &   \\
0 \arrow[r] & C_n \arrow[r, "="] \arrow[d, "B_n", "\cong"'] & C_n \arrow[d, "A_n"] \arrow[r]                  & 0 \arrow[d]                                &   \\
0 \arrow[r] & C_{n-1}' \arrow[r, ] \arrow[d]     & C_{n-1} \arrow[r, ] \arrow[d, "A_{n-1}"]     & C_{n-1}'' \arrow[d, "A_{n-1}''"] \arrow[r] & 0 \\
            & 0 \arrow[r]                           & C_{n-2} \arrow[r, "="] \arrow[d, "A_{n-2}"] & C_{n-2} \arrow[d, "A_{n-2}"] \arrow[r]     & 0 \\
            &                                       & \vdots                                          & \vdots                                     &  
\end{tikzcd}
    \] 
    Repeating this procedure for the vertical exact sequence on the right eventually yields matrices $B_n,\ldots, B_1$ that form a non-degenerate matrix chain for $C_*$.

    Conversely, suppose $\mathcal A=\{\mathcal I_0,\ldots,\mathcal I_n\}$ is a non-degenerate matrix chain for $C_*$. Decompose $C_{n-1}=C_{n-1}'\oplus C_{n-1}''$, yielding the the commutative diagram as above.    
    Denote the vertical chain complexes as:
    \begin{align*}
        &C_*'=(0\ra C_n\stackrel{B_n}\longrightarrow C_{n-1}'\ra 0 \ra \cdots \ra 0),\\
        &C_*''=(0\ra 0\ra C_{n-1}''\stackrel{A_{n-1}''}\longrightarrow C_{n-2}\stackrel{A_{n-2}}\longrightarrow \cdots\stackrel{}\longrightarrow C_1\stackrel{A_1}\longrightarrow C_0\stackrel{}\longrightarrow 0).
    \end{align*}
    Then $\{\mathcal I_0,\ldots,\mathcal I_{n-1}\}$ is a non-degenerate matrix chain for $C_*''$ and
    we have the short exact sequence of based $\Z G$-chain complex
    \[
        0\ra C_*'\ra C_*\ra C_*''\ra0.
    \]
    By induction on $n$ we may assume that $C_*''$ is $L^2$-acyclic with $$\tautwo_u(C_*'')=\prod_{i=1}^{n-1} \redet (B_i)^{(-1)^i}\in \widetilde K_1(\mcalD G).$$ Then by Proposition \ref{Proposition Axiom of torsion of chain complexes},
    \[
        \tautwo_u(C_*)=\tautwo_u(C_*')\cdot \tautwo_u(C_*'')=\prod_{i=1}^n \redet (B_i)^{(-1)^i}\in \widetilde K_1(\mcalD G).
    \]
\end{proof}

\begin{proposition}\label{Proposition computation example}
    We compute the universal $L^2$-torsion for the following manifold pairs.
    \begin{enumerate}[\rm\quad (1)]
        \item Let $N$ be a compact smooth manifold whose fundamental group is torsion-free satisfying the Atiyah Conjecture. Then for any $s\in [0,1]$,
        \[
            \tautwo_u(N\times I,N\times\{s\})=1\in \operatorname{Wh}(\mcalD {\Pi(N)}).
        \]
       \item Let $S^1$ be the circle with fundamental group $\pi_1(S^1)=\langle t\rangle$. Then $$\tautwo_u(S^1)=[t-1]\inv\in \operatorname{Wh}(\mcalD \Z).$$
       \item Let $T^2$ be the torus. Then
       \[
            \tautwo_u(T^2)=1\in \operatorname{Wh}(\mcalD {\Z^2}).
       \]
    \end{enumerate}
\end{proposition}
\begin{proof}
    For (1), let $f\colon N\times \{s\}\ra N\times I$ be the inclusion map, then $f$ is a simple-homotopy equivalence and $\tautwo_u(N\times I,N\times\{s\})=\tautwo_u(f)=1$ by Proposition \ref{Proposition Properties of torsion of mappings}.

    For (2), consider the standard CW-structure on $S^1$ with one 0-cell $p$ and one 1-cell $e$. After choosing appropriate lifts $\hat p$ and $\hat e$, the cellular chain complex of the universal cover is 
    \[
        C_*(\widehat S^1)=(0\ra \Z[t^\pm]\cdot \langle\hat  e\rangle\xrightarrow{(t-1)}\Z[t^\pm]\cdot \langle \hat p\rangle\ra0)
    \]
    and hence $\tautwo_u(S^1)=[t-1]\inv$.

    For (3), consider the standard CW-structure for $T^2$ given by identifying opposite sides of a square. Let $p$ be the 0-cell, $e_1,e_2$ the 1-cells and $\sigma$ the 2-cell. The boundary of $\sigma$ is the loop $e_1e_2e_1\inv e_2\inv$. If $e_1,e_2$ represents generators $t_1,t_2\in\pi_1(T^2)$ respectively, then with appropriate lifts we obtain the chain complex $C_*(\widehat T^2)$: 
    \[
        0\ra\Z[t_1^\pm,t_2^\pm]\cdot \langle \hat \sigma\rangle\xrightarrow{\begin{pmatrix}
            1-t_2 & t_1-1
        \end{pmatrix}} \Z[t_1^\pm,t_2^\pm]\cdot \langle \hat e_1,\hat e_2\rangle\xrightarrow{\begin{pmatrix}
            t_1-1\\t_2-1
        \end{pmatrix}}\Z[t_1^\pm,t_2^\pm]\cdot \langle \hat p\rangle\ra0.
    \]
    A matrix chain is given by $B_2=(1-t_2)$ and $B_1=(t_2-1)$, hence $\tautwo_u(T^2)=[1-t_2]\cdot [t_2-1]\inv=1.$
\end{proof}

\section{Leading term map, restriction map and polytope map}\label{Section Leading term map restriction map and polytope map}
Let $G$ be a torsion-free group satisfying the Atiyah Conjecture. We will define the leading term map $L_\phi:\mcalD G\ra \mcalD G$ associated to a character $\phi\in H^1(G;\R)$. The main objectives of this section are to establish two key results: Theorem \ref{Theorem Leading term of matrices}, which relates the leading term map to the Dieudonn\'e determinant, and Theorem \ref{Theorem leading term map commutes with restrictions}, which shows that the leading term map commutes with the restriction map. We conclude the section by proving the ``if" direction of Theorem \ref{Main Theorem fibered iff face map zero} using the polytope maps (stated as Theorem \ref{Theorem if part}).

\subsection{Ore localization}
\begin{definition}[Ore localization]
    Let $R$ be a ring and let $S\subset R$ be a multiplicatively closed subset. The pair $(R,S)$ satisfies the \emph{(right) Ore condition} if the following two conditions hold:
    \begin{enumerate}
        \item for any $(r,s)\in R\times S$ there exists $(r',s')\in R\times S$ such that $rs'=sr'$, and
        \item for any $(r,s)\in R\times S$ with $sr=0$, there is $t\in S$ with $rt=0$.
    \end{enumerate}
    If $(R,S)$ satisfies the Ore condition, define an equivalence relation on $R\times S$ by
    \[
        (r,s)\sim(rx,sx) \text{\quad whenever\quad} x\in R
        \text{ and } sx\in S.
    \]
    The quotient set $R\times S/\sim$ is denoted by $RS\inv$. Define a ring structure on $RS\inv$ as follows. Given two representatives $(r,s),(r',s')\in RS\inv$, choose $c\in R,\ d\in S$ such that $sc=s'd\in S$ and define addition by
\[
    (r,s)+(r',s')=(rc+r'd,sc).
\]
Similarly, choose $e\in R,\ f\in S$ with $se=r'f$ and define
\[
    (r,s)\cdot (r',s')=(re,s'f).
\]
The resulting ring $RS\inv$ is called the \emph{Ore localization} of $R$ at $S$.
\end{definition}
Intuitively, a pair $(r,s)\in RS\inv$ is viewed as a formal fraction $rs\inv$. The first Ore condition can be understood as that any ``left fraction" $s\inv r$ can be rewritten as a ``right fraction" $r'(s')\inv$ such that $rs'=sr'$. The second condition is automatically satisfied if $S$ contains no zero divisors. In this paper, we will only need Ore localizations of this simple type:
\begin{lemma}[{\cite[Corollary 1.3.3]{cohn1995skew}}]\label{Lemma Ore domain definition}
    Let $R$ be an integral domain and $R^\times$ the set of nonzero elements. Then $(R,R^\times)$ satisfies the Ore condition if and only if $aR\cap bR\not=\{0\}$ for all nonzero $a,b\in R$. In this case, the Ore localization of $R$ at $R^\times$ is a skew field $K$ and the natural homomorphism $\lambda\colon R\ra K$ is an embedding.
\end{lemma}
\begin{definition}
    An integral domain $R$ satisfying the equivalent conditions in Lemma \ref{Lemma Ore domain definition} is called an \emph{Ore domain}. The skew field $K$ is called the \emph{field of fractions} of $R$.
\end{definition}

\subsection{Crossed products}
Assume that $G$ is a torsion-free group which satisfies the Atiyah conjecture. Consider a short exact sequence of groups
\[
    1\ra K\ra G\xrightarrow{\nu} H\ra 1.
\]
Then $K$ is also torsion-free and satisfies the Atiyah Conjecture by Proposition \ref{Proposition subgroup satisfies Atiyah conjecture}. Denote by $\mcalD K$ and $\mcalD G $ the Linnell's skew fields of $K$ and $G$, respectively.

Choose a set-theoretic section $s\colon H\ra G$ such that $\nu\circ s=\operatorname{id}_H$. Consider the following subset of $\mcalD G$:
\[
    \mcalD K*_sH:=\bigg\{\sum_{h\in H} x_h\cdot s(h)\in \mcalD G\ \bigg|\  x_h\in \mcalD K,\ x_h=0 \text{ for all but finitely many }h\in H\bigg\}.
\]
This set contains the zero element $0$, the identity element $1=s(1_H)\inv\cdot s(1_H)$, and is closed under addition. Moreover, it is closed under multiplication, since
\begin{align*}
    &\bigg(\sum_{h\in H} x_h\cdot s(h)\bigg)\cdot\bigg( \sum_{h\in H} y_h\cdot s(h)\bigg)\\&\quad=\sum_{h_1,h_2\in H} x_{h_1}\cdot s(h_1)\cdot y_{h_2}\cdot s(h_2)\\
    &\quad=\sum_{h_1,h_2\in H}  \underbrace{x_{h_1} s(h_1)y_{h_2}s(h_1)\inv}_{\in \mcalD K}\cdot \underbrace{s(h_1)s(h_2)s(h_1h_2)\inv}_{\in K} \cdot s(h_1h_2).
\end{align*}
Recall that the group automorphism of $K$ given by conjugation by $s(h_1)$ extends to an automorphism of $\mcalD K$ by Proposition \ref{Proposition subgroup satisfies Atiyah conjecture}. It follows that $\mcalD K*_ sH$ is a subring of the skew field $\mcalD G$. 
\begin{proposition}\label{Proposition change another section}
    With notations as above, we have the following properties.
    \begin{enumerate}[\rm\quad(1)]
        \item An element $\sum_{h\in H}x_h\cdot s(h)$ of $\mcalD K*_s H$ is zero if and only if $x_h=0$ for all $h\in H$.
        \item Given another section $s'\colon H\ra G$, then 
        \[
            \sum_{h\in H}x_h\cdot s(h)=\sum_{h\in H}y_h\cdot s'(h)
        \]
        if and only if $y_h=x_hs(h)s'(h)\inv$ for all $h\in H$.
    \end{enumerate}
\end{proposition}
\begin{proof}
    The first statement is a consequence of \cite[Lemma 10.57]{lueck2002l2}. The second statement follows from the previous one.
\end{proof}

As a corollary, the subring $\mcalD K*_sH\subset\mcalD G$ is independent of the choice of the section $s$. We call this subring the \emph{crossed product} of $\mcalD K$ and $H$, denoted by $\mcalD K*H$. Clearly $\mcalD K*H$ is an integral domain since it embeds into the skew field $\mcalD G$.

In some cases, the relation between $\mcalD G$ and its subring $\mcalD K*H$ is particularly simple.

\begin{proposition}\label{Proposition Linnell field of group extension} Let $1\ra K\ra G \ra H\ra 1$ be a group extension.
    \begin{enumerate}[\quad\rm(1)]
        \item If $H$ is finite, then $\mcalD G=\mcalD K* H$.
        \item If $H$ is virtually a finitely generated abelian group, then the integral domain $\mcalD K*H$ is an Ore domain whose field of fractions is $\mcalD G$.
    \end{enumerate}
\end{proposition}
\begin{proof}
    These statements are Lemmas 10.59 and 10.69 of \cite{lueck2002l2}, respectively.
\end{proof}

\subsection{The leading term map}\label{Section leading term map} Assume further that $G$ is finitely generated and consider the short exact sequence
\[
    1\ra K\ra G\xrightarrow{\nu} H_1(G)_f\ra 1.
\] where $\nu\colon G\ra H_1(G)_f$ is the natural quotient map to the free abelian part of the first homology group $H_1(G)_f$. Fix a set-theoretic section $s\colon H_1(G)_f\ra G$ and let $\phi\in H^1(G;\R)$ be a real cohomology class. 
\begin{definition}\label{Definition support of a element in DG}
   For any nonzero element $u\in(\mcalD K*H_1(G)_f)^\times$, write
\[
    u=\sum_{h\in H_1(G)_f}x_h\cdot s(h)\in \mcalD K* H_1(G)_f.
\]
The \emph{support of $u$} is the finite set $\operatorname{supp}(u):=\{h\in H_1(G)_f\mid x_h\not=0\}$, which is independent of the choice of section by Proposition \ref{Proposition change another section}.
Define $\delta_\phi(u)$ to be the minimal value of $\phi(h)$ for all $h\in \operatorname{supp}(u)$. Define
\[
    L_\phi(u):=\sum_{\substack{h\in \operatorname{supp}(u),\\ \phi(h)=\delta_\phi(z)}}x_h\cdot s(h).
\]
This element is nonzero and belongs to $(\mcalD K*H_1(G)_f)^\times$.
\end{definition}

    \begin{lemma}
    The definitions of $\delta_\phi(u)$ and $L_\phi(u)$ are independent of the choice of section. Moreover, for all $u_1,u_2\in (\mcalD K*H_1(G)_f)^\times$,
    \begin{align*}
        &\delta_\phi(u_1u_2)=\delta_\phi(u_1)+\delta_\phi(u_2),\\
        &L_\phi(u_1u_2)=L_\phi(u_1)\cdot L_\phi(u_2).
    \end{align*}
    Hence,
    \[\aligned
        &\delta_\phi\colon (\mcalD K*H_1(G)_f)^\times\ra \R,\\
        &L_\phi\colon (\mcalD K*H_1(G)_f)^\times\ra (\mcalD K*H_1(G)_f)^\times\endaligned
    \]
    are well-defined group homomorphisms.  
\end{lemma}
\begin{proof}
    Let $u=\sum_h x_h\cdot s(h)$ and let $s'$ be another section. By Proposition \ref{Proposition change another section}, $u=\sum_h y_h\cdot s'(h)$ where $y_h=x_hs(h)s'(h)\inv$. It follows that $\delta_\phi(u)$ and $L_\phi(u)$ do not depend on the choice of section. The terms of $u_1u_2$ with minimal $\phi$-value are exactly from the products of that of $u_1$ and $u_2$. This implies the homomorphism properties.
\end{proof}
Recall that $G$ is finitely generated and by Proposition \ref{Proposition Linnell field of group extension} the Linnell's skew field $\mcalD G$ is the field of fractions of the subring $\mcalD K*H_1(G)_f$.
\begin{definition}[Leading term map $L_\phi$]\label{Definition of leading term map}
    The group homomorphisms $\delta_\phi$ and $L_\phi$ extend to group homomorphisms 
    \[
    \begin{aligned}
        &\delta_\phi\colon \mcalD G^\times\ra \R, && \delta_\phi(uv\inv):=\delta_\phi(u)-\delta_\phi(v),\\&
    L_\phi\colon \mcalD G^\times \ra \mcalD G^\times,\ \ && L_\phi(uv^{-1}):=L_\phi(u)L_\phi(v)^{-1}
    \end{aligned}\]
for all $u,v\in (\mcalD K*H_1(G)_f)^\times$.

Furthermore, we set $\delta_\phi(0)=+\infty$ and $L_\phi(0)=0$. The extended maps satisfy
\[\begin{aligned}
    &\delta_\phi\colon \mcalD G\ra \R\cup\{+\infty\},\ \ && \delta_\phi(z_1z_2)=\delta_\phi(z_1)+\delta_\phi(z_2),\\&L_\phi\colon \mcalD G \ra \mcalD G,&& L_\phi(z_1z_2)=L_\phi(z_1)\cdot L_\phi(z_2)
\end{aligned}\]
 for all $z_1,z_2\in \mcalD G$.
\end{definition}
\begin{proof}
    We first verify well-definedness. Suppose $z\in\mcalD G^\times$ admits two representations 
    $z=u_1v_1^{-1}=u_2v_2^{-1}$, then there exists $w_1,w_2\in (\mcalD K*H_1(G)_f)^\times$ such that $u_1w_1=u_2w_2$ and $v_1w_1=v_2w_2$. Hence
    \[\aligned
        L_\phi(u_1)L_\phi(v_1)^{-1}&=L_\phi(u_1)L_\phi(w_1)L_\phi(w_1)^{-1}L_\phi(v_1)\inv\\&
        =L_\phi(u_1w_1)L_\phi(v_1w_1)^{-1}\\&
        =L_\phi(u_2w_2)L_\phi(v_2w_2)^{-1}\\&
        =L_\phi(u_2)L_\phi(v_2)^{-1}.\endaligned
    \]
    To show that $L_\phi$ is a homomorphism, let $z_1,z_2\in \mcalD G^\times$. By the Ore condition, we may write $z_1=u_1w\inv$ and $z_2=wu_2\inv$ for some $u_1,u_2,w\in (\mcalD K*H_1(G)_f)^\times$. Then
    \[
        L_\phi(z_1)L_\phi(z_2)=L_\phi(u_1)L_\phi(w)\inv\cdot L_\phi(w) L_\phi(u_2)\inv=L_\phi(u_1)L_\phi(u_2)\inv=L_\phi(z_1z_2).
    \]
    The corresponding properties for $\delta_\phi$ can be proved similarly. The properties of the extended maps directly follows.
\end{proof}

Here are some basic facts about the mappings $\delta_\phi,L_\phi$, particularly about their properties under addition in $\mcalD G$. Most properties clearly hold in the subring $\mcalD K*H_1(G)_f$ and it is routine to verify them in its field of fractions $\mcalD G$.

\begin{proposition}\label{Proposition properties of the leading term map}
    Let $G$ be a finitely generated torsion-free group which satisfies the Atiyah Conjecture. Let $\phi\in H^1(G;\R)$ be a real cohomology class and $\delta_\phi\colon \mcalD G\ra \R$, $L_\phi\colon \mcalD G\ra \mcalD G$ be as in Definition \ref{Definition of leading term map}. For any $z,z_1,\ldots,z_n\in \mcalD G$, the following hold:
    \begin{enumerate}[\rm\quad(1)]
        \item $\delta_{r\phi}(z)=r\cdot \delta_\phi(z)$ and $L_{r\phi}(z)=L_\phi(z)$ for all $r\in \R_+$.

        \item $\delta_\phi(cz)=\delta_\phi(z)$ and  $L_\phi(cz)=c\cdot L_\phi(z)$ for all $c\in\Q\setminus\{0\}$.
        
        \item $\delta_\phi(L_\phi(z))=\delta_\phi(z)$ and $L_\phi(L_\phi(z))=L_\phi(z)$.

        \item If $L_\phi(z_1)=L_\phi(z_2)\not=0$, then $\delta_\phi(z_1)=\delta_\phi(z_2)<\delta_\phi(z_1-z_2)$.
        
        \item  $\delta_\phi(z_1+z_2)\geqslant\min\{\delta_\phi(z_1), \delta_\phi(z_2)\}$. If $\delta_\phi(z_1)<\delta_\phi(z_2)$, then $$\delta_\phi(z_1+z_2)=\delta_\phi(z_1),\quad L_\phi(z_1+z_2)= L_\phi(z_1).$$

        \item If $\delta_\phi(z_1)=\cdots=\delta_\phi(z_n)=\colon \delta$ and $\sum_{k=1}^nL_\phi(z_k)\not=0$, then
        \[
            \delta_\phi\bigg(\sum_{k=1}^nz_k\bigg)=\delta,\quad L_\phi\bigg(\sum_{k=1}^nz_k\bigg)=\sum_{k=1}^nL_\phi(z_k).
        \]

        \item For any open neighborhood $U\subset H^1(G;\R)$ of $\phi$, there is a rational cohomology class $\psi\in U$ such that $L_\psi(z)=L_\phi(z)$.

        \item Let $L\subset G$ be a finitely generated subgroup with induced inclusion $\mcalD L\subset \mcalD G$, and let $\phi|_L\colon L\ra \R$ be the restriction of $\phi$ to $L$. Then the mappings
        \[\delta_{\phi|_L}\colon \mcalD L\ra \R\cup\{+\infty\},\quad L_{\phi|_L}\colon \mcalD L \ra\mcalD L\] are precisely the restrictions of $\delta_\phi$ and $L_\phi$ to $\mcalD L$.
    \end{enumerate}
\end{proposition}
\begin{proof}

    For (1)--(3), the statements hold in $\mcalD K*H_1(G)_f$ and directly extends to $\mcalD G$ by definition.

    For (4), write $z_1=u_1w\inv,\ z_2=u_2w\inv$ with $u_1,u_2,w\in \mcalD K*H_1(G)_f$, then $L_\phi(z_1)=L_\phi(z_2)$ implies that $L_\phi(u_1)=L_\phi(u_2)$. It follows that $\delta_\phi(u_1)=\delta_\phi(u_2)<\delta_\phi(u_1-u_2)$. Therefore $\delta_\phi(z_1)=\delta_\phi(z_2)<\delta_\phi(z_1-z_2)$.

    For (5), write $z_1=u_1w\inv,\ z_2=u_2w\inv$ with $u_1,u_2,w\in \mcalD K*H_1(G)_f$. Since $\delta_\phi(u_1+u_2)\geqslant\min\{\delta_\phi(u_1), \delta_\phi(u_2)\}$, we have 
    \begin{align*}
        \delta_\phi(z_1+z_2)&=\delta_\phi(u_1+u_2)-\delta_\phi(w)\\&\geqslant\min\{\delta_\phi(u_1), \delta_\phi(u_2)\}-\delta_\phi(w)=\min\{\delta_\phi(z_1), \delta_\phi(z_2)\}.
    \end{align*} If $\delta_\phi(z_1)<\delta_\phi(z_2)$, then $\delta_\phi(u_1)<\delta_\phi(u_2)$ and $L_\phi(u_1+u_2)=L_\phi(u_1)$. Hence $\delta_\phi(z_1+z_2)=\delta_\phi(z_1)$ and $L_\phi(z_1+z_2)=L_\phi(z_1)$.

    For (6), write $z_i=u_iw\inv$ with $u_i,w\in \mcalD K*H_1(G)_f$ for $i=1,\ldots,n$. By assumption we have $\delta_\phi(u_i)=\delta+\delta_\phi(w)$ and $\sum_{i=1}^kL_\phi(u_i)\not=0$. It follows that $\delta_\phi( \sum_{i=1}^ku_i)=\delta+\delta_\phi(w)$ and $L_\phi(\sum_{i=1}^ku_i)=\sum_{i=1}^kL_\phi(u_i)$. Hence $\delta_\phi( \sum_{i=1}^kz_i)=\delta$ and $L_\phi(\sum_{i=1}^kz_i)=\sum_{i=1}^kL_\phi(z_i)$.

    For (7), write $z=uv\inv$ with $u,v\in \mcalD K*H_1(G)_f$. 
    For another cohomology class $\psi\in H^1(G;\R)$, we have $L_\phi(z)=L_\psi(z)$ if $\phi$ and $\psi$ induce the same ordering on $\operatorname{supp}(u)\cup \operatorname{supp}(v)$, that is:
    \begin{itemize}
        \item for all $h,h'\in\operatorname{supp}(u)\cup \operatorname{supp}(v)$, $\psi(h-h')<0$ whenever $\phi(h-h')<0$;
        \item for all $h,h'\in\operatorname{supp}(u)\cup \operatorname{supp}(v)$, $\psi(h-h')=0$ whenever $\phi(h-h')=0$.
    \end{itemize}
    The domain $\Omega$ of such $\psi$ is the intersection of finitely many closed hyperplanes and open half-spaces of $H^1(G;\R)$, each defined by an integral linear equation. Since $\phi\in\Omega$, any open neighborhood $U\ni \phi$ contains rational classes in $U\cap \Omega$.

     For (8), consider the short exact sequence $1\ra K'\ra L\ra H_1(L)_f\ra 1$. 
    Recall that $\mcalD L$ is the field of fractions of $\mcalD {K'}*H_1(L)_f$ by Proposition \ref{Proposition Linnell field of group extension}. For any nonzero $u\in \mcalD {K'}*H_1(L)_f$, write $u=\sum_{h\in H_1(L)_f}x_h\cdot s(h)$ with $x_h\in \mcalD {K'}$ for a section $s\colon H_1(L)_f\ra L$.
    Let $\delta:=\delta_{\phi|_L}(u)$ and decompose
    \begin{align*}
        u=\sum_{\substack{h\in H_1(L)_f,\\\phi|_L(h)=\delta}}  x_h\cdot s(h)+\sum_{\substack{h\in H_1(L)_f,\\ \phi|_L(h)>\delta}} x_h\cdot s(h)=\colon u_1+u_2.
    \end{align*}
    Then $L_{\phi|_L}(u)=u_1\not=0$. Now identify $\mcalD L$ with its image in $\mcalD G$, we want to show 
    \[
        \delta_\phi(u)=\delta,\quad L_\phi(u)=u_1.
    \]
    Since $\delta_\phi(x_h\cdot s(h))=\phi(h)$ for all $h\in L$, applying (6) to $u_1$ we have    
        $\delta_\phi(u_1)=\delta$ and $L_\phi(u_1)=u_1$.
     By (5) applied to $u_2$ we know that
    \[
        \delta_\phi(u_2)\geqslant \min\{\delta_\phi(x_h\cdot s(h))\mid h\in H_1(L)_f,\ \phi|_L(h)>\delta\}>\delta,
    \]
    Again by (5) applied to $u=u_1+u_2$ we have $\delta_\phi(u)=\delta_\phi(u_1)=\delta$ and $L_\phi(u)=L_\phi(u_1)=u_1$. Hence 
    \[\delta_\phi(u)=\delta_{\phi|_L}(u)\text{\quad and \quad}L_\phi(u)=L_{\phi|_L}(u)
    \] 
    for all nonzero $u\in \mcalD{K'}*H_1(L)_f$. Passing to the field of fractions we conclude that $\delta_\phi(z)=\delta_{\phi|_L}(z)$ and $L_\phi(z)=L_{\phi|_L}(z)$ for all $z\in\mcalD L$.
\end{proof}

\begin{definition}
    An element $z\in \mcalD G$ is called \emph{$\phi$-pure} if $L_\phi(z)=z$.
\end{definition}
\begin{lemma}\label{Lemma Properties of pure elements}
    Here are some properties of the $\phi$-pure elements.
    \begin{enumerate}[\quad\rm (1)]
        \item Elements of $\Z[\ker \phi]\subset \mcalD G$ are $\phi$-pure; elements of $G\subset \mcalD G$ are $\phi$-pure.
        \item The product of two $\phi$-pure elements is $\phi$-pure. If $z_1,\ldots,z_n$ are $\phi$-pure elements with $\delta_\phi(z_1)=\cdots=\delta_\phi(z_n)$, then $\sum_{i=1}^n z_i$ is $\phi$-pure.
        \item For any nonzero $z\in \mcalD G$, the element $L_\phi(z)$ is the unique $\phi$-pure element $w\in \mcalD G$ satisfying $\delta_\phi(z-w)>\delta_\phi(z)$.
    \end{enumerate}
\end{lemma}
\begin{proof}
The properties (1) and the first half of (2) follow from the definition of $L_\phi$. The other half of (2) follows from Proposition \ref{Proposition properties of the leading term map}(6).

For (3), Proposition \ref{Proposition properties of the leading term map}(3)--(4) implies that $L_\phi(z)$ is $\phi$-pure and $\delta_\phi(z-L_\phi(z))>\delta_\phi(z)$. To prove uniqueness, if a $\phi$-pure element $w$ satisfies $\delta_\phi(z-w)>\delta_\phi(z)$, then by Proposition \ref{Proposition properties of the leading term map}(4)--(5),
\[
    w=L_\phi(w)=L_\phi(z-(z-w))=L_\phi(z).
\]

\end{proof}

\begin{remark}
    For any $\phi\in H^1(G;\R)$, the homomorphism $\delta_\phi\colon \mcalD G^\times \ra \R$ induces well-defined homomorphisms $\delta_\phi\colon \Lambda\ra \R$ for $\Lambda=K_1(\mcalD G)$ and $\widetilde K_1(\mcalD G)$. Similarly, the homomorphism $L_\phi\colon \mcalD G^\times \ra \mcalD G^\times$ induces well-defined homomorphisms $L_\phi\colon \Lambda\ra \Lambda$ for $\Lambda=K_1(\mcalD G)$, $\widetilde K_1(\mcalD G)$ and $\operatorname{Wh}(\mcalD G)$.

    We use the same symbols $\delta_\phi,L_\phi$ for their induced maps on $K_1(\mcalD G)$, $\widetilde K_1(\mcalD G)$ and $\operatorname{Wh}(\mcalD G)$; the domain of definition will be clear from the context. The following commutative diagram illustrates the relevant mappings (compare Definition \ref{Definition det, redet, whdet}).
\[\begin{tikzcd}&& \operatorname{GL}(\mcalD G) \arrow[ld, "\det"'] \arrow[d, "\redet"] \arrow[rd, "\whdet"]&\\
{\mcalD G^\times/[\mcalD G^\times,\mcalD G^\times]} \arrow[r, equals, shorten <=3pt, shorten >=3pt] & K_1(\mcalD G) \arrow[r] \arrow["L_\phi"', loop, distance=2em, in=125, out=55] \arrow[d, "\delta_\phi"'] & \widetilde K_1(\mcalD G) \arrow["L_\phi"', loop, distance=2em, in=305, out=235] \arrow[r] \arrow[ld, "\delta_\phi"'{xshift=-2pt, yshift=-4pt}] & \operatorname{Wh}(\mcalD G) \arrow["L_\phi"', loop, distance=2em, in=305, out=235] \\& \mathbb R&&
\end{tikzcd}
\]\end{remark}

The following Theorem \ref{Theorem Leading term of matrices} is of central importance in this paper, relating the leading term map and the Dieudonn\'e determinant.

\begin{theorem}
\label{Theorem Leading term of matrices}
    Let $\phi\in H^1(G;\R)$ be a real cohomology class. Suppose $P$ and $Q$ are $n\times n$ matrices over $\mcalD G$ satisfying the following conditions:
    \begin{itemize}
        \item[(i)] $P$ is invertible over $\mcalD G$.
        \item[(ii)] There exist real numbers $d_1,\ldots,d_n$ and $d_1',\ldots,d_n'$ such that each entry $P_{ij}$ is either zero, or is $\phi$-pure with $\delta_\phi(P_{ij})=d_i-d_j'$.
        \item[(iii)]  $\delta_\phi(Q_{ij})>d_i-d_j'$ for all $i,j$.
    \end{itemize}
    Then $P+Q$ is invertible over $\mcalD G$ and 
        $$L_\phi(\det(P+Q))=\det P\in K_1(\mcalD G).$$
\end{theorem}
\begin{proof}
    We prove by induction on $n$. For the case $n=1$ we have $P,Q\in \mcalD G$. By assumption $P$ is a nonzero $\phi$-pure element, and $\delta_\phi(Q)>\delta_\phi(P)$. Then $L_\phi(\det(P+Q))=\det P$ by Proposition \ref{Proposition properties of the leading term map}(5).

    Now assume the theorem holds for matrices of size $n$. Let $P,Q$ be $(n+1)\times (n+1)$ matrices
    \[
        P=\begin{pmatrix}
            A & U\\
            X & p
        \end{pmatrix},\quad
        Q=\begin{pmatrix}
            B & V\\
            Y & q
        \end{pmatrix}
    \]
    where $P$ is invertible, $p,q\in\mcalD G$, $A,B$ are $n\times n$ matrices over $\mcalD G$ and 
    \[
        \aligned 
        &X=(x_1,\ldots,x_n),\quad Y=(y_1,\ldots,y_n),\\&
        U=(u_1,\ldots,u_n)^T,\quad V=(v_1,\ldots,v_n)^T.\endaligned
    \]
    Without loss of generality, assume $p\not=0$; then $\delta_\phi(q)>\delta_\phi(p)$ and  $p+q\not=0$ by condition (iii). Note that
    \[
        \begin{pmatrix}
            I& -(U+V)(p+q)\inv\\
            0 & 1
        \end{pmatrix}\cdot (P+Q)
        =\begin{pmatrix}
            W & 0\\
            X+Y& p+q
        \end{pmatrix}
    \]
    where 
    $W=A+B-(U+V)(p+q)\inv (X+Y)$ is an $n\times n$ matrix. Examining the expression of $W$, we observe that the terms of lowest $\delta_\phi$-value form the matrix $A-Up\inv X$. We verify this rigorously.

    \begin{claim*}
        Define
        $$W':=A-Up\inv X,\quad W'':=W-W'.$$ Then $W'$ and $W''$ satisfy the three conditions of Theorem \ref{Theorem Leading term of matrices} for size $n$. In particular by the induction hypothesis, $W'$ is invertible over $\mcalD G$ with $\det P=\det W'\cdot \det p$, and $L_\phi\det(W)=\det (W')$.
    \end{claim*}
    Admitting this Claim. We conclude that $W$ is invertible and
    \begin{align*}
        \det(P+Q)&=\det W\cdot\det (p+q),\\
        L_\phi\det (P+Q)&=L_\phi(\det W)\cdot\det p=\det W'\cdot\det p=\det P,
    \end{align*}
    completing the induction. It remains to prove the Claim.
    
    {\noindent\emph{Proof of Claim}.}
    To verify condition (i), note that
    \[
        \begin{pmatrix}
            I& -Up\inv\\
            0 & 1
        \end{pmatrix}\cdot
        P=
        \begin{pmatrix}
            W' & 0\\
            X & p
        \end{pmatrix},
    \]
    so $W'$ is invertible over $\mcalD G$ and $\det P=\det W'\cdot\det p$.

    For condition (ii), note that
    \begin{align*}
        &W_{ij}=A_{ij}+B_{ij}-(u_i+v_i)(p+q)\inv (x_j+y_j),\\&
        W_{ij}'=A_{ij}-u_ip\inv x_j,
        \\&W_{ij}''=B_{ij}+u_ip\inv x_j-(u_i+v_i)(p+q)\inv (x_j+y_j).
    \end{align*}
    If $A_{ij}\not=0$, then $A_{ij}$ is $\phi$-pure with $\delta_\phi(A_{ij})=d_i-d_j'$. If $u_ip\inv x_j\not=0$, then $u_ip\inv x_j$ is also $\phi$-pure with \[
    \delta_\phi(u_ip\inv x_j)=(d_i-d_{n+1}')-(d_{n+1}-d_{n+1}')+(d_{n+1}-d_j')=d_i-d_j'.
    \] Hence, if $W'_{ij}\not=0$ then $W'_{ij}$ is $\phi$-pure with $\delta_\phi(W_{ij}')=d_i-d_j'$ by Lemma \ref{Lemma Properties of pure elements}(2), proving condition (ii).
    
    For condition (iii), we show $$\delta_\phi(u_ip\inv x_j-(u_i+v_i)(p+q)\inv (x_j+y_j))>d_i-d_j'.$$ If $u_ip\inv x_j\not=0$, then $u_i$ and $x_j$ are both nonzero, and
    \begin{align*}
        L_\phi((u_i+v_i)(p+q)\inv (x_j+y_j))=u_ip\inv x_j,\quad
        \delta_\phi(u_ip\inv x_j)=d_i-d_j'
    \end{align*}
    and the inequality follows from Proposition \ref{Proposition properties of the leading term map}(4). If $u_ip\inv x_j=0$ then $u_i=0$ or $x_j=0$, and
    \begin{align*}
        \delta_\phi((u_i+v_i)(p+q)\inv (x_j+y_j))&=\delta_\phi(u_i+v_i)-\delta_\phi(p+q)+\delta_\phi (x_j+y_j)\\&>(d_i-d_{n+1}')-\delta_\phi(p+q)+(d_{n+1}-d_j')\\&=d_i-d_j'.
    \end{align*}
    In both cases, we obtained the desired inequality. Since by assumption $\delta_\phi(B_{ij})>d_i-d_j'$,
    we conclude that $\delta_\phi(W_{ij}'')>d_i-d_j'$ by Proposition \ref{Proposition properties of the leading term map}(5), proving condition (iii).
    \end{proof}

\subsection{The restriction map}\label{Section of restriction map}
Suppose $G$ is a finitely generated torsion-free group satisfying the Atiyah Conjecture, and let $L\triangleleft G$ be a normal subgroup of finite index $d$. In this section we define the restriction map $\operatorname{res}^G_L\colon K_1(\mcalD G)\ra K_1(\mcalD L).$ Recall that $\mcalD G$ is naturally isomorphic to the crossed product $\mcalD L*(G/L)$ by Proposition \ref{Proposition Linnell field of group extension}.

\begin{definition}\label{RestrictionMatrixRepresentation}
    Fix a section $s\colon G/L\ra G$ and let its image be $s(G/L)=\{g_1,\ldots,g_d\}$. Then $G=Lg_1\sqcup\ldots\sqcup Lg_d$. For any element $z\in\mcalD G$ and any $k\in\{1,\ldots,d\}$, there is a unique way to express $ g_k\cdot z$ as
    \[
        g_k\cdot z=\sum_{j=1}^dl_{kj}\cdot g_j, \quad l_{kj}\in \mcalD L.
    \]
    Define $\Lambda_s(z)$ to be the $d\times d$ matrix over $\mcalD L$ whose $(k,j)$-entry is $l_{kj}$. Equivalently, $\Lambda_s(z)$ is the unique matrix over $\mcalD L$ such that
    \[
        \begin{pmatrix}
            g_1\\ \vdots\\ g_d
        \end{pmatrix}\cdot z=\Lambda_s(z)\cdot \begin{pmatrix}
            g_1\\ \vdots\\ g_d
        \end{pmatrix}.
    \]
\end{definition}
   \begin{lemma}\label{Lemma restriction map definition} With the notation of Definition \ref{RestrictionMatrixRepresentation}, the following hold:
    \begin{enumerate}[\rm\quad (1)]
        \item  For any $z_1, z_2\in \mcalD G$, we have $\Lambda_s(z_1z_2)=\Lambda_s(z_1)\cdot \Lambda_s(z_2)$.
        
        \item  If $z\not=0$, then $\Lambda_s(z)$ is invertible over $\mcalD L$.

        \item  If $s'$ is another section, then $\Lambda_{s'}(z)=\Omega\Lambda_s(z)\Omega\inv$ for some invertible matrix $\Omega$ over $\Z L$ depending only on $s$ and $s'$.

        \item Fix $k$ and $\phi\in H^1(G;\R)$, let $\mathcal J$ be the set of indices $j$ such that $\delta_\phi(l_{kj}g_j)$ attains minimum among $j=1,\ldots,d$. Then we have \begin{align*}
            &L_\phi(g_k\cdot z)=\sum_{j\in\mathcal J} L_\phi(l_{kj}\cdot g_j),\\
            &\delta_\phi(g_k\cdot z)=\min_{1\leqslant j\leqslant d}\{\delta_\phi(l_{kj}\cdot g_j)\}.
        \end{align*}
        If in addition $z$ is $\phi$-pure then each $l_{kj}$ is $\phi$-pure. Moreover either $l_{kj}=0$ or $\delta_\phi(g_k\cdot z)=\delta_\phi(l_{kj}\cdot g_j)$.
    \end{enumerate}
\end{lemma}
\begin{proof}
    For (1), observe that 
    \[
        \begin{pmatrix}
            g_1\\ \vdots\\ g_d
        \end{pmatrix}\cdot z_1z_2=\Lambda_s(z_1)\cdot \begin{pmatrix}
            g_1\\ \vdots\\ g_d
        \end{pmatrix}\cdot z_2=\Lambda_s(z_1)\cdot\Lambda_s(z_2)\cdot \begin{pmatrix}
            g_1\\ \vdots\\ g_d
        \end{pmatrix}
    \]
    which implies $\Lambda_s(z_1z_2)=\Lambda_s(z_1)\cdot \Lambda_s(z_2)$. Note that (2) follows directly from (1). 

    For (3), let $s'$ be another section with $s'(L)=\{g_1',\ldots,g_d'\}$ and let $\Omega$ be the $d\times d$ matrix over $\Z L$ such that
    \[
        \begin{pmatrix}
            g_1'\\ \vdots\\ g_d'
        \end{pmatrix}=\Omega\cdot \begin{pmatrix}
            g_1\\ \vdots\\ g_d
        \end{pmatrix}.
    \]
    Then 
    \[
        \begin{pmatrix}
            g_1'\\ \vdots\\ g_d'
        \end{pmatrix}\cdot z=\Omega \cdot \begin{pmatrix}
            g_1\\ \vdots\\ g_d
        \end{pmatrix}\cdot z=\Omega\Lambda_s(z)\cdot \begin{pmatrix}
            g_1\\ \vdots\\ g_d
        \end{pmatrix}=\Omega\Lambda_s(z)\Omega\inv \begin{pmatrix}
            g_1'\\ \vdots\\ g_d'
        \end{pmatrix},
    \]
    so $\Lambda_{s'}(z)=\Omega\Lambda_s(z)\Omega\inv$.

    For (4), we may assume that $z\not=0$, so $l_{kj}\not=0$ for any $j\in \mathcal J$.
     Let $w_1=\sum_{j\in \mathcal J}L_\phi(l_{kj})g_j$. Then $w_1\not=0$ since $\mcalD G=\oplus_{j}\mcalD L\cdot g_j$, and $w_1$ is $\phi$-pure by Lemma \ref{Lemma Properties of pure elements}(2). Now write
    \begin{align*}
        g_k\cdot z&=\sum_{j\in \mathcal J}L_\phi(l_{kj})g_j +\sum_{j\in \mathcal J}(l_{kj}-L_\phi(l_{kj}))g_j +\sum_{j\notin \mathcal J}l_{kj}g_j\\
        &=:w_1+w_2+w_3.
    \end{align*}
    Clearly $\delta_\phi(w_2)>\delta_\phi(w_1)$ by Proposition \ref{Proposition properties of the leading term map}(4), and $\delta_\phi(w_3)>\delta_\phi(w_1)$ by the choice of $\mathcal J$. We conclude that $L_\phi(g_k\cdot z)=w_1$ by Lemma \ref{Lemma Properties of pure elements}(3). In particular, $$\delta_\phi(g_k\cdot z)=\min\{\delta_\phi(l_{kj}\cdot g_j)\mid j=1,\ldots,n\}.$$
    Now suppose additionally that $z$ is $\phi$-pure, then $L_\phi(g_k\cdot z)=g_k\cdot z$, that is 
    \[
        \sum_{j\in \mathcal J}L_\phi(l_{kj})g_j=\sum_{j=1}^nl_{kj}g_j.
    \]
    By the direct sum decomposition $\mcalD G=\oplus_j\mcalD L\cdot g_j$, it follows that $l_{kj}$ is pure with $\delta_\phi(l_{kj}g_j)=\delta_\phi(g_k\cdot z)$ for any $j\in \mathcal J$, and $l_{kj}=0$ for any $j\notin \mathcal J$.
\end{proof}

\begin{definition}[Restriction map]
    Let $L\triangleleft G$ be a normal subgroup of finite index. Choose a section $s\colon G/L\ra G$. Define the \emph{restriction map}
    \[
        \operatorname{res}^G_L\colon \mcalD G^\times\ra K_1(\mcalD L),\quad  z\mapsto\det (\Lambda_s(z))\in K_1(\mcalD L).
    \]
    By Lemma \ref{Lemma restriction map definition}, this is a group homomorphism independent of the choice of section $s$.
    We use the same notation for the induced homomorphism on the $K_1$-group $$\operatorname{res}^G_L\colon K_1(\mcalD G)\ra K_1(\mcalD L).$$
\end{definition}

\begin{remark}

    For $g\in G\subset \mcalD G^\times$, the matrix $\Lambda_s(g)$ is a permutation matrix whose nonzero entries are elements $l_i\in L$. It follows that $$\operatorname{res}^G_L(g)=\det (\Lambda_s(g))=\pm\prod_i [l_i]\in K_1(\mcalD L)$$ is represented by a element of $\pm L$. Therefore the restriction map naturally induces a homomorphism on Whitehead groups
    \[
        \operatorname{res}^G_L\colon \operatorname{Wh}(\mcalD G)\ra \operatorname{Wh}(\mcalD L).
    \]
\end{remark}

The following important result establishes that the restriction map commutes with the leading term map.

\begin{theorem}\label{Theorem leading term map commutes with restrictions}
    Let $G$ be a finitely generated torsion-free group satisfying the Atiyah Conjecture, and let $L\triangleleft G$ be a normal subgroup of finite index. Let $\phi\in H^1(G;\R)$ and denote by $\phi|_L\in H^1(L;\R)$ its restriction to $L$. Then for any $z\in \mcalD G^\times$, we have
    \[
        L_{\phi|_L}(\operatorname{res}^G_L(z))=\operatorname{res}^G_L(L_\phi(z))\in K_1(\mcalD L).
    \]
\end{theorem}
\begin{proof}
    Write $z=L_\phi(z)+z'$. Fix section $s:G/L\ra G$ and write $G=Lg_1\sqcup\ldots\sqcup Lg_d$. Consider $d\times d$ matrices $P:=\Lambda_s(L_\phi(z))$ and $Q:=\Lambda_s(z')$ over $\mcalD L$. Explicitly,
    \[
        g_k\cdot L_\phi(z)=\sum_{j=1}^dP_{kj}\cdot g_j,\quad g_k\cdot z'=\sum_{j=1}^dQ_{kj}\cdot g_j.
    \]
    Then $\operatorname{res}^G_L(z)=\det (P+Q)$ and $\operatorname{res}^G_L(L_\phi(z))=\det P$. Since $L_\phi(z)$ is $\phi$-pure, Lemma \ref{Lemma restriction map definition}(4) implies that $P_{kj}$ is $\phi$-pure, and that 
    \begin{align*}
        &\delta_\phi(P_{kj})=\delta_\phi(z)+\delta_\phi(g_k)-\delta_\phi(g_j)\quad \text{for }P_{kj}\not=0,\\& 
        \delta_\phi(Q_{kj})\geqslant\delta_\phi(z')+\delta_\phi(g_k)-\delta_\phi(g_j) >  \delta_\phi(z)+\delta_\phi(g_k)-\delta_\phi(g_j).
    \end{align*}
    By Proposition \ref{Proposition properties of the leading term map}(8) the maps $\delta_\phi$ and $L_{\phi}$ restrict to $\delta_{\phi|_L}$ and $L_{\phi|_L}$ on $\mcalD L$. We now apply Theorem \ref{Theorem Leading term of matrices} to $P$ and $Q$ over $\mcalD L$, taking 
    \[
    d_i=\delta_\phi(z)+\delta_\phi(g_i) \text{\quad and\quad }d_i'=\delta_\phi(g_i)\text{\quad for }i=1,\ldots,d.\]
    The three conditions of the Theorem are satisfied, and we conclude
    \[
        L_{\phi|_L}(\det (P+Q))=\det P\in K_1(\mcalD L),
    \]
    that is $L_{\phi|_L}(\operatorname{res}^G_L(z))=\operatorname{res}^G_L(L_\phi(z))$, completing the proof.
\end{proof}

\subsection{The polytope map}
Let $H$ be a finitely generated free abelian group. Note that $H_1(H;\R)=\R\otimes_\Z H$ is a finite-dimensional real vector space.
    A \emph{polytope} in $H_1(H;\R)$ is a compact set obtained as the convex hull of a finite subset of points. We allow the empty set $\emptyset$ to be a polytope.

    \begin{definition}[Faces of polytopes]
        Given a polytope $P$ and a character $\phi\in H^1(H;\R)$. Define $\delta_\phi(P):=\inf_{x\in P}\phi(x)$ and the \emph{face associated to $\phi$} by
    \begin{align*}
        F_\phi(P):=\{x\in P\mid \phi(x)=\delta_\phi(P)\}.
    \end{align*}
    Clearly $F_\phi(P)$ is a polytope contained in $P$, and the collection $\{F_\phi(P)\mid \phi\in H^1(H;\R)\}$ consists of all faces of $P$. A face is called a \emph{vertex} if it is a single point. Any polytope is the convex hull of its vertices. A polytope is called \emph{integral} if all its vertices lie in the integral lattice $H\subset H_1(H;\R)$.
    \end{definition}

    Given any two non-empty polytopes $P_1,P_2$ in $H_1(H;\R)$, their \emph{Minkowski sum} is defined to be the polytope
    \[
        P_1+P_2:=\{p_1+p_2\mid p_1\in P_1,\ p_2\in P_2\}.
    \]
    This is the convex hull of the set $\{v_1+v_2\mid v_i \text{ is a vertex of }P_i,\ i=1,2\}$. The operator $\delta_\phi$ and the face map $F_\phi$ are additive under the Minkowski sum:
    \[
        \delta_\phi(P_1+P_2)=\delta_\phi(P_1)+\delta_\phi(P_2),\quad F_\phi(P_1+P_2)=F_\phi(P_1)+F_\phi(P_2)
    \]
    for every character $\phi$ and all polytopes $P_1,P_2$.

    \begin{example}
        Let $M$ be an admissible $3$-manifold. 
        The \emph{Thurston norm ball} $$B_x(M):=\{\phi \in H^1(M;\R)\mid x_M(\phi)\leqslant 1\}$$ is a (possibly non-compact) polyhedron in $H^1(M;\R)$; the \emph{dual Thurston norm ball} is defined as $$B_x^*(M):=\{z\in H_1(M;\R)\mid \phi(z)\leqslant 1\text{ for all }\phi \in B_x(M)\}.$$ Thurston showed that $B_x^*(M)$ is an integral polytope in $H_1(M;\R)$ with vertices $\pm v_1,\ldots,\pm v_k$, and that the Thurston norm ball is determined by these vertices: 
        \[
            B_x(M)=\{\phi\in H^1(M;\R)\mid |\phi(v_i)|\leqslant 1,\ i=1,\ldots,k\}.
        \]
        A \emph{Thurston cone} in $H^1(M;\R)$ is either an open cone formed by the origin and a face of $B_x(M)$, or a maximal connected component of $H^1(M;\R)\setminus\{0\}$ on which the Thurston norm $x_M$ vanishes. It follows that $H^1(M;\R)\setminus\{0\}$ is the disjoint union of all  Thurston cones of various dimensions. A Thurston cone is called \emph{top-dimensional} if its dimension equals $\dim H^1(M;\R)$. The following Lemma \ref{Lemma Top dimensional cone iff face map vertex} is a reformulation of Thurston's theorem.
        \begin{lemma}\label{Lemma Top dimensional cone iff face map vertex}
            A nonzero character $\phi\in H^1(M;\R)$ lies in a top-dimensional Thurston cone if and only if $F_\phi B_x^*(M)$ is a vertex.
        \end{lemma}
    \end{example}
    \begin{definition}[$\mathcal P_\Z(H)$ and $\poly^{\rm Wh}(H)$]
        The \emph{integral polytope group} $\mathcal P_\Z(H)$ is defined as the Grothendieck group of integral polytopes in $H_1(H;\R)$ under the Minkowski sum. More precisely, $\mathcal P_\Z$ is the abelian group generated by symbols $[P]$ where $P$ is a non-empty integral polytope in $H_1(H;\R)$, subject to the relation $[P]+[Q]=[P+Q]$ for each pair of non-empty integral polytopes $P,Q$.
        
        Every element $h\in H$ determines an one-point polytope $[h]$ in $\poly(H)$, defining an embedding of $H$ into $\poly(H)$. The \emph{Whitehead polytope group} is defined as the quotient \[\whpoly(H)=\poly(H)/H.\] In other words, two polytopes are identified in $\whpoly(H)$ if and only if they differ by translation by an element in the lattice $H$.
    \end{definition}

    We make the following remarks:

    \begin{enumerate}[\quad(1)]
        \item  Any element of $\poly(H)$ can be written as a formal difference $[P]-[Q]$ for non-empty integral polytopes $P,Q\subset H_1(H;\R)$. Two such expressions $[P_1]-[Q_1]$ and $[P_2]-[Q_2]$ represents the same element if and only if 
    $P_1+Q_2=P_2+Q_1$. Equivalently, two elements $x,y\in\poly (H)$ are distinct if and only if there exists $\phi\colon H\ra \Z$ such that their images under the induced map $\poly(H)\ra \poly(\Z)$ is different. This observation can be used to show that both $\mathcal P_\Z(H)$ and $\poly^{\rm Wh}(H)$ are free abelian groups \cite[Lemma 4.8]{friedl2017universal}.

    \item The face map $F_\phi$ extends naturally to a homomorphism on the polytope group:
    \begin{align*}
        F_\phi\colon \poly(H)\ra \poly(H),\quad F_\phi([P]-[Q])=[F_\phi(P)]-[F_\phi(Q)].
    \end{align*}
    Since $F_\phi$ preserves the subgroup $H$, it descends to a homomorphism (denoted by the same symbol) $$F_\phi\colon \whpoly(H)\ra \whpoly(H).$$
    \end{enumerate}

    \subsubsection{Polytope homomorphism $\mathbb P$}
    Let $G$ be a finitely generated torsion-free group satisfying the Atiyah Conjecture. Let $H$ be the free abelianization of $G$, then we have the short exact sequence
    \[
        1\ra K\ra G\ra H\ra 1
    \]
    and by Proposition \ref{Proposition Linnell field of group extension} $\mcalD G$ is the field of fractions of the subring $\mcalD K*H$. Recall from Definition \ref{Definition support of a element in DG} that for any section $s\colon H\ra G$, every $u\in \mcalD K*H$ has a unique expression
    \[
        u=\sum_{h\in H} x_h\cdot s(h),\quad x_h\in \mcalD K,
    \]
    and its support $\operatorname{supp}(u):=\{h\in H\mid x_h\not=0\}$ is independent of the choice of $s$.
    \begin{definition}[Polytope map]
        For any $u\in (\mcalD K*H)^\times$, define $\mathbb P(u)\in \poly(H)$ to be the element represented by the convex hull of $\operatorname{supp}(u)$ in $\R\otimes_\Z H$. The \emph{polytope homomorphism} is defined as
    \[
        \mathbb P\colon \mcalD G^\times\ra \poly(H),\quad \mathbb  P(uv\inv):=\mathbb P(u)-\mathbb P(v)
    \]
    for all $u,v\in(\mcalD K*H)^\times$.
    \end{definition}
    
    It is shown in \cite[Lemma 6.4]{Friedl2019l2} that $\mathbb P$ is a well-defined homomorphism and 
    \[
        \mathbb P(z_1z_2)=\mathbb P(z_1)+\mathbb P(z_2)
    \]
    for all $z_1,z_2\in \mcalD G^\times$.
    
    The following commutative diagram is immediate from the definition.
    \[
    \begin{tikzcd}
\mcalD G^\times \arrow[r, "L_\phi"] \arrow[d, "\mathbb P"] & \mcalD G^\times \arrow[d, "\mathbb P"] \\
\poly(H) \arrow[r, "F_\phi"]                & \poly(H)  .             
\end{tikzcd}
    \]
    Recall that $K_1(\mcalD G)$ is the abelianization of $\mcalD G^\times$ and $\operatorname{Wh}(\mcalD G)=K_1(\mcalD G)/[\pm G]$. The polytope homomorphism naturally induces (by abuse of notation)
    \[
         \mathbb P\colon \operatorname{Wh(\mcalD G)}\ra \poly^{\rm Wh} (H).
    \]We then obtain the commutative diagram for the induced homomorphisms
    \[
    \begin{tikzcd}
\operatorname{Wh}(\mcalD G) \arrow[r, "L_\phi"] \arrow[d, "\mathbb P"] & \operatorname{Wh}(\mcalD G) \arrow[d, "\mathbb P"] \\
\whpoly(H) \arrow[r, "F_\phi"]                & \whpoly(H)  .             
\end{tikzcd}
    \]

We now state a key result relating the universal $L^2$-torsion of an admissible 3-manifold to its dual Thurston norm ball. 
    \begin{theorem}[{\cite[Theorem 4.37]{friedl2017universal}}]\label{Theorem universal L2 torsion gives dual Thurston norm ball}
        Let $M$ be an admissible 3-manifold which is not homeomorphic to $S^1\times D^2$. Then $\tautwo_u(M)\not=0$, and
        \[
            [B_x^*(M)]=2\cdot  \mathbb P(\tautwo_u(M))\in \poly^{\rm Wh}(H_1(M)_f),
        \]
        where $B_x^*(M)\subset H_1(M;\R)$ is the dual Thurston norm ball, and $\tautwo_u(M)\in \operatorname{Wh}(\mcalD {\pi_1(M)})$ is the universal $L^2$-torsion of $M$.
    \end{theorem}

    \subsection{Half of Theorem \ref{Main Theorem fibered iff face map zero}}
    We now prove the ``if" part of Theorem \ref{Main Theorem fibered iff face map zero}. A stronger version of the ``only if" part will be stated and proved in Theorem \ref{Theorem only if part} after establishing Theorem \ref{Main Theorem taut decomposition and face map}.
    \begin{theorem}[The ``if" part of Theorem \ref{Main Theorem fibered iff face map zero}]\label{Theorem if part}
        Suppose $M$ is an admissible 3-manifold such that $M$ is not a closed graph manifold without an NPC metric. If $\phi\in H^1(M;\R)$ is a nonzero class such that $L_\phi\tautwo_u(M)=1\in \operatorname{Wh}(\mcalD{\pi_1(M)})$, then $\phi$ is fibered.
    \end{theorem}
    
    \begin{proof}

        If $M$ is homeomorphic to a solid torus then any nonzero class is fibered. Assume therefore that $M$ is not a solid torus. The assumption on $M$ implies that $\pi_1(M)$ is virtually special by a combination of work (c.f. \cite[Section 4.7]{Aschenbrenner20153ManifoldGroups} and references therein). In particular, by Agol's criterion for virtual fibering \cite{agol2008criteria} there exists a regular finite covering $\overline M\ra M$ such that the pull back class $\bar \phi$ lies in the closure of a fibered cone. Let $G$ be the fundamental group of $M$ and $\pi_1(\overline M)=\colon L<G$. By the restriction property of Theorem \ref{Theorem properties of the torsion of cw complexes}(4),
        \[
            \tautwo_u(\overline M)=\operatorname{res} ^G_L\tautwo_u(M)\in \operatorname{Wh}(\mcalD L).
        \]
        Applying $L_{\bar\phi}$ to both sides and using Theorem \ref{Theorem leading term map commutes with restrictions}, we obtain 
        \begin{align*}
            L_{\bar \phi}(\tautwo_u(\overline M))&=L_{\bar \phi}(\operatorname{res} ^G_L\tautwo_u(M))\\&=\operatorname{res} ^G_L(L_\phi\tautwo_u(M)) =1 \in \operatorname{Wh}(\mcalD L).
        \end{align*}
        Now apply the polytope map $\mathbb P:\operatorname{Wh}(\mcalD L)\ra \whpoly (H_1(L)_f)$,
        \begin{align*}
            0&=\mathbb P(L_{\bar \phi}(\tautwo_u(\overline M)))\\&=F_{\bar \phi}\mathbb P(\tautwo_u(\overline M))\text{\quad by commutativity}\\&=2\cdot F_{\bar \phi}[B_x^*(\overline M)]\text{\quad by Theorem \ref{Theorem universal L2 torsion gives dual Thurston norm ball}}.
        \end{align*}
        Since $\whpoly(H_1(\overline M;\R))$ is torsion-free by \cite[Lemma 4.8]{friedl2017universal}, it follows that \[
            F_{\bar \phi}[B_x^*(\overline M)]=0,
        \]
        therefore $F_{\bar\phi}B_x^*(\overline M)$ is a vertex. By Lemma \ref{Lemma Top dimensional cone iff face map vertex} $\bar \phi$ lies in a top dimensional Thurston cone of $H^1(\overline M;\R)$. Since $\bar \phi$ was chose to lie in the closure of a fibered cone, and the boundary of a fibered cone consists of Thurston cones of strictly lower dimensions, $\bar\phi$ must lie in a fibered cone itself. Therefore $\bar\phi$ is a fibered class for $\overline M$ and consequently $\phi$ is a fibered class for $M$.
    \end{proof}

\section{Universal $L^2$-torsion for taut sutured manifolds}\label{Section taut sutured manifolds}
In this section, we first briefly recall the terminologies of sutured manifold theory, then discuss the universal $L^2$-torsion of a taut sutured manifold and prove the decomposition formula Theorem \ref{Main Theorem taut decomposition and face map}, which is then used to complete the proof of Theorem \ref{Main Theorem fibered iff face map zero}. After that, we use a doubling trick to prove Theorem \ref{Main Theorem universal L2 torsion detect product sutured manifold}.

\subsection{Background on sutured manifold theory} Throughout this section $N$ will be an arbitrary compact, oriented 3-manifold.

\begin{definition}[Taut surfaces]
    Given a compact orientable surface $\Sigma$ with path-components $\Sigma_1,\ldots,\Sigma_k$ we define its \emph{complexity} as 
\[
    \chi_-(\Sigma):=\sum_{i=1}^k \max\{0,-\chi(\Sigma_i)\}.
\]
A properly embedded oriented surface $\Sigma$ in $N$ is \emph{taut} if $\Sigma$ is incompressible, and has minimal complexity among all properly embedded oriented surfaces representing the homology class $[\Sigma,\partial \Sigma]\in H_2(N,\nu (\partial \Sigma);\Z)$, where $\nu (\partial \Sigma)$ is a regular neighborhood of $\partial\Sigma$ in $\partial N$.
\end{definition}

\begin{definition}[Sutured manifold]
    A \emph{sutured manifold} $(N,R_+,R_-,\gamma)$ is a compact oriented 3-manifold $N$ with a decomposition of its boundary into two subsurfaces $R_+$ and $R_-$ along their common boundary $\gamma$ which we call the \emph{suture}.
The orientation on $R_\pm$ is defined in the way that the normal vector of $R_+$ points out of $N$ and the normal vector of $R_-$ points inward of $N$. The boundary orientations of $R_\pm$ on $\gamma$ coincide and induce an orientation of the suture $\gamma$. We often abbreviate $(N,R_+,R_-,\gamma)$ as $(N,\gamma)$ where no confusions arises.

A sutured manifold $(N,R_+,R_-,\gamma)$ is called \emph{taut} if $N$ is irreducible and $R_\pm$ are both taut surfaces (viewed as properly embedded surfaces after pushing slightly into $N$). 
\end{definition}

Before introducing sutured manifold decompositions, we establish some notations. Let $S$ be a (possibly disconnected) properly embedded oriented surface in $N$. Choose a product neighborhood $S\times(-1,1)$ of $S$ in $N$ and denote by $N\bb S:= N\setminus S\times (-1,1)$ the complement. Let $S_+$ (resp. $S_-$) be the components of $S\times \{-1\}\cup S\times\{1\}$ in $N\bb S$ whose normal vector points out of (resp. into) $N'$. This notation ensures that $(N\bb S,S_+,S_-,\emptyset)$ forms a sutured manifold when $N$ is closed. 

\begin{definition}[Decomposition surface]
     Let $(N,R_+,R_-,\gamma)$ be a sutured manifold. A properly embedded surface $S$ in $N$ is called a \emph{decomposition surface} if $\partial S$ is transverse to $\gamma$, no component of $\partial S$ bounds a disk in $R_\pm$ and no component of $S$ is a disk $D$ with $\partial D\subset R_\pm$.
\end{definition}

\begin{definition}[Sutured manifold decomposition]
    Given an oriented decomposition surface $S$ for $(N,R_+,R_-,\gamma)$, the \emph{sutured manifold decomposition}
\[
    (N,R_+,R_-,\gamma)\stackrel{S}\rightsquigarrow (N', R_+',R_-',\gamma')
\]
is defined as:
\begin{align*}
    N'&=N\bb S,\\
    R_+'&=(R_+ \cup S_+)\cap N',\\
    R_-'&=(R_- \cup S_-)\cap N',\\
    \gamma'&=\partial R_+'=\partial R_-'.
\end{align*}
See Figure \ref{fig:Separating decomposition}--\ref{fig:Non-separating sutured decomposition} for illustrations. A sutured manifold decomposition $(N,\gamma)\stackrel{S}\rightsquigarrow (N',\gamma')$ is called \emph{taut} if $(N',\gamma')$ is taut. Gabai \cite[Lemma 0.4]{gabai1987foliationsII} proved that if $(N,\gamma)\stackrel{S}\rightsquigarrow (N',\gamma')$ is a taut sutured decomposition then $(N,\gamma)$ is also taut.
\end{definition}

We make the following observations:
\begin{enumerate}[\quad (1)]
    \item Following \cite{AgolDunfield2019Certifying}, we consider sutures as simple closed curves and exclude torus sutures, unlike many other sources where the suture are disjoint union of annuli and tori. In fact, if a sutured manifold $(N,\gamma)$ contains toral sutures, these can be absorbed into $R_\pm$ without affecting the universal $L^2$-torsion: 
    \begin{lemma}
        If $(N,\gamma)$ is a taut sutured manifold and $T\subset \gamma$ is a torus component. Then $\tautwo_u(N,R_+)=\tautwo_u(N,R_+\cup T)$.
    \end{lemma}
    \begin{proof}
        This follows from the Sum Formula \ref{Theorem properties of the torsion of cw complexes} and the fact that $\tautwo_u(T)=1$.
    \end{proof}
    Accordingly, the definition of sutured manifold decomposition is slightly modified. In the classical definition (c.f. \cite{Gabai1983Foliations,gabai1987foliationsII}), a decomposition surface may include boundary components that are cores of a sutured annulus. Such boundary components can be isotoped to entirely contained in $R_+\cup R_-$, yielding isomorphic sutured manifolds after decomposition.

    \item Let $(N,R_+,R_-,\gamma)\stackrel{S}\rightsquigarrow (N',R_+',R_-',\gamma')$ be a taut sutured decomposition, then $S$ is incompressible in $N$. To see this, note that $R_+'$ is formed by gluing the surfaces $S_+$ and $R_+\cap N'$ along $S_+\cap R_+$, which consists of arcs and boundary circles of $S_+$. By assumption, no boundary circles of $S_+$ bound a disk in $R_+$ or in $S_+$, so the closed curves of $S_+\cap R_+$ are homotopically nontrivial in both $S_+$ and $R_+\cap N'$. By Van Kampen Theorem the surface $S_+$ is $\pi_1$-injective in $R_+'$, hence $\pi_1$-injective in $N'$ since $R_+'$ is incompressible in $N'$. This shows that $S$ admits no compressing disk in $N$.
    
    \item Given a taut sutured decomposition $(N,\gamma)\stackrel{S}\rightsquigarrow (N',\gamma')$. Since $S$ is incompressible in $N$, the inclusion of each component of $N'$ into $N$ induces a monomorphism on the fundamental group.
\end{enumerate}

\subsection{Universal $L^2$-torsion for taut sutured 3-manifolds}
\begin{theorem}[\cite{herrmann2023sutured}]\label{Theorem Taut iff L2acyclic}
    Let $(N,\gamma)$ be a sutured 3-manifold with infinite fundamental group. Suppose that $N$ is irreducible and $R_\pm$ are both incompressible. Then $(N,\gamma)$ is taut if and only if the pair $(N,R_+)$ is $L^2$-acyclic.
\end{theorem}
\begin{proof}
    Assume first that $(N,\gamma)$ is taut. Since $N$ has infinite fundamental group and is irreducible, no components of $R_\pm$ are disks or spheres (otherwise $N$ would be a 3-ball, contradicting $\pi_1(N)$ infinite). Therefore the complexity of $R_\pm$ equals $-\chi(R_\pm)$, and tautness implies $\chi(R_+)=\chi(R_-)$. We apply \cite[Theorem 1.1]{herrmann2023sutured} to conclude that $(N,R_+)$ is $L^2$-acyclic.

    Conversely, if $(N,R_+)$ is $L^2$-acyclic, then the Euler characteristic $\chi(N,R_+)=\chi(N)-\chi(R_+)$ is zero. Since $\chi(N)=\frac12\chi(\partial N)=\frac12(\chi(R_+)+\chi(R_-))$, it follows that $\chi(R_+)=\chi(R_-)$. Another application of \cite[Theorem 1.1]{herrmann2023sutured} shows that $(N,\gamma)$ is taut.
\end{proof}

The reader might wonder why we prefer $(N,R_+)$ than $(N,R_-)$. In fact, these two pairs are dual to each other by the following Proposition \ref{Proposition Rpm dual formula}. We prefer the pair $(N,R_+)$ since it better suits our convention of orientations.
\begin{proposition}\label{Proposition Rpm dual formula}
    Let $(N,R_+,R_-,\gamma)$ be a sutured manifold. Suppose the pair $(N,R_+)$ is $L^2$-acyclic. Then $(N,R_-)$ is also $L^2$-acyclic. Moreover 
    \[
        \tautwo_u(N,R_+)=(\tautwo_u(N,R_-))^*.
    \]
\end{proposition}
\begin{proof}
    The proof adapts the argument in \cite[Page 139]{milnor1962duality} and does not make essential use of $L^2$-theory. Choose a smooth triangulation of $N$ such that $R_+$ and $R_-$ are subcomplexes. Let $\widehat N$ be the universal cover and let $\widehat R_\pm$ be the preimages of $R_\pm$. Consider the dual cellular complex $\widehat N'$ of $\widehat N$. Each cell of $\widehat N$ is canonically dual to a cell of $\widehat N'\setminus \partial \widehat N'$, as illustrated in Figure \ref{fig:Triangulation} below.
    \begin{figure}[htbp]
        \centering
        
\def\svgwidth{.95\columnwidth}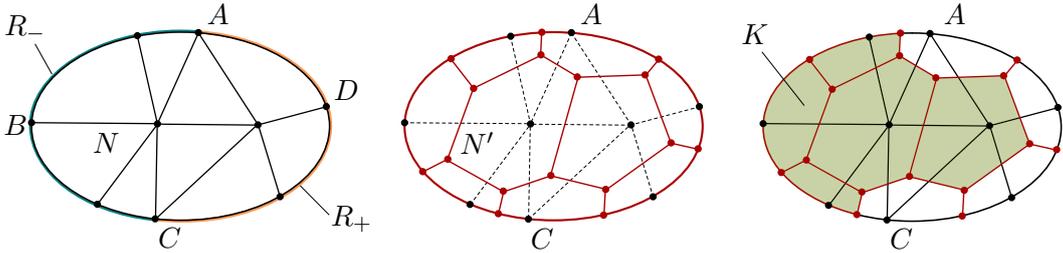

        \caption{An illustration in one lower dimension. Left: A simplicial complex $N$ whose boundary is a union of a subcomplex $R_-$ (the arc $ \wideparen{ABC}$) and a subcomplex $R_+$ (the arc $\wideparen{ADC}$). Middle: the dual cellular complex $N'$ (red). Right: the subcomplex $K\subset N'$ consisting of cells disjoint from $R_+$. The complex $N'$ deformation retracts to $K$ along a product neighborhood of $R_+$.
        }
        \label{fig:Triangulation}
    \end{figure}
    
    More explicitly, define the intersection pairing 
    \[
        p\colon C_*(\widehat N)\times C_{3-*}(\widehat N',\partial\widehat N')\ra \Z G,\quad p(\sigma,\sigma')= \sum_{g\in G}\langle\sigma,g\sigma'\rangle\cdot g
    \]where $\langle\sigma,g\sigma'\rangle$ is the intersection number of $\sigma$ and $g\sigma'$.
    One can verify the identities:
    \begin{align*}
         &p(g\sigma,\sigma')=g\cdot p(\sigma,\sigma'),\\&
         p(\sigma,g\sigma')= p(\sigma,\sigma')\cdot g\inv,\\&
         p(\partial \sigma,\sigma')=\pm p(\sigma,\partial\sigma')
    \end{align*}
    for all $g\in G$. The pairing is non-degenerate in the sense that $\sigma\mapsto p(\sigma,*)$
    gives an isomorphism of $\Z G$-chain complexes $C_*(\widehat N)\cong C^{3-*}(\widehat N',\partial N')$. A cell of $\widehat R_+$ is canonically dual to a cell of $\widehat N'\setminus \partial\widehat  N'$ whose closure intersects $\widehat R_+$ non-trivially, and vise versa. Let $K$ be subcomplex of $N'$ consisting of cells whose closure are disjoint from $R_+$, and let $\widehat K$ be its preimage in the universal cover. Note that $R_-'=K\cap \partial N'$, and we have the following induced non-degenerate pairing:
    \[
        C_*(\widehat N,\widehat R_+)\times C_{3-*}(\widehat K,\widehat {R_-'})\ra\Z G,
    \]
    and hence an isomorphism of $\Z G$-chain complexes $C_*(\widehat N,\widehat R_+)\cong C^{3-*}(\widehat K,\widehat {R_-'})$. By Proposition \ref{Proposition torsion of dual chain complex} the finite based free $\Z G$-chain complex $C_*(\widehat K,\widehat {R_-'})$ is $L^2$-acyclic and $$\tautwo_u(C_*(\widehat N,\widehat R_+))=(\tautwo_u(C_{*}(\widehat K,\widehat {R_-'})))^*.$$
    The pair $(K,R_-')$ is a deformation retract of the pair $(N',R_-')$ (after barycentric-subdividing $N$ if necessary). Therefore $$\tautwo_u(N,R_+)=(\tautwo_u(K,R_-'))^*=(\tautwo_u(N',R_-'))^*.$$ This shows that the pair $(N,R_-)$ is $L^2$-acyclic with $\tautwo_u(N,R_+)=(\tautwo_u(N,R_-))^*$, as the universal $L^2$-torsion of $(N,R_-)$ is independent of its cellular structure.
\end{proof}

\subsection{Turaev's algorithm}
Let $(N,\gamma)\stackrel{S}\rightsquigarrow (N',\gamma')$ be a taut sutured decomposition. To simplify notation, we abbreviate the pair $(N',R_+')=(N\bb S,(S_+\cup R_+)\cap (N\bb S))$ as $(N\bb S,S_+\cup R_+)$, with the understanding that a CW-pair $(X,Y)$ always means $(X,X\cap Y)$.

Turaev's algorithm \cite{Turaev2002Homological} is a procedure that modifies $S$ into a non-separating surface $S'$ without changing the push-forward of the universal $L^2$-torsion $i_*\tautwo_u(N',R_+')$.

\begin{lemma}\label{Lemma L2 torsion invariant under separating decomposition}
    Let $(N,\gamma)$ be a taut sutured manifold. Suppose there is a decomposition surface $S$ such that:
    \begin{enumerate}[\rm \quad (1)]
        \item The sutured manifold decomposition $(N,R_+,R_-,\gamma)\stackrel{S}\rightsquigarrow (N',R_+',R_-',\gamma')$ is taut.
        \item $S$ is separating and $N'$ is a disjoint union of two (possibly disconnected) manifolds $N_0$ and $N_1$.
        \item $S_-\subset N_0$ and $S_+\subset N_1$.
    \end{enumerate}
    Then $\tautwo_u(N,R_+)=i_*\tautwo_u(N',R_+')$ where $i\colon N'\hookrightarrow N$ is the inclusion.
\end{lemma}
\begin{figure}[htbp]
    \centering
    
\def\svgwidth{.9\columnwidth}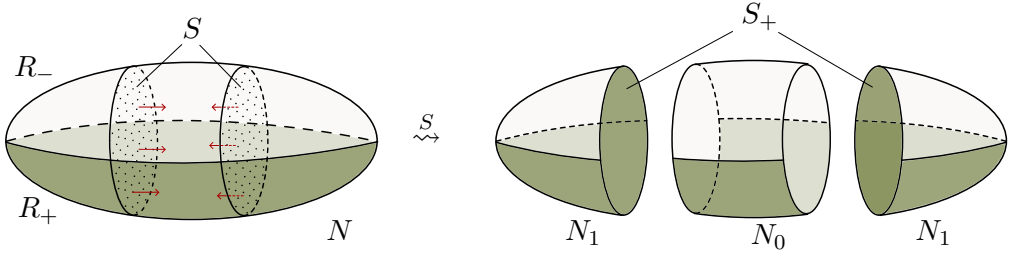

    \caption{Decomposing $N$ along a separating decomposition surface $S$ (dotted region in left figure, with normal direction indicated by red arrows). The $R_+$-regions of the sutured manifolds are shown in green.}
    \label{fig:Separating decomposition}
\end{figure}
\begin{proof}
    Consider the short exact sequence of chain complexes
    \[
        0\ra C_*(\widehat{N_0}, \widehat{R_+})\ra C_*(\widehat N,\widehat{R_+})\ra C_*(\widehat N,\widehat{N_0\cup R_+})\ra 0
    \]
    and a natural isomorphism $C_*(\widehat N,\widehat{N_0\cup R_+})=C_*(\widehat{N_1},\widehat{S_+\cup R_+})$. Note that $N_0\cap R_+$ is the $R_+$-region of $N_0$ and $(S_+\cup R_+)\cap N_1$ is the $R_+$-region of $N_1$. Since both $N_0$ and $N_1$ are taut sutured manifolds by assumption,  Theorem \ref{Theorem Taut iff L2acyclic} implies that all three chain complexes are $L^2$-acyclic, and by the Sum Formula (Proposition \ref{Proposition Axiom of torsion of chain complexes}),
    \[
        \tautwo_u(N,R_+)=(i_0)_*\tautwo_u(N_0,R_+)\cdot (i_1)_*\tautwo_u(N_1,S_+\cup R_+)
    \]
    where $i_0:N_0\hookrightarrow N$ and $i_1:N_1\hookrightarrow N$ are the inclusions. The right-hand side equals $i_*\tautwo_u(N',R_+')$, completing the proof.    
\end{proof}

\begin{definition}[Weighted surface]
    A \emph{weighted surface} $\widehat S$ in a compact oriented 3-manifold $N$ is a finite collection of pairs $(S_i,w_i)$, $i=1,\ldots,n$, where each $S_i$ is a connected properly embedded oriented surface, each $w_i$ is a positive integer, and the surfaces $S_i$ are pairwise disjoint. The \emph{realization} of $\widehat S$ is the properly embedded oriented surface
\[
    \bar S:=\bigcup_{i=1}^n w_i\cdot S_i
\]
where $w_i\cdot S_i$ denotes the union of $w_i$ parallel copies of $S_i$. The \emph{reduction} of $\widehat S$ is the  surface
\[
    S:=\bigcup_{i=1}^n S_i.
\]
A weighted surface is called a \emph{weighted decomposition surface} if its realization is a decomposition surface.
\end{definition}

If $\widehat S$ is a weighted decomposition surface, then $N\bb \bar S$ is a disjoint union of $N\bb S$ and some copies of $S_i\times I$. Consequently, the sutured decomposition along $\bar S$ is taut if and only if the sutured decomposition along $S$ is taut.

The following Proposition \ref{Proposition Turaev method for multisurfaces} is a generalization of \cite[Proposition 3.3]{ben2022leading}.

\begin{proposition}\label{Proposition Turaev method for multisurfaces}
    Let $(N,R_+,R_-,\gamma)$ be a taut sutured manifold and let $(N,\gamma)\stackrel{\Sigma} \leadsto (N',\gamma')$ be a taut sutured decomposition. Then there exists a weighted decomposition surface $\widehat S$ in $N$ (with realization $\bar S$ and reduction $S$) such that 
    \begin{enumerate}[\rm \quad (1)]
        \item $N\bb S$ is connected,
        \item $[\bar S]=[\Sigma]\in H_2(N,\partial N;\Z)$,
        \item the sutured decomposition of $N$ along $S$ is taut, 
        \item $i_*\tautwo_u(N\bb S,S_+\cup R_+)=j_*\tautwo_u(N\bb \Sigma,\Sigma_+\cup R_+)$ where $i,j$ are the natural inclusions of $N\bb S$ and $N\bb \Sigma$ into $N$.
    \end{enumerate}
\end{proposition}

\begin{proof}
    For any weighted surface $\widehat S$ in $N$, define $c(\widehat S):=\#\pi_0(N\bb S)$. Begin by taking $\widehat S$ to be the surface $\Sigma$ with weight $1$ assigned to each component, then $\bar S=S=\Sigma$, and $\widehat S$ satisfies (2)--(4). It suffices to show that given any weighted surface $\widehat S$ with $c(\widehat S)>1$ satisfying (2)--(4), then there exists a weighted surface $\widehat T$ satisfying (2)--(4) with $c(\widehat T)<c(\widehat S)$.

    Given such $\widehat S$. Since $c(\widehat S)>1$ there is a component $C\subset S$ such that $C_\pm$ lie in different components of $N\bb S$. Choose $C$ to be such a component with minimal weight $w$. Let $M_0$ (resp. $M_1$) be the component of $N\bb S$ containing $C_-$ (resp. $C_+$). Let $C_1=C,C_2, \ldots,C_k$ be the components of $S$ whose normal direction point out of $M_1$ and let $D_1,\ldots,D_l$ be the components of $S$ whose normal direction point into $M_1$ (it is possible that $C_i=D_j$ for some $i,j$, meaning both sides of $C_i$ both belong to $M_1$). Then
    \[
        [C_1]+\cdots+[C_k]=[D_1]+\cdots+[D_l]\in H_2(N,\partial N;\Z).
    \]
    We change the weights of $\widehat S$ by increasing the weights of $D_1,\ldots,D_l$ by $w$ and decreasing the weights of $C_1,\ldots,C_k$ by $w$. If a component's weight becomes zero, we discard this component. Denote the new weighted surface by $\widehat T$. Clearly, $[\bar S]=[\bar T]\in H_2(N,\partial N;\Z)$. Since $T$ is a subcollection of $S$ and the decomposition along $S$ is taut, the decomposition along $T$ is also taut. Thus, (2) and (3) hold for $\widehat T$. Moreover, $c(\widehat T)<c(\widehat S)$ because the component $C$ is discarded and $M_0$ and $M_1$ lie in the same component of $N\bb T$.

    It remains to check (4). Let $S_0=S\setminus T$. Then $N\bb S$ is obtained from the sutured manifold decomposition of $N\bb T$ along $S_0$. By construction, $S_0$ separates $N\bb T$; in particular, it separates $M_0$ from $M_1$ and $(S_0)_+\subset M_1$. Apply Lemma \ref{Lemma L2 torsion invariant under separating decomposition} to $N_0:=M_0$ and $N_1:=M_1$ with ambient space $N\bb T$, we have $$\tautwo_u(N\bb T,T_+\cup R_+)=i'_*\tautwo_u(N\bb S,S_+\cup R_+)$$ where $i':N\bb S\hookrightarrow N\bb T$ is the inclusion. It follows that
    \[
    (i'')_*\tautwo_u(N\bb T,T_+\cup R_+)=i_*\tautwo_u(N\bb S,S_+\cup R_+)=j_*\tautwo_u(N\bb \Sigma,\Sigma_+\cup R_+)
    \]
    where $i'',i=i''\circ i',j$ are the inclusions of $N\bb T$, $N\bb S$ and $N\bb \Sigma$ into $N$, respectively. This verifies (4) for $\widehat T$ and completes the proof.
\end{proof}

\subsection{Decomposition formula for universal $L^2$-torsion}

In this section we prove Theorem \ref{Main Theorem taut decomposition and face map}.

\begin{definition}[Leading term of matrices]
    Let $G$ be a finitely generated torsion-free group satisfying Atiyah Conjecture. Let $\phi\in H^1(G;\R)$ be an 1-cohomology class. There are natural multiplicative maps
    \[
        \delta_\phi\colon (\Z G)^\times \ra \R,\quad L_\phi\colon (\Z G)^\times \ra (\Z G)^\times
    \]
    defined by considering the terms with minimal $\phi$-value (see Section \ref{Section leading term map}). For a nonzero matrix $A$ over $\Z G$, let $\delta_\phi(A)$ be the minimum of $\delta_\phi(A_{ij})$ over all nonzero entries $A_{ij}$. Then $A$ admits a unique decomposition
    \[
        A=L_\phi(A)+(A-L_\phi(A))
    \]
    where every group element $g$ appearing in $L_\phi(A)$ satisfies $\phi(g)=\delta_\phi(A)$, and any group element $h$ in $(A-L_\phi(A))$ satisfies $\phi(h)> \delta_\phi(A)$.

    We define $L_\phi(A)=0$ if $A$ is a zero matrix. Otherwise $L_\phi(A)$ is always nonzero.    \end{definition}
    \begin{remark}\label{Remark special case of leading term}
        By Theorem \ref{Theorem Leading term of matrices} with $d_i=\delta_\phi(A)$ and $d_i'=0$, if $L_\phi(A)$ is invertible over $\mcalD G$, then $A$ is also invertible over $\mcalD G$ and $L_\phi(\det A)=\det (L_\phi(A))\in K_1(\mcalD G)$.
    \end{remark}
     
    \begin{definition}[Leading term of chain complexes]
    For any finite based free $\Z G$-chain complex
    \[
        C_*=(0\stackrel{}\longrightarrow C_n\stackrel{A_n}\longrightarrow \cdots\stackrel{A_2}\longrightarrow C_1\stackrel{A_1}\longrightarrow C_0\stackrel{}\longrightarrow 0),
    \]
    define 
    \[
        L_\phi(C_*)=(0\xrightarrow{} C_n\xrightarrow{L_\phi(A_n)}\cdots\xrightarrow{L_\phi(A_2)} C_1\xrightarrow{L_\phi(A_1)} C_0\stackrel{}\longrightarrow 0).
    \]
    One can verify that if $AB=0$, then  $L_\phi(A)L_\phi(B)=0$. In particular $L_\phi(C_*)$ is a well-defined finite based free $\Z G$-chain complex.
\end{definition}

\begin{lemma}\label{Lemma Leading Coefficient of Chain Complexes}
    Suppose $G$ is a finitely generated torsion-free group satisfying the Atiyah Conjecture. Let $\phi\in H^1(G;\R)$ be an 1-cohomology class and let
    \[
        C_*=(0\stackrel{}\longrightarrow C_n\stackrel{A_n}\longrightarrow \cdots\stackrel{A_2}\longrightarrow C_1\stackrel{A_1}\longrightarrow C_0\stackrel{}\longrightarrow 0)
    \] be a finite based free $\Z G$-chain complex.
    If $L_\phi (C_*)$ is $L^2$-acyclic, then $C_*$ is also $L^2$-acyclic and 
    \[
        \tautwo_u(L_\phi(C_*))=L_\phi(\tautwo_u(C_*))\in\widetilde{K_1}(\mcalD G).
    \]
\end{lemma}
\begin{proof}
    Since $L_\phi(C_*)$ is $L^2$-acyclic, there exists a non-degenerate matrix chain $\mathcal A$ of $L_\phi(C_*)$ by Theorem \ref{Matrix chain for acyclic chain complex}. Let $B_1,\ldots,B_n$ be the associated submatrices of $L_\phi(A_1),\ldots, L_\phi(A_n)$, then each $B_i$ is a weak isomorphism, and
    \[
        \tautwo_u(L_\phi(C_*))=\prod_{i=1}^n\redet(B_i)^{(-1)^n}\in\widetilde K_1(\mcalD G)
    \]
    Let $C_1,\ldots,C_n$ be the submatrices of $A_1,\ldots,A_n$ associated  to the same matrix chain $\mathcal A$. Since $B_i\not=0$, we have $L_\phi(C_i)=B_i$. By Theorem \ref{Theorem Leading term of matrices} or Remark \ref{Remark special case of leading term} above, each $C_i$ is a weak isomorphism and
    \[
        L_\phi\redet(C_i)=\redet (B_i)\in\widetilde K_1(\mcalD G).
    \]
    Therefore $\mathcal A$ is a non-degenerate matrix chain of $C_*$. Applying Theorem \ref{Matrix chain for acyclic chain complex} again,
    \[
        \tautwo_u(C_*)=\prod_{i=1}^n\redet (C_i)^{(-1)^n}\in\widetilde K_1(\mcalD G)
    \]
    and consequently $\tautwo_u(L_\phi(C_*))=L_\phi(\tautwo_u(C_*))$.
\end{proof}

\begin{theorem}[Theorem \ref{Main Theorem taut decomposition and face map}]\label{Theorem main theorem decomposition formula}
    Let $(N,R_+,R_-,\gamma)\stackrel{\Sigma}\leadsto (N',R_+',R_-',\gamma')$ be a taut sutured decomposition and let $\phi\in H^1(N;\Z)$ be the Poincar\'e dual of the decomposition surface $\Sigma$, then
    \[
        j_*\tautwo_u(N',R_+')=L_\phi\tautwo_u(N,R_+)
    \]
    where $j\colon N'\hookrightarrow N$ is the natural inclusion.
\end{theorem}
\begin{proof}
    By Proposition \ref{Proposition Turaev method for multisurfaces}, there is a weighted decomposition surface $\widehat S$ in $N$ such that $N\bb S$ is connected, $\phi=PD([\bar S,\partial \bar S])\in H^1(N;\Z)$, and $$j_*\tautwo_u(N',R_+')=i_*\tautwo_u(N\bb S,S_+\cup R_+)$$ where $i\colon N\bb S\ra N$ is the inclusion. We are left to show that 
    \[
        L_\phi(\tautwo_u(N,R_+) )=i_*\tautwo_u(N\bb S,S_+\cup R_+).
    \]
    
    Chose a CW-structure for $N$ such that $S\times I$ and $R_\pm$ are subcomplexes. Fix a base point $p\in N\setminus (S\times I)$. For each cell $\sigma$ in the CW-structure of $N$, choose a path $\gamma_\sigma$ (shown in red in Figure \ref{fig:Non-separating sutured decomposition}) connecting $p$ and $\sigma$ such that
    \begin{itemize}
        \item $\gamma_\sigma$ is disjoint with $S_-$ if $\sigma \subset N\setminus S_-$,
        \item $\gamma_\sigma$ is disjoint with $S_+$ if $\sigma\subset S_-$.
    \end{itemize}
    Lift the base point $p$ to $\hat p$ in the universal cover $\widehat N$ and lift each cell $\sigma$ to $\hat \sigma$ using the path $\gamma_\sigma$. The cells $\hat \sigma$ form a basis for the finite based free $\Z[\pi_1(N)]$-chain complex $C_*(\widehat N)$.
    \begin{figure}[htbp]
        \centering
        
\def\svgwidth{.95\columnwidth}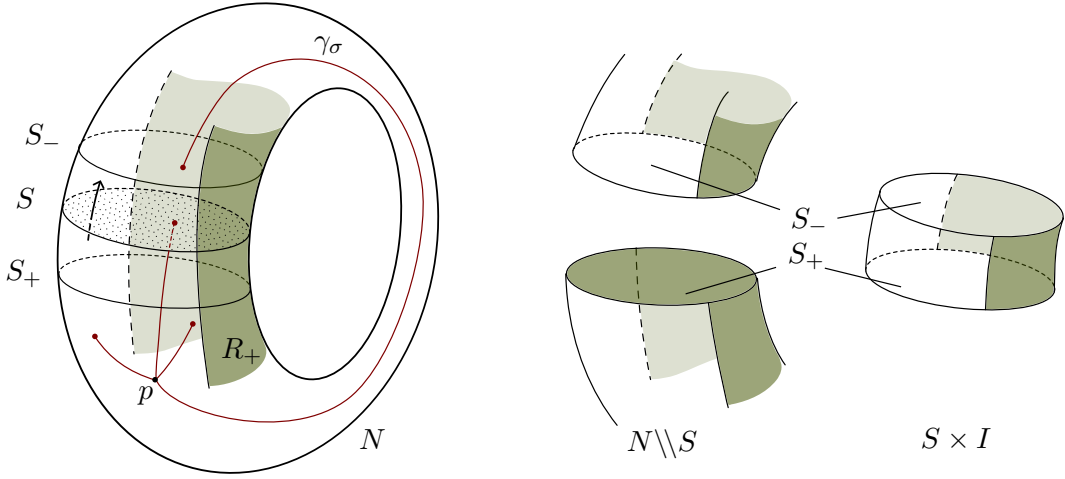

        \caption{Consider $N$ as the union of $N\bb S$ and $S\times I$.}
        \label{fig:Non-separating sutured decomposition}
    \end{figure}
    Now consider the finite based chain complex $C_*(\widehat N,\widehat{R_+})$, for each $k=0,1,2,3$ the chain module admits the following direct sum decomposition:
    \[
        C_k(\widehat N,\widehat {R_+})=C_k(\widehat {N\bb S}, \widehat {S_+\cup R_+})\oplus C_k(\widehat {S\times I},\widehat {S_-\cup R_+}).
    \]
    Accordingly, the boundary homomorphism $\partial _k\colon C_k(\widehat N,\widehat {R_+})\ra C_{k-1}(\widehat N,\widehat {R_+})$ admits the following decomposition
    \begin{align*}
        &\partial _k=\begin{pmatrix}
            \partial_k^1 &\partial ^2_k\\
            \partial _k^3 & \partial _k^4
        \end{pmatrix},\\
        &\partial _k^1\colon C_k(\widehat {N\bb S}, \widehat {S_+\cup R_+})\ra C_{k-1}(\widehat {N\bb S}, \widehat {S_+\cup R_+}),\\
        &\partial _k^2\colon  C_k(\widehat {N\bb S}, \widehat {S_+\cup R_+})\ra C_{k-1}(\widehat {S\times I}, \widehat {S_-\cup R_+}),\\&
        \partial _k^3\colon C_k(\widehat {S\times I}, \widehat {S_-\cup R_+})\ra C_{k-1}(\widehat {N\bb S}, \widehat {S_+\cup R_+}),\\&
        \partial_k^4\colon C_k(\widehat {S\times I}, \widehat {S_-\cup R_+})\ra C_{k-1}(\widehat {S\times I}, \widehat {S_-\cup R_+}).
    \end{align*}
    In particular, if $\hat \sigma$ is a basis element of $C_k(\widehat {S\times I}, \widehat {S_-\cup R_+})$ then there are basis elements $\tau_i$ (allowing repetitions) of $C_{k-1}(\widehat {N\bb S}, \widehat {S_+\cup R_+})$ and $g_i\in \pi_1(N)$ such that
    \[
        \partial^3_k(\hat\sigma)=\sum_i g_i \hat \tau_i.
    \]
    In fact, these cells $\hat\tau_i$ must lie in $\widehat S_-$. By our choice of the liftings, each $g_i$ satisfies $\phi(g_i)>0$. A similar argument shows that any group element $h$ appearing in $\partial_k^1,\partial_k^2$ or $\partial_k^4$ satisfies $\phi(h)=0$. It follows that $$L_\phi(\partial_*)=\begin{pmatrix}
        \partial_*^1 & \partial_*^2\\
        0 & \partial_*^4
    \end{pmatrix}$$
    and we obtain a short exact sequence of chain complexes
    \[
        0\ra C_*(\widehat {N\bb S},\widehat {S_+\cup R_+})\ra L_\phi(C_*(\widehat N,\widehat {R_+}))\ra  C_*(\widehat {S\times I},\widehat {S_-\cup R_+})\ra 0.
    \]
    By the product formula \ref{Proposition Axiom of torsion of chain complexes} and the induction property Theorem \ref{Theorem properties of the torsion of cw complexes}, we have 
    \[
        \tautwo_u(L_\phi(C_*(\widehat N,\widehat {R_+})))=i_*\tautwo_u(N\bb S,S_-\cup R_+)\cdot i_*'\tautwo_u(S\times I,S_-\cup R_+),
    \]
    where $i:N\bb S\hookrightarrow N$ and $i':S\times I\hookrightarrow N$ are the inclusions. The left-hand side equals $L_\phi (\tautwo_u(N,R_+))$ by Lemma \ref{Lemma Leading Coefficient of Chain Complexes}. On the right-hand side, since $(S_-\cup R_+)\cap (S\times I)$ deformation retracts onto $S_-$, we have $\tautwo_u(S\times I,S_-\cup R_+)=\tautwo_u(S\times I,S_-)=1$. 
    Therefore, $L_\phi (\tautwo_u(N,R_+))=i_*\tautwo_u(N\bb S,S_+\cup R_+)$, completing the proof.
\end{proof}

The (stronger) ``only if" part of Theorem \ref{Main Theorem fibered iff face map zero} follows easily. 
\begin{theorem}[The ``only if" part of Theorem \ref{Main Theorem fibered iff face map zero}]\label{Theorem only if part}
        Suppose $M$ is an admissible 3-manifold and $\phi\in H^1(M;\R)$ is a fibered class. Then $L_\phi\tautwo_u(M)=1\in \operatorname{Wh}(\mcalD{G})$.
    \end{theorem}
    \begin{proof}
        By Proposition \ref{Proposition properties of the leading term map}(7), there exists a rational fibered class $\psi\in H^1(M;\Q)$ such that $L_\phi(\tautwo_u(M))=L_\psi(\tautwo_u(M))$. Choose a positive integer $n$ such that $n\psi$ is an integral fibered class, and let $S$ be a fiber surface dual to $n\psi$, then $M\bb S=S\times [0,1]$ is a product. By Theorem \ref{Main Theorem taut decomposition and face map} we have
        \[
            L_{n\psi}\tautwo_u(M)=j_*\tautwo_u(M\bb S,S_+)=1,
        \]
        where $j\colon M\bb S\hookrightarrow M$ is the inclusion. It follows that $L_\phi \tautwo_u(M)=L_{n\psi}\tautwo_u(M)=1$.
    \end{proof}
    Combined with Theorem \ref{Theorem if part}, this completes the proof of Theorem \ref{Main Theorem fibered iff face map zero}.

    \subsection{Proof of Theorem \ref{Main Theorem universal L2 torsion detect product sutured manifold}}

 We use a refined doubling trick to show that universal $L^2$-torsion detects product sutured manifolds.

    \begin{definition}[Double of taut sutured manifolds]
        Let $(N,R_+,R_-,\gamma)$ be a taut sutured manifold and let $f\colon R_+\ra R_+$ be an orientation-preserving homeomorphism. Let $(N,\overline R_+,\overline R_-,\bar\gamma)$ be the sutured manifold with the same underlying oriented manifold as $N$, but with $R_+$ and $R_-$ interchanged (see Figure \ref{fig:Double of sutured manifold} below). That is, 
        \[
            \bar\gamma=-\gamma,\quad \overline R_+=-R_-,\quad  \overline R_-=-R_+.
        \]
        The \emph{double of $N$ with monodromy $f$} is defined to be the admissible 3-manifold $$DN_f=(N,\gamma)\cup (N,\bar\gamma)/\sim$$ formed by identifying $R_-$ and $\overline R_+$ via the identity map and identifying $R_+$ and $\overline R_-$ via $f$.
    \end{definition}
    \begin{figure}[htbp]
        \centering
        
\def\svgwidth{.8\columnwidth}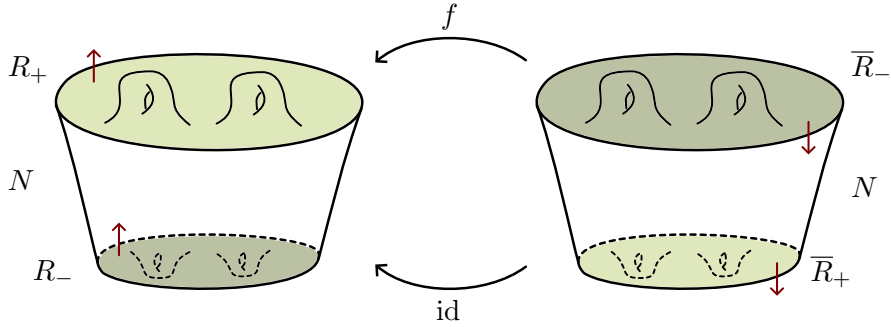

        \caption{Double of a sutured manifold $N$ with monodromy $f$}
        \label{fig:Double of sutured manifold}
    \end{figure}
    \begin{lemma}\label{Lemma double is not closed graph}
        If $(N,R_+, R_-,\gamma)$ is a taut sutured manifold with $R_+$ and $R_-$ both non-empty. If not every component of $R_\pm$ is a torus, then there exists $f\colon R_+\ra R_+$ such that $DN_f$ is not a closed graph manifold.
    \end{lemma}
    \begin{proof}
        If the suture $\gamma$ is non-empty, then $DN_f$ has non-empty torus boundary for any $f$. If $\pi_1(N)$ is finite then $N$ is a product 3-ball and has non-empty sutured annuli.

        Now assume $N$ has infinite fundamental group and $\gamma=\emptyset$. Then each component of $R_\pm$ is a closed surfaces with positive genus. By hypothesis, $R_\pm$ contains a component of genus at least 2. Since $\chi(R_+)=\chi(R_-)$, the surface $R_+$ must contain such a component and we call it $\Sigma$. Define $$f=(\phi,\operatorname{id}),\quad \phi\colon \Sigma\ra \Sigma,\quad \operatorname{id}\colon R_+\setminus \Sigma\ra R_+\setminus \Sigma$$
        where $\phi$ is to be specified later. 
        Let $M$ be the manifold obtained by gluing together $(N,\gamma)$ and $(N,\bar \gamma)$ via the identity maps on $R_-\ra \overline R_+$ and on $R_+\setminus \Sigma\ra \overline R_-\setminus (-\Sigma)$. Then $\partial M$ consists of two copies of $\Sigma$ (denoted $\Sigma_1,\Sigma_2$) that remain unglued.  The double $DN_f$ is obtained by identifying $\Sigma_1$ and $\Sigma_2$ via $\phi$.
        
        It remains to prove the following Claim: 
        \begin{claim*}
            Let $M$ be a compact, orientable, irreducible 3-manifold with $\partial M=\Sigma_1\sqcup \Sigma_2$, where $\Sigma_1,\Sigma_2$ are connected incompressible surfaces of genus $g\geqslant2$. Then there exists a homeomorphism $\phi\colon \Sigma_1\ra \Sigma_2$ such that the closed manifold $M_\phi$, obtained by gluing $\Sigma_1$ to $\Sigma_2$ via $\phi$, is not a graph manifold.
        \end{claim*}
        
        \noindent\emph{Proof of Claim.}        
        If $M=\Sigma_1\times I$, then a pseudo-Anosov homeomorphism $\phi$ suffices by Thurston's Hyperbolization Theorem. Now suppose $M$ is not a product.
        
        By the Characteristic Submanifold Theorem \cite[Chapter V]{Jaco1979SeifertFibered}, there exists a submanifold $X\subset M$ that is a disjoint union of Seifert fibered spaces and $I$-bundles over surfaces, such that any incompressible torus and annulus of $M$ can be homotoped into $X$. The intersections $X\cap \Sigma_1$ and $X\cap \Sigma_2$ are incompressible \emph{proper} subsurfaces of $\Sigma_1$ and $\Sigma_2$ respectively; otherwise, if (say) $\Sigma_1\subset X$, then $\Sigma_1$ is contained in an $I$-bundle component $P$ of $X$. Since $\partial P\subset\partial M$, we must have $P=M$, contradicting the assumption that $M$ is not a product.
        
        Choose an arbitrary component of $\partial (X\cap \Sigma_i)$ and denote it by $C_i$. Then $C_i$ is an essential simple closed curve on $\Sigma_i$, $i=1,2$. Choose a homeomorphism $\phi\colon \Sigma_1\ra \Sigma_2$ such that the curves $\phi(C_1)$ and $C_2$ fill $\Sigma_2$. This can be achieved by taking $\phi=\psi^n\circ\phi_0$, where $\phi_0\colon \Sigma_1\ra \Sigma_2$ is any homeomorphism, $\psi\colon \Sigma_2\ra \Sigma_2$ is pseudo-Anosov and $n$ sufficiently large. Then the distance between $[\phi(C_1)]$ and $[C_2]$ in the curve complex of $\Sigma_2$ becomes arbitrarily large \cite[Proposition 4.6]{masur1999geometry}, implying that $\phi(C_1)$ and $C_2$ fill $\Sigma_2$.

        We now show that $M_\phi$ is not a graph manifold. Suppose $T\subset M_\phi$ is an essential torus. After isotopy we assume that $T$ is transverse to $\Sigma=\Sigma_1=\Sigma_2$ in $M_\phi$, and any component of $T\cap \Sigma$ (if nonempty) is an essential simple closed curves on $T$. If $T\cap \Sigma\not=\emptyset$ and let $C$ be such a intersection curve, then $T\setminus C$ is an essential annulus in $M$, so the image of $C$ on $\Sigma_2\subset \partial M$ can be homotoped into both $X\cap \Sigma_2$ and $\phi(X\cap \Sigma_1)$. Hence the geometric intersection numbers $i(C,\phi(C_1))$ and $i(C,C_2)$ are both zero, contradicting the fact that $\phi(C_1)$ and $C_2$ fill $\Sigma_2$.
        
        Thus, any essential torus $T$ in $M_\phi$ can be isotoped to be disjoint with $\Sigma$. Let $Y$ be the JSJ-piece of $M_\phi$ containing $\Sigma$. If $Y$ were Seifert fibered, then $\Sigma$ would be a horizontal surface in $Y$ since it is essential with genus $\geqslant 2$. The circle fiber over any essential simple closed curves of $\Sigma$ would yield essential tori in $Y$ intersecting $\Sigma$, a contradiction. Therefore $Y$ is hyperbolic and $M_\phi$ is not a closed graph manifold.
    \end{proof}

    \begin{theorem}[Theorem \ref{Main Theorem universal L2 torsion detect product sutured manifold}]
        Let $(N,R_+,R_-,\gamma)$ be a taut sutured manifold with $R_+$ and $R_-$ both non-empty. Then $(N,\gamma)$ is a product sutured manifold if and only if $\tautwo_u(N,R_+)=1\in \operatorname{Wh}(\mcalD{\pi_1(N)})$.
    \end{theorem}
    \begin{proof}
        If $(N,\gamma)$ is a product sutured manifold then $\tautwo_u(N,R_+)=1$ by Proposition \ref{Proposition computation example}.

        Now suppose $\tautwo_u(N,R_+)=1$ and we show that $N$ is a product.

        \noindent\textbf{Case 1:} Each component of $R_\pm$ is a torus.    
        Then $N$ is an admissible 3-manifold and  $\tautwo_u(N)=\tautwo_u(N,R_+)=1$. By Theorem \ref{Main Theorem fibered iff face map zero}, every nonzero cohomology class is fibered, and its Thurston norm vanishes by Theorem \ref{Theorem universal L2 torsion gives dual Thurston norm ball}. Hence $N$ a disk or annulus bundle over the circle and is homeomorphic to one of the following: the solid torus $D^2\times S^1$, the thickened torus $T^2\times I$, or the twisted $I$-bundle over Klein bottle $K\widetilde\times I$. Since $R_+$ and $R_-$ are both non-empty, it follows that $$(N,R_+,R_-,\gamma)=(T^2\times I,T^2\times\{1\},T^2\times \{0\},\emptyset)$$ which is a product sutured manifold.
        
        \noindent\textbf{Case 2:} Not every component of $R_\pm$ is a torus. By Lemma \ref{Lemma double is not closed graph} there is a homeomorphism $f\colon R_+\ra R_+$ such that the double $DN_f$ is not a closed graph manifold. Then $R_+\cup R_-$ is a Thurston norm minimizing surface dual to a cohomology class $\phi$. By Theorem \ref{Main Theorem taut decomposition and face map},
        \[
            L_\phi\tautwo_u(DN_f)=(j_1)_*\tautwo_u(N,R_+)\cdot(j_2)_*\tautwo_u(N,\overline R_+)
        \]
        where $j_1,j_2$ are the natural inclusions. By Proposition \ref{Proposition Rpm dual formula}, $$\tautwo_u(N,R_+)=\tautwo_u(N,R_-)=\tautwo_u(N,\overline R_+)=1,$$ so $L_\phi\tautwo_u(DN_f)=1$. By Theorem \ref{Main Theorem fibered iff face map zero}, $\phi$ is a fibered class. Therefore, $R_+\cup R_-$ is a fiber surface and $(N,\gamma)$ is a product sutured manifold.
    \end{proof}

\section{Applications of the universal $L^2$-torsions}\label{Section Applications}
\subsection{Universal $L^2$-torsion for group homomorphisms}
Let $G_1,G_2$ be torsion-free groups satisfying the Atiyah Conjecture. Assume that they have finite classifying spaces $X_i$ and their Whitehead groups $\operatorname{Wh}(G_i)$ vanish for $i=1,2$. Then any homomorphism $\varphi\colon G_1\ra G_2$ determines a homotopy class of mapping $\Phi\colon X_1\ra X_2$. We call $\Phi$ a \emph{realization of $\varphi$}.

\begin{definition}[Universal $L^2$-torsion for group homomorphisms]\label{Definition of universal l2 torsion for free group endomorphism}
Let $\varphi\colon G_1\ra G_2$ be a homomorphism and let $\Phi\colon X_1\ra X_2$ be its realization. The \emph{universal $L^2$-torsion of $\varphi$} is defined to be $\tautwo_u(\varphi):=\tautwo_u(\Phi)\in\operatorname{Wh}(\mcalD {G_2})\sqcup \{0\}$. We say $\varphi$ is an \emph{$L^2$-homology equivalence} if $\tautwo_u(\varphi)\not=0$.
\end{definition}
If $X_1', X_2'$ are different choices of classifying spaces then we have the following commutative diagram
\[
    \begin{tikzcd}
X_1 \arrow[r, "\Phi"] \arrow[d, "\psi_1"'] & X_2 \arrow[d, "\psi_2"] \\
X_1' \arrow[r, "\Phi'"]   & X_2'             
\end{tikzcd}
    \]
where $\psi_i\colon X_i\ra X_i'$ are homotopy equivalences (hence simple-homotopy equivalences since the Whitehead groups vanish) and $\Phi'$ is the realization of $\varphi$ with respect to the classifying spaces $X_1',X_2'$. By Lemma \ref{Lemma simple homotopy invariance of mappings} we know that
\[
    \tautwo_u(\Phi')=(\psi_2)_*\tautwo_u(\Phi),
\]
hence the definition of $\tautwo_u(\varphi)$ does not depend on the choice of classifying spaces $X_i$.

\subsubsection{Homomorphism between free groups}
Let $F_1$ and $F_2$ be finitely generated free groups. They satisfy the Atiyah Conjecture and have trivial Whitehead groups \cite{stallings1965whitehead}.
We explicitly calculate $\tautwo_u(\varphi)$ for a homomorphism $\varphi\colon F_1\ra F_2$ under given free basis $F_1=\langle x_1,\ldots,x_n\rangle$ and $F_2=\langle y_1,\ldots,y_m\rangle$. 

Let $X_1=\vee_{i=1}^n S^1$ and $X_2=\vee_{i=1}^m S^1$ be the wedge of circles. Endow $X_1$ with the usual CW-structure with one 0-cell $p$ and $n$ 1-cells $e_1,\ldots,e_n$. Identify the fundamental group $\pi_1(X_1,p)$ with $F_1$ in such a way that $x_i=[e_i]$. Similarly, $X_2$ is given the CW-structure with one 0-cell $q$ and $m$ 1-cells $f_1,\ldots,f_m$ and $y_i=[f_i]$.

Let $\Phi\colon X_1\ra X_2$ be a cellular realization of $\varphi$. Form the wedge space $X_1\vee X_2$ by identifying $p\in X_1$ and $q\in X_2$. Attach a 2-cell $\sigma_i$ whose boundary is the concatenation of $\Phi(e_i)$ and $e_i\inv$ for each $i=1,\ldots,n$. The resulting CW-complex $M_\Phi$ is simple-homotopy equivalent to the mapping cylinder of $\Phi$. Let $\widehat M_\Phi$ be the universal cover of $M_\Phi$ and let $\widehat X_1$ be the preimage of $X_1\subset M_\Phi$. Fix a lifting $\hat p\in \widehat M_\Phi$ and lift the other cells with respect to the base point $\hat p$. Then we have the following $\Z F_2$-chain complex
\begin{equation}\label{Equation chain complex of mapping cylinder}
    C_*(\widehat M_\Phi,\widehat {X_1})=(0\ra \Z F_2\langle \hat \sigma_1,\ldots,\hat \sigma_n\rangle\xrightarrow{J_\varphi} \Z F_2\langle \hat f_1,\ldots,\hat f_m\rangle \ra  0 \ra 0).\tag{$\dagger$}
\end{equation}
The square matrix $J_\varphi$ is called the \emph{Fox Jacobian} of $\varphi$ with respect to the basis $\langle x_1,\ldots,x_n\rangle $ and $\langle y_1,\ldots,y_m\rangle $.
Recall that the \emph{Fox derivative} $\frac\partial{\partial y_i}\colon \Z F_2\ra \Z F_2$, $i=1,\ldots,m$ are $\Z$-linear maps characterized by the following two properties:
\begin{itemize}
    \item $\frac\partial{\partial y_i}1=0$ and $\frac\partial{\partial y_i}y_j=\delta_{ij}$.
    \item $\frac\partial{\partial y_i}(uv)=\frac\partial{\partial y_i}u+u\cdot \frac\partial{\partial y_i}v$ for all $u,v\in F_2$.
\end{itemize}
 The entries of $J_\varphi$ are then given by the Fox derivatives
\[
    (J_\varphi)_{ij}=\frac{\partial \varphi(x_i)}{\partial y_j}\in\Z F_2,\quad 1\leqslant i\leqslant n,\ 1\leqslant j\leqslant m.
\]

\begin{theorem}\label{Theorem Properties of universal torsion for free group homomorphism}
Let $\varphi\colon F_1\ra F_2$ be a homomorphism between finitely generated free groups. Then:
\begin{enumerate}[\rm\quad (1)]
    \item $\tautwo_u(\varphi)=\whdet(J_\varphi)$ where $J_\varphi$ is the Fox Jacobian of $\varphi$ with respect to any choice of basis of 
    $F_1$ and $F_2$. In particular $\varphi$ is an $L^2$-homology equivalence if and only if $J_\varphi$ is a weak isomorphism.
    \item $\tautwo_u(\varphi)=1$ if $\varphi$ is an isomorphism; $\tautwo_u(\varphi)=0$ if $m\not=n$ or $\varphi$ is not injective.
    \item Suppose $\psi\colon F_2\ra F_3$ is an injective homomorphism to a finitely generated free group $F_3$, and either  $\varphi$ or $\psi$ is an $L^2$-homology equivalence. Then $$\tautwo_u(\psi\circ \varphi)=\psi_*\tautwo_u(\varphi)\cdot \tautwo_u(\psi).$$
\end{enumerate}
    
\end{theorem}
\begin{proof}
    Firstly, (1) is immediate from Definition \ref{Definition of universal l2 torsion for free group endomorphism} and Equation (\ref{Equation chain complex of mapping cylinder}).

    For (2), if $\varphi$ is an isomorphism, then $\Phi\colon X_1\ra X_2$ is a homotopy equivalence. Since the Whitehead group of a free group is trivial, $\Phi$ is a simple homotopy equivalence and $\tautwo_u(\varphi)=\tautwo_u(\Phi)=1$ by Proposition \ref{Proposition Properties of torsion of mappings}(3). If $m\not=n$, then $J_\varphi$ is not a square matrix and clearly $\tautwo_u(\varphi)=0$. The proof of (2) is finished once we establish the following Lemma \ref{Lemma weak isomorphism implies injective}. 

    \begin{lemma}\label{Lemma weak isomorphism implies injective}
    If $n=m$ and $\varphi$ is not injective, then the Fox Jacobian $J_\varphi$ is not a weak isomorphism.
\end{lemma}
\begin{proof}[{Proof of Lemma \ref{Lemma weak isomorphism implies injective}}]
    Since $\varphi$ is not injective, there exists a reduced word $w\in F_1$ such that $\varphi(w)=1$. Let $w=x_{i_1}^{\epsilon_1}\cdots x_{i_k}^{\epsilon_k}$ be a such word with shortest length $k\geqslant1$, where $i_1,\ldots,i_k\in\{1,\ldots,n\}$ and $\epsilon_1,\ldots,\epsilon _k\in\{\pm1\}$. We may assume that $x_{i_1}^{\epsilon_1}=x_1$ and $x_{i_k}^{\epsilon_k}\not=x_1\inv$.
    Denote by $w_s:=x_{i_1}^{\epsilon_1}\cdots x_{i_s}^{\epsilon_s}$, $s=1,\ldots,k$ to be the prefix of $w$ of length $s$ and set $w_0=1$. For any $j=1,\ldots,n$, apply $\frac\partial{\partial y_j}$ to both sides of the identity $\varphi(x_{i_1})^{\epsilon_1}\cdots \varphi(x_{i_k})^{\epsilon_k}=1,$ we have
    \[
        \sum_{s=1}^k u_s\cdot \frac{\partial \varphi(x_{i_s})}{\partial y_j}=0,\quad \text{where\quad}u_s=\begin{cases}
            \varphi(w_{s-1}), & \epsilon_s=1,\\
            -\varphi(w_s),& \epsilon_s=-1.
        \end{cases}
    \]
    Note that $u_s$ is independent of $j$. Rearranging the identities, we have 
    \[
        \sum_{i=1}^nU_i\cdot \frac{\partial \varphi(x_{i})}{\partial y_j}=0,\quad j=1,\ldots,n
    \]
    where $U_i$ is the sum of all $u_s$ such that $i_s=i$. Therefore
    \[
        (U_1,U_2,\ldots,U_n)\cdot J_\varphi =0.
    \]    
    We prove that $U_1\not=0$ and this would imply that $J_\varphi$ is not a weak isomorphism. Let $1\leqslant s_1<s_2<\cdots<s_r\leqslant k$ be the indices such that $i_{s_1}=\cdots=i_{s_r}=1$. By assumption $s_1=1$. Then 
    \begin{align*}
        U_1&=1+u_{s_2}+\cdots+u_{s_r}
        \\&=1\pm \varphi(w_{s_2'})\pm\cdots\pm \varphi(w_{s_r'})
    \end{align*}where each $w_{s_j'}$ equals either $w_{s_j}$ or $w_{s_{j}-1}$, depending on the sign of $\epsilon_{s_j}$. Since $w$ is reduced and does not end in $x_1\inv$, we have $0<s_2'<\cdots <s_r'<k$. We claim that the elements $1,\varphi(w_{s_{2}'}),\ldots, \varphi(w_{s_{r}'})$ are pairwise distinct. Indeed, if $\varphi(w_{s_{p}'})= \varphi(w_{s_{q}'})$ for some $p<q$ then the reduced word $w_{s_{p}'}\inv w_{s_{q}'}$ would be a non-trivial element of $\ker\varphi$ of length shorter than $k$, contradicting the minimality of $k$. Therefore $U_1$ is a nonzero element of $\Z F_2$, completing the proof.
\end{proof}
    
    Finally, Theorem \ref{Theorem Properties of universal torsion for free group homomorphism}(3) follows from Proposition \ref{Proposition Properties of torsion of mappings}(4).
\end{proof}

In particular, Theorem \ref{Theorem Properties of universal torsion for free group homomorphism} implies Theorem \ref{Main Theorem formula for universal L2 torsion of free groups}.

It is natural to ask what group theoretic information are reflected in the universal $L^2$-torsion of a group homomorphism. A finitely generated subgroup $H$ of a free group $F$ is called \emph{compressed} if for any subgroup $L$ of $F$ containing $H$ we have $\rank H\leqslant \rank L$. Suppose $\varphi\colon F_1\ra F_2$ is injective with $\rank F_1=\rank F_2$. Jaikin-Zapirain \cite{jaikin2024free} proved that $\varphi$ is an $L^2$-homology equivalence if and only if $\im \varphi\subset F_2$ is a compressed subgroup.

We propose the following conjecture:

\begin{conjecture}\label{Conjecture free goup isomorphism}
    A homomorphism $\varphi$ between finitely generated free groups is an isomorphism if and only if $\tautwo_u(\varphi)=1$.
\end{conjecture}

\subsection{$3$-dimensional handlebodies}
In the remaining part of this paper, a \textit{handlebody} refers to a compact connected orientable $3$-manifold obtained from attaching 1-handles to a 3-ball. The boundary of a handlebody is a connected closed orientable surface, whose genus determines the homeomorphism type of the handlebody. A \emph{genus-$g$ handlebody} $H_g$ refers to a handlebody whose boundary is a surface of genus $g\geqslant 0$. Note that $H_g$ deformation retracts to the one-point union of $g$ circles. 

Fix $g\geqslant1$. Consider the sutured manifold $(H_g,R_+,R_-,\gamma)$ such that $R_+$ and $R_-$ are both connected and $\chi(R_+)=\chi(R_-)$. It is clear that $\chi(R_+)=\chi(R_-)=1-g$, and the fundamental group of $R_\pm$ are both isomorphic to a free group of rank $g$. 
\begin{proposition}\label{Proposition universal L2torsion of handlebody}
    Suppose $(H_g,R_+,R_-,\gamma)$ is a sutured manifold such that $R_+,R_-$ are both connected and $\chi(R_+)=\chi(R_-)$ (we do not assume that $R_\pm$ are incompressible). Then:
    \begin{enumerate}[\rm\quad (1)]
    \item $\tautwo_u(H_g,R_+)=\tautwo_u(\varphi)$ where $\varphi\colon \pi_1(R_+)\ra \pi_1(H_g)$ is the homomorphism induced by the inclusion. 
        \item $(H_g,R_+,R_-,\gamma)$ is a taut sutured manifold if and only if $\tautwo_u(H_g,R_+)\not=0$.
        \item $(H_g,R_+,R_-,\gamma)$ is a product sutured manifold if and only if $\tautwo_u(H_g,R_+)=1$.
    \end{enumerate}
    In particular, Conjecture \ref{Conjecture free goup isomorphism} holds true for any homomorphism which can be realized as $\varphi\colon \pi_1(R_+)\ra \pi_1(H_g)$.
\end{proposition}
\begin{proof}
 For (1), note that $R_+$ and $H_g$ are classifying spaces for $\pi_1(R_+)$ and $\pi_1(H_g)$, respectively. Therefore the inclusion map $\Phi\colon R_+\ra H_g$ is a realization of $\varphi$, hence
 \[
    \tautwo_u(H_g,R_+)=\tautwo_u(\Phi)=\tautwo_u(\varphi).
 \]
 
    For (2), the forward direction follows from Theorem \ref{Theorem Taut iff L2acyclic}. Now suppose $\tautwo_u(H_g,R_+)$ is nonzero. By Proposition \ref{Proposition Rpm dual formula} we know that $\tautwo_u(H_g,R_-)$ is also nonzero. Then it follows from (1) and Theorem \ref{Theorem Properties of universal torsion for free group homomorphism}(2) that $R_\pm\subset H_g$ are both incompressible surfaces. Applying Theorem \ref{Theorem Taut iff L2acyclic} again implies that $(H_g,\gamma)$ is taut. 

    Finally, (3) is a direct corollary of (2) and Theorem \ref{Main Theorem universal L2 torsion detect product sutured manifold}.
\end{proof}

\begin{figure}[h]
        \centering
        
\def\svgwidth{.7\columnwidth}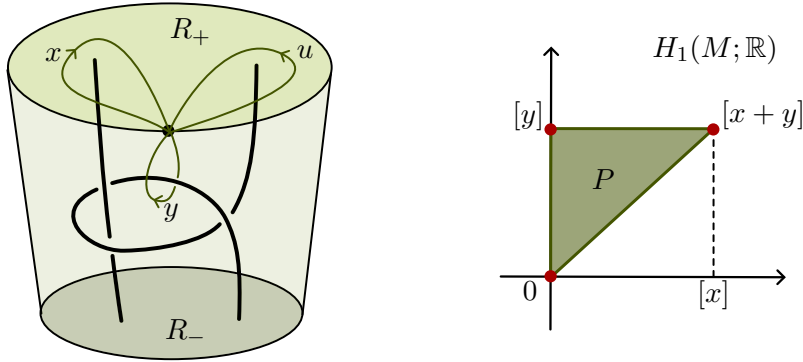

        \caption{A sutured manifold $M$ and a representative $P$ of its $L^2$-polytope in $H_1(M;\R)$}
        \label{fig:genus-2 handlebody}
    \end{figure}
\begin{example}
    Consider a sutured manifold $M$ as depicted on the left of Figure \ref{fig:genus-2 handlebody}. It is a 3-ball with two arcs removed, and three sutures separate $\partial M$ into two pairs of pants $R_+$ and $R_-$. The manifold $M$ is homeomorphic to a genus-2 handlebody whose fundamental group is generated by the loops $x$ and $y$. The fundamental group of $R_+$ is generated by the loops $x$ and $u$ where $u=yxyx\inv y\inv$ in $\pi_1(M)$. Let $\varphi\colon \pi_1(R_+)\ra \pi_1(M)$ be the homomorphism induced by inclusion, then under the basis $\pi_1(R_+)=\langle x,u\rangle$, $\pi_1(M)=\langle x,y\rangle$, we have
    \[
        J_\varphi=\begin{pmatrix}
            \frac {\partial x}{\partial x} &\frac {\partial x}{\partial y}\\
            \frac {\partial u}{\partial x} & \frac {\partial u}{\partial y}
        \end{pmatrix}=\begin{pmatrix}
            1 & 0\\
            y-yxyx\inv & 1+yx-u
        \end{pmatrix}.
    \]
    Therefore by Proposition \ref{Proposition universal L2torsion of handlebody}, 
    \[
        \tautwo_u(M,R_+)=\whdet(J_\varphi)=[1+yx-u].
    \]
    The polytope map $\mathbb P\colon \operatorname{Wh}(\mcalD{\pi_1(M)})\ra \mathcal \whpoly (H_1(M;\Z))$ sends $\tautwo_u(M,R_+)$ to a polytope $[P]$, where $P$ is the convex hull of $\{0,[u]=[y],[x+y]\}$. Proposition \ref{Proposition universal L2torsion of handlebody} implies that $M$ is a taut sutured manifold but not a product sutured manifold.
\end{example}

\begin{figure}[htbp]
    \centering
    
\def\svgwidth{0.85\columnwidth}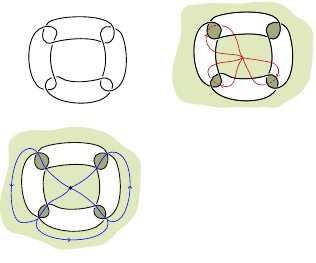

    \caption{The $n$-chain link example, where $n=4$}
    \label{fig:4-chain link}
\end{figure}

\begin{example}[$n$-chain link]\label{Example n-chain link}

For each $n\geqslant 3$, the $n$-chain link $L_n$ is an alternating link obtained by linking $n$ unknots in a cyclic way. A diagram for the case $n=4$ is shown in Figure \ref{fig:4-chain link}(a). Consider the natural Seifert surface $\Sigma$ as illustrated in Figure \ref{fig:4-chain link}(b), where the positive-side of $\Sigma$ is in light green and the negative-side of $\Sigma$ is in dark green. The surface $\Sigma$ is obtained from two disks by attaching $n$ twisted-bands, which deformation retracts to a wedge of $(n-1)$ circles.

The complement $S^3\bb \Sigma$ is a handle body of genus $(n-1)$, whose boundary is a union of two copies of $\Sigma$, namely $\Sigma_+$ and $\Sigma_-$. Choose a free basis for the fundamental group $\pi_1(S^3\bb \Sigma)=\langle x_1,\ldots,x_{n-1}\rangle$ where each $x_i$ is represented by a red loop in Figure \ref{fig:4-chain link}(b). A free basis for $\pi_1(\Sigma_+)$ is represented by the blue loops $u_1,\ldots,u_{n-1}$ in Figure \ref{fig:4-chain link}(c). By pushing $u_i$ slightly into the positive direction, we obtain its image $u_i^+$ under the inclusion $\Sigma_+\hookrightarrow S^3\bb\Sigma$, which is represented by $x_ix_{i+1}\inv$ (we assume that $x_n:=x_1\inv\cdots x_{n-1}\inv$, see Figure \ref{fig:4-chain link}(b)). Let $\varphi\colon \pi_1(\Sigma_+)\ra \pi_1(S^3\bb\Sigma)$ be the homomorphism induced by inclusion, then 
\[
    \im \varphi=\langle x_1x_2\inv,x_2x_3\inv,\ldots,x_{n-2}x_{n-1}\inv, x_{n-1}^2x_{n-2}\cdots x_2x_1\rangle \subset \langle x_1,\ldots,x_{n-1}\rangle
\]
and the Fox Jacobian of $\varphi$ is 
\begin{align*}
    J_\varphi=\begin{pmatrix}
        1 & -x_1x_2\inv & & &\\
         & 1 & -x_2x_3\inv & & \\
          &  &    \ddots & &\\
           & & & 1 & -x_{n-2}x_{n-1}\inv\\
           x_{n-1}^2x_{n-2}\cdots x_2 & x_{n-1}^2x_{n-2}\cdots x_3 & \cdots & x_{n-1}^2 & 1+x_{n-1}
    \end{pmatrix}.
\end{align*}
\begin{lemma}\label{Lemma computation of a determinant}
    For $n\geqslant 3$, let \begin{align*}
    B_n=\begin{pmatrix}
        1 & -s_1 & & &\\
         & 1 & -s_2 & & \\
          &  &    \ddots & &\\
           & & & 1 & -s_{n-2}\\
           f_1 & f_2 & \cdots & f_{n-2} & f_{n-1}
    \end{pmatrix}
\end{align*}
be a matrix over a skew field. Then its Dieudonn\'e determinant is given by $$\det(B_n)= [f_1s_1s_2\cdots s_{n-2}+f_2s_2s_3\cdots s_{n-2}+\cdots f_{n-2}s_{n-2}+f_{n-1}].$$
\end{lemma}
\begin{proof}
    For the base case $n=3$, we have $$\det(B_3)=\det\begin{pmatrix}
        1 & -s_1\\
        f_1 & f_2
    \end{pmatrix}=[f_1s_1+f_2].$$
     For general $n$, we 
        left-multiply the first row by $(-f_1)$ and add it to the last row to eliminate the bottom-left entry. Thus
        \[
            \det(B_n)=
    \det\begin{pmatrix}
        
          1 & -s_2 & & \\
            &    \ddots & &\\
            & & 1 & -s_{n-2}\\
             f_1s_1+f_2 & \cdots & f_{n-2} & f_{n-1}
    \end{pmatrix}.
        \]
    The result then follows by induction on $n$.
\end{proof}

Now apply Lemma \ref{Lemma computation of a determinant} to $J_\varphi$, note that $s_is_{i+1}\cdots s_{n-2}=x_ix_{n-1}\inv$ for $i\leqslant n-2$. It follows that 
\[
    f_is_is_{i+1}\cdots s_{n-2}=x_{n-1}^2 x_{n-2}\cdots x_ix_{n-1}\inv,\quad 1\leqslant i<n-1
\] and $f_{n-1}=1+x_{n-1}.$ Define $y_i:=x_{n-1}x_{n-2}\cdots x_i$, $i=1,\ldots,n-1$. Then $\{y_1,\ldots,y_{n-1}\}$ forms another free basis of $\pi_1(S^3\bb \Sigma)$.
By Proposition \ref{Proposition universal L2torsion of handlebody},
\begin{align*}
\tautwo_u(S^3\bb\Sigma,\Sigma_+)&=\whdet (J_\varphi)\\
    &=[x_{n-1}\cdot (y_1+y_2+\cdots+y_{n-1}+1)\cdot x_{n-1}\inv]\\&
    =[y_1+y_2+\cdots+y_{n-1}+1].
\end{align*}
The polytope map $\mathbb P$ sends the universal $L^2$-torsion $\tautwo_u(S^3\bb\Sigma,\Sigma_+)$ to an $(n-1)$-simplex in $H_1(S^3\bb \Sigma;\R)$ spanned by vertices $\{0,[y_1],\ldots,[y_{n-1}]\}$.
By Proposition \ref{Proposition universal L2torsion of handlebody} $S^3\bb\Sigma$ is a taut sutured manifold and $\Sigma$ is a norm-minimizing Seifert surface for $L$. Moreover, $\tautwo_u(S^3\bb\Sigma,\Sigma_+)\not=1$ and $L$ is not fibered as an oriented link.

\begin{remark}
    Suppose $G$ is group satisfying the Determinant Conjecture (see \cite[Section 13]{lueck2002l2} for a details), then the Fuglede--Kadison determinant defines a homomorphism $\det_{\mathcal NG}\colon \operatorname{Wh}(\mcalD{G})\ra \R_+$. For any $L^2$-acyclic finite CW-complex $X$ with $\pi_1(X)=G$, applying $\det_{\mathcal NG}$ to the universal $L^2$-torsion $\tautwo_u(X)$ yields the $L^2$-torsion $\tautwo(X)$ \cite[Section 3.4]{friedl2017universal}. Let $G$ be the fundamental group $\pi_1(S^3\bb \Sigma)$. According to \cite{ben2022fuglede},
\[
    \operatorname{det}_{\mathcal NG}([1+y_1+\cdots+y_{n-1}])=\frac{(n-1)^{\frac{n-1}{2}}}{n^{\frac{n-2}{2}}}.
\]
Therefore we obtain the $L^2$-torsion $\tautwo(S^3\bb \Sigma,\Sigma_+)=(n-1)^{\frac{n-1}{2}}/{n^{\frac{n-2}{2}}}.$

For an admissible $3$-manifold $N$ and a cohomology class $\phi\in H^1(N;\Z)$, the \emph{$L^2$-Alexander torsion} is a function $\tautwo(N,\phi)\colon \R_+\ra [0,+\infty)$ \cite{Dubois2016L2AlexanderTorsion}. This function has a well-defined ``degree" which equals to the Thurston norm of $\phi$ \cite{Friedl2019l2,liu2017degree}, and a ``leading coefficient" $C(N,\phi)\geqslant 1$ \cite{liu2017degree}.
It is proved in \cite{duan2025Guts} that $C(N,\phi)$ equals the $L^2$-torsion of the pair $(N\bb\Sigma,\Sigma_+)$ where $\Sigma$ is a norm-minimizing surface dual to $\phi$. Let $X_n$ be the $n$-chain link complement and let $\phi\in H^1(X_n;\Z)$ be the Poincar\'e dual of $\Sigma$. Then
\[
    C(X_n,\phi)=\tautwo(X_n,\Sigma_+)=\frac{(n-1)^{\frac{n-1}{2}}}{n^{\frac{n-2}{2}}}\sim \sqrt{n/e},\quad \text{as $n\ra +\infty$}.
\]
Hence, the $n$-chain link complements $X_n$ form an infinite family of hyperbolic manifolds for which $C(X_n,\phi)>1$ for some nonzero class $\phi\in H^1(X_n;\Z)\setminus\{0\}$, answering a question raised in \cite[Conjecture 1.7]{ben2022leading}.
\end{remark}

\end{example}

\bibliography{ref}
\end{document}